\documentclass[a4paper,12pt]{amsart}
\title[Stable pairs on local K3 surfaces]{{\bf Stable pairs on local K3 surfaces}}
\date{}
\author{Yukinobu Toda}

\usepackage{amscd}
\usepackage{amsmath}
\usepackage{amssymb}
\usepackage{amsthm}
\usepackage{float}
\usepackage[dvips]{graphicx}

\usepackage[all,ps,dvips]{xy}

\usepackage{array}
\usepackage{amscd}
\usepackage[all]{xy}

\DeclareFontFamily{U}{rsfs}{%
\skewchar\font127}
\DeclareFontShape{U}{rsfs}{m}{n}{%
<-6>rsfs5<6-8.5>rsfs7<8.5->rsfs10}{}
\DeclareSymbolFont{rsfs}{U}{rsfs}{m}{n}
\DeclareSymbolFontAlphabet
{\mathrsfs}{rsfs}
\DeclareRobustCommand*\rsfs{%
\@fontswitch\relax\mathrsfs}

\theoremstyle{plain}
\newtheorem{thm}{Theorem}[section]
\newtheorem{prop}[thm]{Proposition}
\newtheorem{lem}[thm]{Lemma}
\newtheorem{sublem}[thm]{Sublemma}
\newtheorem{defi}[thm]{Definition}
\newtheorem{rmk}[thm]{Remark}
\newtheorem{cor}[thm]{Corollary}

\newtheorem{step}{Step}
\newtheorem{sstep}{Step}
\newtheorem{ssstep}{Step}
\newtheorem{sssstep}{Step}
\newtheorem{ssssstep}{Step}

\newtheorem{prop-defi}[thm]{Proposition-Definition}
\newtheorem{thm-defi}[thm]{Theorem-Definition}
\newtheorem{lem-defi}[thm]{Lemma-Definition}

\newtheorem{conj}[thm]{Conjecture}

\newdimen\argwidth
\def\db[#1\db]{
 \setbox0=\hbox{$#1$}\argwidth=\wd0
 \setbox0=\hbox{$\left[\box0\right]$}
  \advance\argwidth by -\wd0
 \left[\kern.3\argwidth\box0 \kern.3\argwidth\right]}

\newcommand{\aA}{\mathcal{A}}
\newcommand{\bB}{\mathcal{B}}
\newcommand{\cC}{\mathcal{C}}
\newcommand{\dD}{\mathcal{D}}
\newcommand{\eE}{\mathcal{E}}
\newcommand{\fF}{\mathcal{F}}

\newcommand{\hH}{\mathcal{H}}

\newcommand{\mM}{\mathcal{M}}

\newcommand{\oO}{\mathcal{O}}
\newcommand{\pP}{\mathcal{P}}

\newcommand{\rR}{\mathcal{R}}
\newcommand{\sS}{\mathcal{S}}
\newcommand{\tT}{\mathcal{T}}
\newcommand{\uU}{\mathcal{U}}

\newcommand{\xX}{\mathcal{X}}
\newcommand{\yY}{\mathcal{Y}}
\newcommand{\zZ}{\mathcal{Z}}

\newcommand{\Supp}{\mathop{\rm Supp}\nolimits}
\newcommand{\Hom}{\mathop{\rm Hom}\nolimits}

\newcommand{\dotimes}{\stackrel{\textbf{L}}{\otimes}}
\newcommand{\dR}{\mathbf{R}}

\newcommand{\Hilb}{\mathop{\rm Hilb}\nolimits}

\newcommand{\Pic}{\mathop{\rm Pic}\nolimits}

\newcommand{\id}{\textrm{id}}

\newcommand{\ch}{\mathop{\rm ch}\nolimits}

\newcommand{\td}{\mathop{\rm td}\nolimits}
\newcommand{\Ext}{\mathop{\rm Ext}\nolimits}
\newcommand{\Spec}{\mathop{\rm Spec}\nolimits}
\newcommand{\rank}{\mathop{\rm rank}\nolimits}
\newcommand{\Coh}{\mathop{\rm Coh}\nolimits}

\newcommand{\cneq}{\mathrel{\raise.095ex\hbox{:}\mkern-4.2mu=}}
\newcommand{\eqcn}{\mathrel{=\mkern-4.5mu\raise.095ex\hbox{:}}}

\newcommand{\Cok}{\mathop{\rm Cok}\nolimits}

\newcommand{\Aut}{\mathop{\rm Aut}\nolimits}

\newcommand{\Stab}{\mathop{\rm Stab}\nolimits}

\newcommand{\DT}{\mathop{\rm DT}\nolimits}
\newcommand{\PT}{\mathop{\rm PT}\nolimits}

\newcommand{\Imm}{\mathop{\rm Im}\nolimits}

\newcommand{\Ker}{\mathop{\rm Ker}\nolimits}
\newcommand{\Ree}{\mathop{\rm Re}\nolimits}

\newcommand{\tr}{\mathop{\rm tr}\nolimits}
\newcommand{\ex}{\mathop{\rm ex}\nolimits}

\newcommand{\Auteq}{\mathop{\rm Auteq}\nolimits}

\newcommand{\cl}{\mathop{\rm cl}\nolimits}

\begin{document}
\maketitle
\begin{abstract}
We prove a formula which 
relates Euler characteristic of moduli 
spaces of stable pairs on local K3 surfaces
to counting invariants of semistable sheaves on 
them. 
Our formula generalizes Kawai-Yoshioka's
formula for stable pairs
with irreducible curve classes to 
arbitrary curve classes. 
We also propose a 
conjectural multiple cover formula 
of sheaf counting invariants 
which, combined with our 
main result, 
leads to an Euler characteristic version 
of Katz-Klemm-Vafa conjecture for 
stable pairs. 
\end{abstract}

\section{Introduction}
Let $S$ be a smooth projective K3 surface 
over $\mathbb{C}$, and $X$ 
the total space of the canonical line bundle 
(i.e. trivial line bundle)
on $S$, 
\begin{align*}
X =S \times \mathbb{C}. 
\end{align*}
The space $X$ is a non-compact 
Calabi-Yau 3-fold. 
We first state our main result, 
then discuss its motivation, 
background and outline 
of the proof. 
\subsection{Main result}
Our goal is to 
prove a formula which relates 
the following two kinds of invariants on $X$.

{\bf (i)
Stable pair invariants: } 
The notion of stable pairs is introduced by 
Pandharipande-Thomas~\cite{PT}
in order to give a refined 
Donaldson-Thomas curve counting 
invariants on Calabi-Yau 3-folds. 
By definition, 
a stable pair on $X$ consists of a pair 
\begin{align*}
(F, s), \quad s\colon \oO_X \to F, 
\end{align*}
where $F$ is a pure one dimensional 
coherent sheaf on $X$
and $s$ is surjective in dimension one. 
We always assume that $F$ is supported on the fibers of 
the second projection, 
\begin{align}\label{second:proj}
X=S\times \mathbb{C} \to \mathbb{C}.
\end{align}
For $\beta \in H_2(X, \mathbb{Z})$ 
and $n\in \mathbb{Z}$, 
the moduli space of such pairs $(F, s)$
satisfying 
\begin{align*}
[F]=\beta, \quad \chi(F)=n,
\end{align*}
is denoted by 
$P_{n}(X, \beta)$.
We are interested in its 
topological Euler characteristic, 
\begin{align}\label{inv:chi}
\chi(P_n(X, \beta)) \in \mathbb{Z}. 
\end{align}

{\bf (ii) Sheaf counting invariants: }
Let $\omega$ be an ample divisor on $S$, and 
take a vector
\begin{align}\label{vec:sheaf}
v=(r, \beta, n) \in \mathbb{Z} \oplus H^2(S, \mathbb{Z}) \oplus 
\mathbb{Z}.
\end{align}
The moduli stack of $\omega$-Gieseker semistable 
sheaves on $X$ supported on the fibers of 
the projection (\ref{second:proj})
 with Mukai vector
$v$ is denoted by 
\begin{align}\label{stack}
\mM_{\omega}(r, \beta, n). 
\end{align}
We consider its `Euler characteristic', 
\begin{align}\label{inv:J}
J(r, \beta, n)=`\chi'(\mM_{\omega}(r, \beta, n)). 
\end{align}
We will see that the RHS does not depend on $\omega$, 
so $\omega$ is not included in the LHS. 
When the vector $v$ is primitive, 
then the invariant
 (\ref{inv:J}) is the usual Euler characteristic 
of the moduli space (\ref{stack}). 
However when $v$ is not primitive, 
then the stack
(\ref{stack}) may have complicated stabilizers 
and the definition of its
`Euler characteristic'
is not obvious. 
In this case, 
we apply Joyce's theory on 
counting invariants, developed in~\cite{Joy1}, \cite{Joy2}, 
\cite{Joy3}, \cite{Joy4}, \cite{Joy5}
to the definition of $`\chi'$. 
Namely  
the invariant (\ref{inv:J})
is defined by taking the `logarithm'
in the Hall algebra and its 
Euler characteristic. 
(See Subsection~\ref{subsec:sheafcount}.)
Similar invariants have
been also studied in~\cite{Tst3}.

 Our main result is the following theorem. 
\begin{thm}{\bf [Theorem~\ref{maintheorem:PTJ}]}
\label{thm:main}
The generating series 
\begin{align*}
\PT^{\chi}(X) \cneq 
\sum_{\beta, n}\chi(P_n(X, \beta))y^{\beta}z^n,
\end{align*}
is written as the following product expansion, 
\begin{align}\notag
\PT^{\chi}(X)=\prod_{r\ge 0, \beta>0, n\ge 0} &
\exp\left((n+2r)J(r, \beta, r+n)y^{\beta}z^n \right) \\
\label{main:formula}
&\cdot \prod_{r> 0, \beta>0, n>0}
\exp\left((n+2r)J(r, \beta, r+n) y^{\beta} z^{-n} \right).
\end{align}
\end{thm}
Here we have 
regarded 
$\beta \in H_2(X, \mathbb{Z})$ as an 
element of $H^2(S, \mathbb{Z})$,
and $\beta>0$ means that $\beta$ is 
a 
Poincar\'e dual of 
an effective one cycle on $S$. 
The background of the above formula will be discussed below.

\subsection{Motivation and Background}
The curve counting theories on 
Calabi-Yau 3-folds have 
drawn much attention recently. 
Now there are three kinds of such
theories: Gromov-Witten (GW) theory~\cite{BGW}, \cite{LTV}, 
Donaldson-Thomas (DT) 
theory~\cite{Thom} and Pandharipande-Thomas (PT) 
theory~\cite{PT}. These theories are 
conjectured to be equivalent by
Maulik-Nekrasov-Okounkov-Pandharipande~\cite{MNOP}
and Pandharipande-Thomas~\cite{PT}. 
Among the above three theories, DT and PT
theories also count objects 
in the derived category of coherent sheaves
on $X$. 
Based on this observation, 
it was speculated in~\cite{PT} that DT/PT
theories should be related by wall-crossing 
phenomena w.r.t. Bridgeland's 
space of stability conditions on 
the derived category
of coherent sheaves~\cite{Brs1}. 
In recent years, 
general theories of 
 wall-crossing formula 
of DT type invariants are
established by Joyce-Song~\cite{JS}
and Kontsevich-Soibelman~\cite{K-S}.
By applying the wall-crossing formula, 
several geometric applications have been obtained,
e.g. DT/PT 
correspondence, rationality of the generating series, 
flop invariance, etc. 
(cf.~\cite{Tcurve1}, \cite{Tcurve2}, \cite{Tolim2}, \cite{StTh}, \cite{BrH}.)
Our purpose is to 
give a further application 
of the wall-crossing 
in the derived category to 
the curve counting invariants on 
local K3 surfaces.  

When $X=S \times \mathbb{C}$ for 
a K3 surface $S$, usual
curve counting invariants 
are rather trivial. 
This is because that, although 
the curve counting invariants are unchanged 
under deformations of $S$, the K3 surface $S$
can be deformed in a non-algebraic way 
so that the resulting invariants are always zero. 
Instead, the reduced curve counting invariants 
should be the correct mathematical objects to be
studied. 
(See Subsections~\ref{subsec:KKVconj}, \ref{subsec:reduced2}.)
These reduced theories are introduced 
and studied in~\cite{MP1}, \cite{KMPS}, \cite{MPT}, 
and unchanged under deformations 
of $S$ preserving the curve class 
to be algebraic. 

One of the goals in the study of curve counting 
invariants on $X$ is to prove a conjecture 
by Katz-Klemm-Vafa~\cite[Section~6]{KKV}, 
which we call \textit{KKV conjecture}.  
It predicts a certain evaluation of 
reduced curve counting 
invariants on $X$ in terms of modular forms, 
and is derived from the duality between 
the M-theory on $S$ and 
the heterotic string theory on $T^3$.
Mathematically the KKV conjecture 
 is formulated in terms of 
generating series of reduced GW invariants, 
(cf.~Conjecture~\ref{conj:KKV},)
and proved for primitive\footnote{A curve class 
$\beta \in H_2(X, \mathbb{Z})$ is \textit{primitive} 
if it is not a multiple of some 
other element in $H_2(X, \mathbb{Z})$.} curve classes by 
Maulik-Pandharipande-Thomas~\cite{MPT}. 
The strategy
in the paper~\cite{MPT} 
for primitive classes is as follows: 
\begin{itemize}
\item First prove a reduced version of
 GW/PT correspondence for primitive classes.
Then the problem can be reduced to a computation of 
reduced PT invariants for primitive  
classes.
The latter computation  
can be reduced to the
case of irreducible\footnote{A curve class 
$\beta \in H_2(X, \mathbb{Z})$ is \textit{irreducible} 
if it is not written as $\beta_1 +\beta_2$
for $\beta_i>0$.}
 curve classes
by a deformation argument. 
\item If the curve class is irreducible, then the reduced
PT invariant coincides with the 
Euler characteristic
of the moduli space (\ref{inv:chi}). 
The invariants (\ref{inv:chi})
for irreducible $\beta$ are completely 
calculated by Kawai-Yoshioka~\cite{KY}, and 
apply their formula. 
(cf.~Theorem~\ref{KKV:KY}.)
\end{itemize}
Now suppose that we try to solve KKV 
conjecture for arbitrary curve classes, 
following the above strategy. 
Then it is natural to try to  
generalize Kawai-Yoshioka's 
formula~\cite{KY} for the invariants (\ref{inv:chi}) 
with irreducible curve classes to
arbitrary curve classes. 
Our main theorem
(Theorem~\ref{thm:main}) has grown out of such an 
attempt. 
In fact we will see that the formula (\ref{main:formula}) 
reconstructs Kawai-Yoshioka's formula~\cite{KY}.
(See Subsection~\ref{subsec:reduced2}.)
Also the formula (\ref{main:formula}) 
reduces the computation of the stable pair
invariants (\ref{inv:chi})
for arbitrary $\beta$
to that of the sheaf counting invariants (\ref{inv:J}). 
As we will discuss in the next subsection, 
the latter invariants are 
expected to be related to the Euler characteristic 
of the Hilbert scheme of points on $S$ in terms of
the multiple cover formula. 
(cf.~Conjecture~\ref{conj:intro}.)
Assuming such a multiple cover formula, 
the series $\PT^{\chi}(X)$ is written as
an infinite product similar to 
Borcherd's product~\cite{Borch}, 
giving a complete calculation of the invariants (\ref{inv:chi})
for arbitrary curve classes. 
(See Subsection~\ref{subsec:multi} 
and the formula (\ref{Borcherd}) below.)

At this moment, we don't know any relationship between 
reduced PT invariants and the invariants (\ref{inv:chi})
when the curve class $\beta$ is not irreducible. 
So a reduced version of GW/PT together with
Theorem~\ref{thm:main}
do not immediately imply the KKV conjecture. 
Also there is a technical gap in 
proving a version of the formula (\ref{main:formula})
 for reduced PT invariants.
The issue
is that, although the wall-crossing formula 
in~\cite{JS}, \cite{K-S}
is established
for invariants expressed by the Behrend function~\cite{Beh}, 
it is not clear whether  
reduced PT invariants are
expressed in that way or not. 
Nevertheless, we expect that there is a close relationship 
between reduced PT invariants and the 
Euler characteristic invariants (\ref{inv:chi}), 
and knowing the invariants (\ref{inv:chi}) 
give a geometric intuition of the reduced PT invariants. 
Such an attempt, 
namely 
proving an Euler characteristic version
first and then back to the virtual one, 
 is also employed 
and successful in proving 
DT/PT correspondence and the rationality 
conjecture in~\cite{Tcurve1}, \cite{Tolim2}, \cite{StTh}, \cite{BrH}. 
In fact, we will see that the reduced GW/PT correspondence, 
 together with a
conjectural version of the formula (\ref{main:formula}) 
for reduced PT invariants,
yield KKV conjecture. 
(See Subsection~\ref{spec:KKV}.)
In this sense, our result is expected to be 
a first step toward
a complete proof of KKV conjecture.

\subsection{Conjectural multiple cover formula}\label{subsec:multi}
An interesting point of the formula (\ref{main:formula})
is that it gives a relationship between 
invariants with different features. Namely, 
\begin{itemize}
\item A stable pair invariant (\ref{inv:chi})
is easy to define and an integer. 
However it is not easy to relate it to the 
geometry of K3 surfaces, nor see any 
interesting dualities among the invariants (\ref{inv:chi}).  
\item A sheaf counting invariant (\ref{inv:J})
is difficult to define and it is not necessary an 
integer. However it has a nice automorphic property, 
and seems to be related to the Euler characteristic 
of the Hilbert scheme of points on $S$. 
\end{itemize}
Let us focus on the property of the invariants (\ref{inv:J}). 
First the invariant $J(v)$
is completely calculated when 
$v$ is a primitive algebraic class. 
In this case, we have~\cite{KY}, \cite{Yoshi2},
\begin{align}\label{J=Hilb}
J(v)=\chi(\Hilb^{(v, v)/2 +1}(S)).
\end{align}
Here $(\ast, \ast)$ is 
the Mukai pairing on the Mukai lattice, 
\begin{align*}
\widetilde{H}(S, \mathbb{Z}) \cneq 
\mathbb{Z} \oplus H^2(S, \mathbb{Z}) \oplus \mathbb{Z},
\end{align*} 
and $\Hilb^n(S)$ is the Hilbert scheme of 
$n$-points in $S$.
Its Euler characteristic is computed by the 
G$\ddot{\rm{o}}$ttsche's
formula~\cite{Got}, 
\begin{align}\label{Goett}
\sum_{n \ge 0}
\chi(\Hilb^n(S))q^n 
=\prod_{n \ge 1}\frac{1}{(1-q^n)^{24}}.
\end{align} 
The formulas (\ref{main:formula}), (\ref{J=Hilb})
and (\ref{Goett}) reconstruct 
Kawai-Yoshioka's
formula~\cite{KY} for 
the invariants (\ref{inv:chi})
with irreducible $\beta$.  
(See Subsection~\ref{subsec:reduced2}.)
 
When $v$ is not necessarily 
primitive, 
we are not able to
give a complete computation of 
the invariant (\ref{inv:J})
at this moment. 
However there is an advantage 
of the invariant (\ref{inv:J}), as 
it has a certain automorphic property.
Recall that 
the lattice $\widetilde{H}(S, \mathbb{Z})$
has a weight two Hodge structure. 
(See Subsection~\ref{subsec:K3}.)
The following result is a refinement of~\cite[Corollary~5.26]{Tst3}. 
\begin{thm}\label{intro:aut}
{\bf [Theorem~\ref{thm:aut}] }
Let $g$ be a Hodge isometry 
of the lattice $\widetilde{H}(S, \mathbb{Z})$. Then we have 
\begin{align*}
J(gv)=J(v). 
\end{align*}
\end{thm}
The above result means that, if
we are able to compute the invariant (\ref{inv:J})
for specific $(r, \beta, n)$, e.g. $r=0$, then 
we can also compute 
the invariant (\ref{inv:J}) for another 
$(r, \beta, n)$ by applying 
a Hodge isometry $g$. 
On the other hand, the invariant of the form
$J(0, \beta, n)$ is expected to 
satisfy a certain multiple cover formula
as discussed in~\cite[Conjecture~6.20]{JS}, \cite[Theorem~6.4]{Tsurvey}. 
Combining these arguments, we 
propose the following conjecture.  
\begin{conj}{\bf [Conjecture~\ref{conj:mult3}] }
\label{conj:intro}
If $v$ is an algebraic class, then 
$J(v)$ is written as 
\begin{align*}
J(v)=\sum_{k\ge 1, k|v}
\frac{1}{k^2}\chi(\Hilb^{(v/k, v/k)/2 +1}(S)). 
\end{align*} 
\end{conj}
We will give some evidence of the above conjecture 
in Subsection~\ref{subsec:mult3}. 
If we assume the above conjecture, then the formula 
(\ref{main:formula})
is written as 
\begin{align}\notag
\PT^{\chi}(X)
=& \prod_{r\ge 0, \beta>0, n\ge 0}(1-y^{\beta}z^n)^{-(n+2r)\chi(\Hilb^{\beta^2/2 -r(n+r)+1}(S))} \\
\label{Borcherd}
&\cdot \prod_{r> 0, \beta>0, n> 0}(1-y^{\beta}z^{-n})^{-(n+2r)\chi(\Hilb^{\beta^2/2 -r(n+r)+1}(S))}.
\end{align}
The above formula, 
which resembles Borcherd's product~\cite{Borch}, 
 is interpreted as an Euler characteristic 
version of KKV conjecture for stable pairs. 
As we will discuss in Subsection~\ref{spec:KKV}, 
a similar formula for reduced PT invariants 
together with reduced GW/PT correspondence 
give a complete proof of KKV conjecture.

\subsection{Outline of the proof of Theorem~\ref{thm:main}}
\begin{sssstep}
\end{sssstep}
We compactify $X$ as 
\begin{align*}
\overline{X}=S\times \mathbb{P}^1, 
\end{align*}
and prove a formula for $\PT^{\chi}(\overline{X})$. 
(See Subsection~\ref{subsec:stablepair}.)
The series 
$\PT^{\chi}(\overline{X})$ is related to $\PT^{\chi}(X)$ by, 
(cf.~Lemma~\ref{lem:eqPT},) 
\begin{align}\label{intro:PT2}
\PT^{\chi}(\overline{X})=\PT^{\chi}(X)^2. 
\end{align}
We also introduce the invariant,
(cf.~Definition~\ref{defi:Ninv},) 
\begin{align}\label{intro:N}
N(r, \beta, n) \in \mathbb{Q}, 
\end{align}
counting certain semistable objects 
supported on 
the fibers of the projection, 
\begin{align*}
\pi \colon \overline{X}=S \times \mathbb{P}^1 \to \mathbb{P}^1,
\end{align*}
with Chern character (\textit{not} Mukai vector) 
equal to $(r, \beta, n)$.  
The invariant (\ref{intro:N})
is related to the invariant 
(\ref{inv:J}) by, (cf.~Proposition~\ref{prop:inv},) 
\begin{align}\label{intro:N=J}
N(r, \beta, n)=2J(r, \beta, r+n).
\end{align} 
\begin{sssstep}
\end{sssstep}
In~\cite[Theorem~1.3]{Tolim2}, 
the author proved the following formula
by using Joyce's wall-crossing formula~\cite{Joy4}, 
\begin{align}\label{intro:pexpand}
\PT^{\chi}(\overline{X})=\prod_{\beta>0, n>0}
\exp\left(nN(0, \beta, n)y^{\beta} z^{n}  \right)
\left(\sum_{\beta, n}L(\beta, n)y^{\beta} z^n  \right). 
\end{align}
The invariant
$L(\beta, n)$
 counts certain objects in the derived category
$D^b \Coh(\overline{X})$, 
which are $\mu_{i\omega}$-limit 
semistable objects in the notation of~\cite[Section~3]{Tolim2}. 
(cf.~Definition~\ref{def:lim}.)
Our idea is to decompose the 
series of $L(\beta, n)$ 
further, 
using the wall-crossing formula again. 
More precisely, we 
study the triangulated category, 
\begin{align*}
\dD \cneq \langle \pi^{\ast}\Pic(\mathbb{P}^1), \Coh_{\pi}(\overline{X})
 \rangle_{\tr}
\subset D^b \Coh(\overline{X}). 
\end{align*}
Here
\begin{align*}
\Coh_{\pi}(\overline{X}) \subset \Coh(\overline{X}),
\end{align*}
is the subcategory consisting of sheaves 
supported on the fibers of $\pi$. 
For a fixed ample divisor $\omega$ on $S$, 
we will construct the heart of a bounded t-structure, 
(cf.~Definition~\ref{defi:Ao},)
\begin{align*}
\aA_{\omega} \subset \dD.
\end{align*}
The above heart is unchanged if $\omega$
is replaced by $t\omega$ for $t\in \mathbb{R}_{>0}$. 
Moreover the above heart fits into a pair, 
\begin{align*}
\sigma_{t}=(Z_{t\omega}, \aA_{\omega}),
\end{align*}
for $t \in \mathbb{R}_{>0}$, 
giving a weak stability condition introduced in~\cite{Tcurve1}. 
(cf.~Lemma~\ref{lem:property}.)
We will construct the invariant, (cf.~Definition~\ref{defi:DTchi},)
\begin{align*}
\DT^{\chi}_{t\omega}(r, \beta, n) \in \mathbb{Q}
\end{align*}
as an Euler characteristic version of 
Donaldson-Thomas type invariant, counting 
$Z_{t\omega}$-semistable objects $E\in \aA_{\omega}$
satisfying 
\begin{align*}
\ch(E)&=(1, -r, -\beta, -n) \\
&\in H^0(\overline{X})
 \oplus H^2(\overline{X}) \oplus H^4(\overline{X}) \oplus H^6(\overline{X}). 
\end{align*}
Here
$r$ and $n$ are regarded
as integers, and $\beta$ is regarded as 
an element of $H^2(S, \mathbb{Z})$. 
(See Subsection~\ref{subsec:Ch}.)
\begin{sssstep}
\end{sssstep}
We investigate the generating series, 
\begin{align*}
\DT^{\chi}_{t\omega}(\overline{X})
 \cneq \sum_{r, \beta, n}\DT_{t\omega}^{\chi}(r, \beta, n)
x^r y^{\beta} z^n, 
\end{align*}
which is regarded as an element of a 
certain topological vector space. 
(See Subsection~\ref{subsec:genser}.)
If we consider big $t$, we see that,
(cf.~Proposition~\ref{prop:DT=L},) 
\begin{align}\label{lim1}
\lim_{t \to \infty}
\DT^{\chi}_{t\omega}(\overline{X})
 =\sum_{r, \beta, n}L(\beta, n)x^r y^{\beta}z^n.
\end{align}
On the other hand if we consider small $t$, we 
see that, (cf.~Proposition~\ref{prop:DT=L},)
\begin{align}\label{lim2}
\lim_{t \to 0}
\DT^{\chi}_{t\omega}(\overline{X})
=\lim_{t \to 0}
\sum_{r, \beta}\DT^{\chi}_{t\omega}(r, \beta, 0)x^r y^{\beta}. 
\end{align}
The wall-crossing formula enables us 
to see how the series 
$\DT^{\chi}_{t\omega}(\overline{X})$ varies
if we change $t$. 
Here an interesting phenomena happens: 
two dimensional semistable objects on 
the fibers of $\pi$ are involved 
in the wall-crossing formula. 
Since so far
only one dimensional 
objects have been 
involved in the 
wall-crossing formula, 
e.g. the formula (\ref{intro:pexpand}), 
 this seems a new phenomena in 
this field of study. 

By the wall-crossing formula, we obtain a formula 
relating (\ref{lim1}) and (\ref{lim2}). 
As a result, (\ref{lim1}) is 
obtained by the product of (\ref{lim2})
and the following infinite product, (cf.~Corollary~\ref{cor:DTN},) 
\begin{align}\label{intro:wcfe}
\prod_{\beta>0, rn>0} &
\exp\left( (n+2r)N(r, \beta, n) x^r y^{\beta} z^n  \right)^{\epsilon(r)},
\end{align}
where $\epsilon(r)=1$ if $r>0$ and $\epsilon(r)=-1$
if $r<0$. 
Unfortunately the above 
argument is not enough to obtain the desired formula, 
as the RHS of (\ref{lim2}) still remains unknown.  
In order to complete the proof, we focus on 
some abelian subcategory, (cf.~Definition~\ref{subsec:abelian12},)
\begin{align*}
\aA_{\omega}(1/2) \subset \aA_{\omega}.
\end{align*}
We introduce
finer weak stability conditions on $\aA_{\omega}(1/2)$,
and apply the wall-crossing formula
again. 
Then we obtain (cf.~Proposition~\ref{prop:DT=N'},)
\begin{align}\label{intro:wcf12}
\lim_{t \to 0}
\sum_{r, \beta}\DT^{\chi}_{t\omega}(r, \beta, 0)x^r y^{\beta}
=\prod_{r>0, \beta>0}
\exp\left(2rN(r, \beta, 0)x^r y^{\beta} \right) \cdot 
\sum_{r \in \mathbb{Z}}x^r. 
\end{align}
Combining
(\ref{intro:PT2}), (\ref{intro:N=J}),
(\ref{lim1}), 
(\ref{lim2}), (\ref{intro:wcfe}), (\ref{intro:wcf12}),
Theorem~\ref{intro:aut}
and looking at the $x^{0}$-term,
 we obtain
the desired formula (\ref{main:formula}). 
(cf.~Theorem~\ref{maintheorem:PTJ}.)

\subsection{Plan of the paper}
In Section~\ref{sec:tri}, we introduce 
several notation and introduce 
the abelian category $\aA_{\omega}$. 
In Section~\ref{section:wstab}, we construct weak 
stability conditions on $\aA_{\omega}$
and state some results on wall-crossing phenomena. 
In Section~\ref{sec:inv}, we introduce 
counting invariants via the Hall algebra 
of $\aA_{\omega}$, their variants, 
and investigate the property of the invariants. 
In Section~\ref{sec:WCF}, we apply wall-crossing formula 
and give a proof of Theorem~\ref{thm:main}.
In Section~\ref{sec:KKV}, we 
give some discussions toward KKV conjecture. 
From Section~\ref{sec:tech}
to Section~\ref{sec:A12},  
we give several proofs postponed in the previous 
sections.

\subsection{Acknowledgement}
The author is grateful to
Davesh Maulik, 
 Rahul Pandharipande and Richard Thomas 
for the valuable comments and advice. 
This work is supported by World Premier 
International Research Center Initiative
(WPI initiative), MEXT, Japan. This work is also supported by Grant-in Aid
for Scientific Research grant (22684002), 
and partly (S-19104002),
from the Ministry of Education, Culture,
Sports, Science and Technology, Japan.

\subsection{Notation and Convention}
For a triangulated category $\dD$, 
the shift functor is denoted by $[1]$. 
For a set of objects $\sS \subset \dD$, 
we denote by $\langle \sS \rangle_{\tr}$
the smallest triangulated subcategory 
which contains $\sS$ and $0 \in \dD$. 
Also we denote by $\langle \sS \rangle_{\ex}$
the smallest extension closed 
subcategory of $\dD$ which contains 
$\sS$ and $0 \in \dD$. 
The abelian category of coherent sheaves on 
a variety $X$ is denoted by 
$\Coh(X)$. We say 
$F \in \Coh(X)$ is 
$d$-dimensional if its support is 
$d$-dimensional, and 
we write $\dim F=d$. 
For a surface $S$, its Neron-Severi 
group is denoted by $\mathrm{NS}(S)$. 
For an element $\beta \in \mathrm{NS}(S)$, 
we write $\beta>0$ 
if $\beta$ is 
a Poincar\'e dual of 
an effective one cycle on $S$. 
An element $\beta \in \mathrm{NS}(S)$
with $\beta>0$ is 
irreducible when $\beta$ is not written 
as $\beta_1 +\beta_2$ with $\beta_i >0$. 
For a finitely generated abelian group $\Gamma$, 
an element $v\in \Gamma$ is primitive 
if $v$ is not a multiple of some other element in 
$\Gamma$.

\section{Triangulated category of local K3 surfaces}\label{sec:tri}
In this section, 
we recall some notions used in the 
study of K3 surfaces. 
We also introduce 
a certain triangulated category
associated to a K3 surface, 
and construct the heart of a bounded t-structure
on it. 
\subsection{Generalities on K3 surfaces}\label{subsec:K3}
Let $S$ be a smooth projective K3 surface over $\mathbb{C}$, i.e. 
\begin{align*}
K_S \cong \oO_S, \quad H^1(S, \oO_S)=0. 
\end{align*}
We begin with recalling the generalities on $S$. 
The \textit{Mukai lattice} is defined by 
\begin{align}\label{def:mukai}
\widetilde{H}(S, \mathbb{Z})
 \cneq H^0(S, \mathbb{Z}) \oplus H^2(S, \mathbb{Z}) \oplus H^4(S, \mathbb{Z}).
\end{align}
In what follows, 
we naturally regard $H^0(S, \mathbb{Z})$ and 
$H^4(S, \mathbb{Z})$ as $\mathbb{Z}$. 
For two elements, 
\begin{align*}
v_i=(r_i, \beta_i, n_i) \in \widetilde{H}(S, \mathbb{Z}), \quad i=1, 2, 
\end{align*}
the \textit{Mukai pairing} is defined by  
\begin{align}\label{Mpair}
(v_1, v_2) \cneq \beta_1 \beta_2 -r_1 n_2 -r_2 n_1. 
\end{align}
Recall that there is a weight 
two Hodge structure on $\widetilde{H}(S, \mathbb{Z}) \otimes \mathbb{C}$
given by, 
\begin{align*}
&\widetilde{H}^{2, 0}(S) \cneq H^{2, 0}(S), \ 
\widetilde{H}^{0, 2}(S) \cneq H^{0, 2}(S), \\
&\widetilde{H}^{1, 1}(S) \cneq H^{0, 0}(S) \oplus H^{1, 1}(S) \oplus
H^{2, 2}(S).  
\end{align*}
We define the lattice $\Gamma_0$ to be
\begin{align}\notag
\Gamma_0 &\cneq \widetilde{H}(S, \mathbb{Z}) \cap \widetilde{H}^{1, 1}(S) \\
\label{def:Gamma0}
&= \mathbb{Z} \oplus \mathrm{NS}(S) \oplus \mathbb{Z}. 
\end{align}
For an object $E\in D^b \Coh(S)$, its \textit{Mukai vector}
$v(E) \in \Gamma_0$ 
is defined by 
\begin{align}\notag
v(E) &\cneq \ch(E) \sqrt{\td_S} \\
\label{Mvec}
&=(\ch_0(E), \ch_1(E), \ch_0(E)+\ch_2(E)). 
\end{align}
For any $E, F \in D^b \Coh(S)$, 
the Riemann-Roch theorem yields, 
\begin{align}\label{asum1}
\sum_{i}(-1)^i \dim \Hom_{S}(E, F[i])
=-(v(E), v(F)). 
\end{align}
\subsection{Local K3 surfaces}\label{subsec:local}
Let $S$ be a 
K3 surface. We are interested in the 
total space of the canonical line bundle on $S$, 
\begin{align*}
X=S \times \mathbb{C}. 
\end{align*}
We compactify $X$ as 
\begin{align*}
\overline{X}=S \times \mathbb{P}^1.
\end{align*}
Our strategy is to 
deduce the geometry of $X$ from that of $\overline{X}$. 
Let $\pi$ be the second projection, 
\begin{align*}
\pi \colon \overline{X} \to \mathbb{P}^1. 
\end{align*}
We define the abelian category 
\begin{align*}
\Coh_{\pi}(\overline{X}) \subset \Coh(\overline{X}),
\end{align*}
to be the subcategory consisting of sheaves 
supported on the fibers of $\pi$. 
Its derived category is denoted by
$\dD_0$,  
\begin{align}\label{def:D0}
\dD_0 \cneq D^b \Coh_{\pi}(\overline{X}). 
\end{align}
We introduce the following triangulated category. 
\begin{defi}\label{defi:D}
We define the triangulated category $\dD$ to be 
\begin{align*}
\dD \cneq  \langle \pi^{\ast}\Pic(\mathbb{P}^1), 
\Coh_{\pi}(\overline{X}) \rangle_{\tr} 
\subset D^b \Coh(\overline{X}). 
\end{align*} 
\end{defi}
Note that $\dD_0$ is a triangulated 
subcategory of $\dD$. 
The triangulated category $\dD$ is not 
a Calabi-Yau 3 category, but close to it
by the following lemma. 
\begin{lem}\label{cy3}
Take objects $E, F \in \dD$, and suppose that either  
$E$ or $F$ is an object in $\dD_0$.
Then we have the isomorphism, 
\begin{align*}
\Hom_{\dD}(E, F) \cong \Hom_{\dD}(F, E[3])^{\vee}. 
\end{align*}
\end{lem}
\begin{proof}
The result follows from the Serre duality on $\overline{X}$ and the 
isomorphism 
\begin{align*}
E \otimes \omega_{\overline{X}} \cong E, 
\end{align*}
for $E\in \dD_0$. 
\end{proof}

\subsection{Chern characters on $\dD_0$ and $\dD$}\label{subsec:Ch}
We fix some notation on 
Chern characters for objects in $\dD_0$ and $\dD$. 
Let $p$ be the first projection, 
\begin{align*}
p \colon \overline{X} =S \times \mathbb{P}^1 \to S. 
\end{align*}
We define the group homomorphism 
$\cl_0$ to be the composition, 
\begin{align}\label{def:cl}
\cl_0 \colon K(\dD_0) 
\stackrel{p_{\ast}}{\to} K(S) \stackrel{\ch}{\to} \Gamma_0.
\end{align}
By the definition of $\Coh_{\pi}(\overline{X})$, 
the push-forward $p_{\ast}E$ is an element of 
$K(S)$, hence the above map is well-defined. 
Instead of the Chern character, we can
also consider the Mukai vector on $K(\dD_{0})$, 
\begin{align}\label{DMukai}
v\colon K(\dD_0) \stackrel{p_{\ast}}{\to}
K(S) 
\stackrel{\ch \sqrt{\td_S}}{\to}\Gamma_0,
\end{align}
as in (\ref{Mvec}).
\begin{rmk}
Although the Mukai vector is usually 
used in the study of K3 surfaces, we will 
use both of Chern characters and Mukai vectors. 
The reason is that, 
 the Chern characters 
are useful in describing
wall-crossing formula, 
while the Mukai vectors 
are useful in discussing Fourier-Mukai transforms. 
\end{rmk}

Next we consider
 the Chern character map on $K(\overline{X})$, 
\begin{align*}
\ch \colon K(\overline{X}) \to 
H^{\ast}(\overline{X}, \mathbb{Q}).
\end{align*}
If we restrict 
the above map 
to the Grothendieck group of 
$\dD$, then it factors through 
the subgroup, 
\begin{align*}
\Gamma \cneq H^0(\overline{X}, \mathbb{Z}) 
\oplus \left(\Gamma_0 \boxtimes H^2(\mathbb{P}^1, \mathbb{Z})
 \right) \subset H^{\ast}(\overline{X}, \mathbb{Q}). 
\end{align*}
Hence we obtain the group homomorphism, 
\begin{align*}
\cl \cneq \ch \colon K(\dD) \to \Gamma. 
\end{align*}
We naturally identify 
$H^0(\overline{X}, \mathbb{Z})$ and 
$H^2(\mathbb{P}^1, \mathbb{Z})$ with $\mathbb{Z}$. 
Then $\Gamma$ is identified with 
\begin{align}
\label{isom:gam}
\Gamma &= \mathbb{Z} \oplus \Gamma_0, \\
\notag
&=\mathbb{Z} \oplus \mathbb{Z} \oplus \mathrm{NS}(S) \oplus \mathbb{Z}. 
\end{align} 
We usually write an element $v\in \Gamma$ as a vector
\begin{align*}
v=(R, r, \beta, n), 
\end{align*}
where $R$, $r$, $n$ are integers and 
$\beta \in \mathrm{NS}(S)$.  
If $v=\cl(E)$ for $E \in \dD$, 
then the above vector 
corresponds to the Chern character, 
\begin{align*}
\cl(E)=(\ch_0(E), \ch_1(E), \ch_2(E), \ch_3(E)).
\end{align*}
Under the above identification, 
we always regard $\ch_0(E)$, $\ch_1(E)$, $\ch_3(E)$
as integers, and $\ch_2(E)$ as an element of $\mathrm{NS}(S)$. 
Hence for instance,  
the intersection number 
$\ch_2(E) \cdot \omega$ for 
a divisor $\omega$ on $S$ makes sense. 
Note that this is equal to $\ch_2(E) \cdot p^{\ast}\omega$
in the usual sense. 

Also in the above notation, we sometimes write
\begin{align*}
\rank(v) \cneq R, \quad \rank(E) \cneq \rank(\cl(E)),
\end{align*}
for $E\in \dD$.

By the Grothendieck Riemann-Roch theorem, 
the maps $\cl$ and $\cl_0$ are compatible. 
Namely under the identification (\ref{isom:gam}),
we have 
\begin{align}\label{compati}
\cl(E)=(0, \cl_0(E)),
\end{align}
for $E\in \dD_0 \subset \dD$.

\subsection{Classical stability conditions on $\Coh_{\pi}(\overline{X})$}
\label{subsec:classical}
We recall some classical notions of 
stability conditions on 
$\Coh_{\pi}(\overline{X})$.  
For an object $E\in \Coh_{\pi}(\overline{X})$, we write 
\begin{align*}
\cl_0(E)=(r, \beta, n) \in \mathbb{Z} \oplus \mathrm{NS}(S) \oplus \mathbb{Z}. 
\end{align*}
For an ample divisor $\omega$ on $S$, the 
slope of $E$ is defined to be 
\begin{align*}
\mu_{\omega}(E) \cneq \left\{
\begin{array}{cc}
\infty, & \mbox{ if } r=0, \\
\omega \cdot \beta/ r, & \mbox{ if } r \neq 0. 
\end{array}  \right. 
\end{align*}
Also the Hilbert polynomial of $E$ is defined by 
\begin{align*}
\chi_{\omega, E}(m) &\cneq 
\chi(\overline{X}, E \otimes \oO_{\overline{X}}(m p^{\ast}\omega)) \\
&=a_d m^d +a_{d-1}m^{d-1}+ \cdots,
\end{align*}
with $a_d \neq 0$ and $d=\dim E$. The reduced 
Hilbert polynomial is defined to be
\begin{align*}
\overline{\chi}_{\omega, E}(m) \cneq \chi_{\omega, E}(m)/a_d. 
\end{align*}
If $r\neq 0$, or equivalently $d=2$, then we have 
\begin{align*}
\overline{\chi}_{\omega, E}(m)=
m^2 +\frac{2\mu_{\omega}(E)}{\omega^2}m +(\mbox{constant term}). 
\end{align*}
Also there is a map,
\begin{align}\label{red:hilb}
\Gamma_{0} \ni v \mapsto \overline{\chi}_{\omega, v}(m)
\in \mathbb{Q}[m], 
\end{align}
such that 
we have 
$\overline{\chi}_{\omega, \cl_0(E)}(m)=\overline{\chi}_{\omega, E}(m)$
for $E\in \Coh_{\pi}(\overline{X})$.

The total order 
 $\succ$ on $\mathbb{Q}[m]$ 
is defined as follows: 
for $p_i(m) \in \mathbb{Q}[m]$ with $i=1, 2$, we have 
$p_1(m) \succ p_2(m)$ if and only if 
\begin{align*}
\deg p_1(m)&< \deg p_2(m), \mbox{ or } \\
\deg p_1(m)&=\deg p_2(m), \
p_1(m) > p_2(m), \ m\gg 0. 
\end{align*}
The above notions
determine slope stability and Gieseker-stability 
on $\Coh_{\pi}(\overline{X})$. 
\begin{defi}
(i) An object $E\in \Coh_{\pi}(\overline{X})$ is $\mu_{\omega}$-(semi)stable 
if for any exact sequence $0 \to F \to E \to G \to 0$
in $\Coh_{\pi}(\overline{X})$
with $F, G \neq 0$, we have 
\begin{align*}
\mu_{\omega}(F)<(\le)\mu_{\omega}(G). 
\end{align*}
(ii) An object $E\in \Coh_{\pi}(\overline{X})$ is $\omega$-Gieseker (semi)stable if 
 for any subsheaf $0\neq F \subsetneq E$, we have 
\begin{align*}
\overline{\chi}_{\omega, F}(m) \prec(\preceq) \overline{\chi}_{\omega, E}(m).
\end{align*} 
\end{defi}
For more detail, see~\cite{Hu}. 
It is easy to see that if an object $E\in \Coh_{\pi}(\overline{X})$ 
is $\mu_{\omega}$-(or $\omega$-Gieseker) stable, then 
$E$ is written as 
\begin{align}\label{ip}
E \cong i_{p\ast} E', \quad p\in \mathbb{P}^1, 
\end{align}
for some $\mu_{\omega}$-(or $\omega$-Gieseker)
stable sheaf $E'$ on a fiber 
\begin{align*}
X_p \cneq \pi^{-1}(p) \cong S. 
\end{align*} 
The map $i_{p}$ is the inclusion
$i_p \colon X_p \hookrightarrow \overline{X}$. 
We will use the 
following Bogomolov-type inequalities. 
\begin{lem}\label{lem:Bog}
(i)
Let $E \in \Coh_{\pi}(\overline{X})$ 
be an $\omega$-Gieseker stable
sheaf with $\cl_0(E)=(r, \beta, n)$. 
Then we have 
\begin{align}\label{ineq:Bog}
\beta^2 +2\ge 2r(r+n). 
\end{align}
(ii)
If $E \in \Coh_{\pi}(\overline{X})$ is 
$\omega$-Gieseker semistable
with $\beta \cdot \omega \neq 0$,
then we have 
\begin{align}\label{ineq:Bog0}
\beta^2 +2(\beta \cdot \omega)^2 \ge 2r(r+n). 
\end{align}
\end{lem}
\begin{proof}
If $E\in \Coh_{\pi}(\overline{X})$
 is $\omega$-Gieseker stable, 
then $E$ is written as 
(\ref{ip})
for a $\mu_{\omega}$-stable sheaf 
$E'$ on $X_p$. Then 
the inequality (\ref{ineq:Bog0}) 
is a well-known consequence of the 
Riemann-Roch theorem and the 
Serre duality.
(cf.~\cite[Corollary~2.5]{Mu2}.)
Let $E \in \Coh_{\pi}(\overline{X})$
 be an $\omega$-Gieseker semistable 
sheaf on $\overline{X}$. 
Then $p_{\ast}E$ is also an $\omega$-Gieseker
semistable sheaf on $S$. 
Let $E_1, \cdots, E_k \in \Coh(S)$ be 
 $\omega$-Gieseker stable factors
of $p_{\ast}E$.  
We have 
\begin{align}\notag
\dim \Hom_S(p_{\ast}E, p_{\ast}E) &\le \sum_{i, j}\dim \Hom_S(E_i, E_j) \\
\label{ineq:Bogsemi}
&\le k^2. 
\end{align}
Also since $\beta_i \cdot \omega$ and $\beta \cdot \omega$
have the same sign and $\beta$ is equal 
to the sum $\sum_{i}\beta_i$, we have 
$k \le \lvert \beta \cdot \omega \rvert$. 
By (\ref{asum1}), the Serre duality and (\ref{ineq:Bogsemi}), we have 
\begin{align*}
-\beta^2 +2r(r+n) \le 2k^2 
\le 2(\beta \cdot \omega)^2.
\end{align*}
\end{proof}

\subsection{The heart of a bounded
 t-structure on $\dD_0$}\label{subsec:heartD0}
For an ample divisor $\omega$ on $S$, let 
us consider $\mu_{\omega}$-stability on 
$\Coh_{\pi}(\overline{X})$. 
For each $E \in \Coh_{\pi}(\overline{X})$, there is a
Harder-Narasimhan  
filtration 
\begin{align*}
0=E_0 \subset E_1 \subset \cdots \subset E_N=E, 
\end{align*}
i.e. each subquotient $F_i=E_i/E_{i-1}$ is 
$\mu_{\omega}$-semistable with 
\begin{align*}
\mu_{\omega}(F_1)>\mu_{\omega}(F_2) > \cdots >\mu_{\omega}(F_N). 
\end{align*}
We set 
\begin{align*}
\mu_{\omega, +}(E) \cneq \mu_{\omega}(F_1), \quad \mu_{\omega, -}(E) \cneq 
\mu_{\omega}(F_N). 
\end{align*}
We define the pair of full subcategories 
$(\tT_{\omega}, \fF_{\omega})$ in $\Coh_{\pi}(\overline{X})$ to be 
\begin{align}\label{def:Tomega}
\tT_{\omega} \cneq \{ E \in \Coh_{\pi}(\overline{X}) : 
\mu_{\omega, -}(E)>0\}, \\
\label{def:Fomega}
\fF_{\omega} \cneq \{ E \in \Coh_{\pi}(\overline{X}) : 
\mu_{\omega, +}(E) \le 0\}.
\end{align}
In other words, an object $E\in \Coh_{\pi}(\overline{X})$
is contained in $\tT_{\omega}$, (resp.~$\fF_{\omega}$,) 
iff $E$ is filtered by $\mu_{\omega}$-semistable 
sheaves $F_i$ with $\mu_{\omega}(F_i)>0$. (resp.~$\mu_{\omega}(F_i) \le 0$.)
The existence of Harder-Narasimhan filtrations implies that 
$(\tT_{\omega}, \fF_{\omega})$ is a torsion pair, i.e. 
\begin{itemize}
\item For $T\in \tT_{\omega}$
and $F\in \fF_{\omega}$, we have $\Hom(T, F)=0$. 
\item For any object $E\in \Coh_{\pi}(\overline{X})$, 
there is an exact sequence 
\begin{align*}
0 \to T \to E \to F \to 0, 
\end{align*}
with $T \in \tT_{\omega}$ and $F \in \fF_{\omega}$. 
\end{itemize}
The associated tilting is defined 
in the following way. 
(cf.~\cite{HRS}.)
\begin{defi}\label{def:Bom}
We define the category $\bB_{\omega}$ to be 
\begin{align*}
\bB_{\omega} \cneq \langle \fF_{\omega}, 
\tT_{\omega}[-1] \rangle_{\ex} \subset \dD_0. 
\end{align*}
\end{defi}
The category $\bB_{\omega}$ is the heart of a bounded t-structure 
on $\dD_0$,
hence in particular an abelian category. 
We note that $\bB_{\omega}$ is unchanged if 
we replace $\omega$ by $t\omega$
for $t>0$. 
\begin{rmk}
The construction of the heart $\bB_{\omega}$ is an analogue 
of similar constructions on the derived categories of 
coherent sheaves on 
surfaces by Bridgeland~\cite{Brs2}, Arcara-Bertram~\cite{AB}. 
\end{rmk}
\subsection{The heart of a bounded t-structure on $\dD$}\label{subsec:heartD}
Let $\dD$ be a triangulated category defined in 
Definition~\ref{defi:D}.
We define the category $\aA_{\omega}$
as follows. 
\begin{defi}\label{defi:Ao}
We define $\aA_{\omega}$ to be 
\begin{align}
\aA_{\omega}\cneq \langle \pi^{\ast}\Pic(\mathbb{P}^1), \bB_{\omega} 
\rangle_{\ex} \subset \dD. 
\end{align}
\end{defi}
We have the following proposition. 
\begin{prop}\label{heart}
The subcategory $\aA_{\omega} \subset \dD$
is the heart of a bounded t-structure on $\dD$. 
\end{prop}
\begin{proof}
More precisely, we can show that 
there is the heart of a bounded t-structure 
\begin{align*}
\aA_{\omega}' \subset D^b \Coh(\overline{X}),
\end{align*}
which restricts to the heart $\aA_{\omega}$
on $\dD$. The construction of $\aA_{\omega}'$
will be given in Definition~\ref{defi:Ao'}. 
The proof follows from the exactly same argument 
of~\cite[Proposition~3.6]{Tcurve1}, 
by replacing the notation $(\dD, \aA, \dD', \dD_E, \aA_{E})$
in~\cite[Proposition~3.6]{Tcurve1}
by 
\begin{align*}
(D^b \Coh(\overline{X}), \aA_{\omega}', \dD_0, 
\dD, \aA_{\omega}).
\end{align*}
The only modification in the proof 
is that we use Lemma~\ref{instead}
instead of~\cite[Lemma~6.1]{Tcurve1}. The statement and 
the proof of Lemma~\ref{instead} will be given in 
Subsection~\ref{tech:t-st}. 
\end{proof}
The heart $\aA_{\omega}$
satisfies the following property. 
\begin{lem}\label{lem:pos}
For any $E\in \aA_{\omega}$, we have 
\begin{align}\label{ineq:use}
\rank(E) \ge 0, \quad -\ch_2(E) \cdot \omega \ge 0. 
\end{align}
\end{lem}
\begin{proof}
By the definitions of 
$\bB_{\omega}$ and 
$\aA_{\omega}$, 
we may assume $E\in \pi^{\ast}\Pic(\mathbb{P}^1)$
or $E\in \fF_{\omega}$ or $E\in \tT_{\omega}[-1]$. 
In each case, the inequalities (\ref{ineq:use}) are obviously 
satisfied by noting (\ref{compati}). 
\end{proof}

\begin{rmk}
As in~\cite[Lemma~3.5]{Tcurve1}, 
one might expect that there is the heart of 
a bounded t-structure on 
$\langle \oO_{\overline{X}}, \Coh_{\pi}(\overline{X}) \rangle_{\tr}$ given by 
$\langle \oO_{\overline{X}}, \bB_{\omega} \rangle_{\ex}$.
However this is not true, since 
the natural map $\oO_{\overline{X}} \to \oO_{X_0}$
in the category $\langle \oO_{\overline{X}}, \bB_{\omega} \rangle_{\ex}$
does not have a kernel.  
\end{rmk}

\subsection{Bilinear map $\chi$}\label{subsec:gen:bi}
We define the bilinear map $\chi$
\begin{align}\label{def:chi}
\chi \colon \Gamma \times \Gamma_0 \to \mathbb{Z},
\end{align}
as follows:
\begin{align}\label{chi:form}
\chi((R, r, \beta, n), (r', \beta', n'))
=R(2r'+n'). 
\end{align}
By the Riemann-Roch theorem and Lemma~\ref{cy3}, 
we have 
\begin{align}\notag
\chi(\cl(E), \cl_0(F)) &=\dim \Hom_{\dD}(E, F)-\dim
\Ext_{\dD}^1(E, F) \\
\label{chi}
& +\dim  \Ext_{\dD}^1(F, E)- \dim
\Hom_{\dD}(F, E),
\end{align}
for $E\in \aA_{\omega}$ and $F \in \bB_{\omega}$. 
If we define $\chi(v, v') \cneq-\chi(v', v)$
for $v\in \Gamma_0$ and $v' \in \Gamma$, 
then (\ref{chi}) also holds
for $E \in \bB_{\omega}$ and $F \in \aA_{\omega}$. 
The above bilinear map will be used 
in describing the wall-crossing formula 
in Section~\ref{sec:WCF}. 

Note that $\Gamma_0$ is regarded as a 
subgroup of $\Gamma$ via 
$v\mapsto (0, v)$. 
The map $\chi$ restricts to the bilinear pairing, 
\begin{align}\label{chi0}
\chi|_{\Gamma_0 \times \Gamma_0} :
\Gamma_0 \times \Gamma_0 \to \mathbb{Z}, 
\end{align}
which is trivial, i.e. 
$\chi(v, v')=0$ for any $v, v' \in \Gamma_0$. 
In particular for any $E, F \in \bB_{\omega}$, we have 
\begin{align*}
\dim &\Hom_{\dD_0}(E, F)-\dim
\Ext_{\dD_0}^1(E, F) \\
& +\dim  \Ext_{\dD_0}^1(F, E)- \dim
\Hom_{\dD_0}(F, E)=0.
\end{align*}

\subsection{Abelian categories $\aA(r)$}\label{sub:R1}
Here we introduce some
abelian subcategories of $\aA_{\omega}$. 
First we introduce the following 
subcategory of $\Coh_{\pi}(\overline{X})$, 
\begin{align*}
\Coh_{\pi}^{\le 1}(\overline{X})
\cneq \{ E\in \Coh_{\pi}(\overline{X}) :
\dim E \le 1\}.  
\end{align*}
\begin{defi}\label{defi:A(r)}
For $r\in \mathbb{Z}$, we
define the category $\aA(r)$
 to be 
\begin{align*}
\aA(r)\cneq \langle
\pi^{\ast}\oO_{\mathbb{P}^1}(r), 
\Coh_{\pi}^{\le 1}(\overline{X})[-1] \rangle_{\ex}
\subset \aA_{\omega}. 
\end{align*}
\end{defi}
The category $\aA(r)$ 
has a structure of an abelian category. 
In fact it is essentially 
shown in~\cite[Lemma~3.5]{Tcurve1}
that $\aA(0)$ is the heart of a bounded
t-structure on the triangulated category, 
\begin{align*}
\langle 
\oO_{\overline{X}}, \Coh_{\pi}^{\le 1}(\overline{X}) \rangle_{\tr}
\subset D^b \Coh(\overline{X}). 
\end{align*}
Since there is an equivalence of 
categories,  
\begin{align}\label{equiv:A}
\otimes \pi^{\ast}\oO_{\mathbb{P}^1}(r) \colon 
\aA(0) \stackrel{\sim}{\to} \aA(r),
\end{align}
the category $\aA(r)$ also has a structure of 
an abelian category.

\section{Weak stability conditions on $\dD$}\label{section:wstab}
In this section, we construct weak stability 
conditions on our triangulated category $\dD$. 
The notion of weak stability conditions
on triangulated categories
is introduced in~\cite{Tcurve1}, 
generalizing Bridgeland's stability conditions~\cite{Brs1}. 
A weak stability condition 
is interpreted as a limiting degeneration of Bridgeland's 
stability conditions, and 
it is a coarse version of Bayer's polynomial 
stability condition~\cite{Bay}.  It is easier to construct examples
of weak stability
than those of Bridgeland stability, 
and the wall-crossing formula in~\cite{Joy4}, \cite{JS}, \cite{K-S}
is also applied in this framework. 
We remark that most of the results stated in this 
section are technical, and 
their proofs will be given in 
Sections~\ref{sec:reswe}, \ref{sec:semistable}
and \ref{sec:A12}. 
\subsection{General definition}\label{subsec:general}
In this subsection, we recall 
the definition of weak stability 
conditions on triangulated categories in a 
general setting. Let $\tT$ be a triangulated
category, and $K(\tT)$ its 
Grothendieck group.
We fix a finitely generated free 
abelian group $\Gamma$ and 
a group homomorphism, 
\begin{align*}
\cl \colon K(\tT) \to \Gamma. 
\end{align*}
We also fix a filtration, 
\begin{align*}
0=\Gamma_{-1} \subset 
\Gamma_{0} \subset \Gamma_{1} \subset \cdots 
\subset \Gamma_{N}=\Gamma, 
\end{align*}
such that each subquotient $\Gamma_{i}/\Gamma_{i-1}$
is a free abelian group. 
\begin{defi}\label{defi:weak}
A \textit{weak stability condition} on $\tT$ consists of 
data 
\begin{align*}
(Z=\{Z_i\}_{i=0}^{N}, \aA),
\end{align*}
such that each $Z_i$ is a group homomorphism, 
\begin{align*}
Z_i \colon \Gamma_{i}/\Gamma_{i-1} \to \mathbb{C}, 
\end{align*}
and $\aA \subset \tT$ is the heart of a bounded
t-structure on $\tT$, 
satisfying the following conditions: 
\begin{itemize}
\item For any non-zero $E\in \aA$
with $\cl(E) \in \Gamma_{i} \setminus \Gamma_{i-1}$, 
we have 
\begin{align}\label{cond}
Z(E)\cneq Z_i([\cl(E)]) \in \mathbb{H}. 
\end{align}
Here
$[\cl(E)] \in \Gamma_i/\Gamma_{i-1}$ is the class of 
$\cl(E) \in \Gamma_i \setminus \Gamma_{i-1}$ and 
 \begin{align}\label{def:mH}
\mathbb{H} \cneq \{r\exp(i\pi \phi) : r>0, 0<\phi \le 1\}.
\end{align}
We say $E\in \aA$ is \textit{$Z$-(semi)stable} if for any 
exact sequence in $\aA$, 
\begin{align*}
0 \to F \to E \to G \to 0,
\end{align*}
we have the inequality,
\begin{align*}
\arg Z(F) < (\le) \arg Z(G). 
\end{align*} 
\item  For any $E\in \aA$, there is a filtration
in $\aA$, (Harder-Narasimhan filtration,)
\begin{align*}
0=E_0 \subset E_1 \subset \cdots \subset E_n=E, 
\end{align*}
such that each subquotient $F_i=E_i/E_{i-1}$
is $Z$-semistable with
\begin{align*}
\arg Z(F_i)> \arg Z(F_{i+1}),
\end{align*}
for all $i$.
\end{itemize}
\end{defi}
If we take a filtration on $\Gamma$
trivial, i.e. $N=0$, then we 
call a weak stability condition as 
a \textit{stability condition}. 
In this case, the pair 
$(Z, \aA)$ determines a stability condition 
in the sense of
Bridgeland~\cite{Brs1}. 

We denote by $\Stab_{\Gamma_{\bullet}}(\tT)$
the set of weak stability conditions 
on $\tT$, satisfying some technical 
conditions. 
(Local finiteness, Support property.)
The detail of these properties will be
recalled in Section~\ref{sec:reswe}. 
The following theorem 
is proved in~\cite[Theorem~2.15]{Tcurve1}, 
along with the same argument of Bridgeland's 
theorem~\cite[Theorem~7.1]{Brs1}.
 \begin{thm}\label{thm:top}
There is a natural topology on $\Stab_{\Gamma_{\bullet}}(\tT)$
such that the map 
\begin{align*}
\Pi \colon 
\Stab_{\Gamma_{\bullet}}(\tT) 
\to \prod_{i=0}^{N}\Hom_{\mathbb{Z}}(\Gamma_{i}/\Gamma_{i-1}, \mathbb{C}),
\end{align*}
sending $(Z, \aA)$ to $Z$ 
is a local homeomorphism. In particular each connected component 
of $\Stab_{\Gamma_{\bullet}}(\tT)$ is a complex manifold. 
\end{thm}
If $N=0$, then the space $\Stab_{\Gamma_{0}}(\tT)$
is nothing but Bridgeland's space of stability 
conditions~\cite{Brs1}. 
\subsection{Stability conditions on $\dD_0$}
Let $S$ be a K3 surface and $\overline{X}=S \times \mathbb{P}^1$
as in the previous section. 
In this subsection, we construct 
stability conditions on $\dD_0$
where $\dD_0$ is defined by (\ref{def:D0}).
In Subsection~\ref{subsec:Ch}, we
constructed 
a group homomorphism, 
\begin{align*}
\cl_0 \colon K(\dD_0) \to \Gamma_0,
\end{align*}
for $\Gamma_0=\mathbb{Z} \oplus \mathrm{NS}(S)
\oplus \mathbb{Z}$. 
Therefore we have the
space of Bridgeland's stability conditions on 
$\dD_0$, 
\begin{align}\label{Stab0}
\Stab_{\Gamma_0}(\dD_0). 
\end{align}
We construct elements 
of (\ref{Stab0}) following the 
same arguments of~\cite{Brs2}, \cite{AB}, \cite{Tst2}. 
For an ample divisor $\omega$ on $S$, 
we set the group homomorphism, 
$Z_{\omega, 0} \colon \Gamma_0 \to \mathbb{C}$
to be 
\begin{align}\label{charge}
Z_{\omega, 0}(v) \cneq \int_{S} e^{-i\omega}v, \quad 
v\in \Gamma_0.
\end{align}
If we write $v=(r, \beta, n) \in \mathbb{Z}
 \oplus \mathrm{NS}(S) \oplus \mathbb{Z}$, 
then (\ref{charge}) is written as 
\begin{align*}
Z_{\omega, 0}(v)=n-\frac{1}{2}r\omega^2-
(\omega \cdot \beta) \sqrt{-1}. 
\end{align*}
Let $\bB_{\omega} \subset \dD_0$
be the heart of a bounded t-structure 
defined in Definition~\ref{def:Bom}. 
We have the following lemma. 
\begin{lem}\label{lem:Bstab}
For any ample divisor $\omega$ on 
$S$ and $t\in \mathbb{R}_{>0}$, we have 
\begin{align}\label{stabd0}
(Z_{t\omega, 0}, \bB_{\omega}) \in \Stab_{\Gamma_0}(\dD_0). 
\end{align}
\end{lem}
\begin{proof}
The same proofs
as in~\cite[Proposition~7.1]{Brs2}, \cite[Corollary~2.1]{AB}
are applied. 
Also see~\cite[Lemma~6.4]{Tst2}.
\end{proof}
\subsection{Constructions of weak stability conditions on $\dD$}
\label{subsec:Const}
Let $\dD$ be a triangulated category 
defined in Definition~\ref{defi:D}. 
In this subsection, we 
construct weak stability conditions on 
$\dD$. 
Recall that we constructed 
a group homomorphism, 
\begin{align*}
\cl \colon K(\dD) \to \Gamma, 
\end{align*}
for $\Gamma=\mathbb{Z} \oplus \Gamma_0$
in Subsection~\ref{subsec:Ch}. 
We take a filtration of $\Gamma$, 
\begin{align*}
0=\Gamma_{-1} \subset \Gamma_0 \subset \Gamma_1 \cneq \Gamma, 
\end{align*}
where 
$\Gamma_0$ is given by (\ref{def:Gamma0}) and 
the second inclusion is 
given by $v \mapsto (0, v)$. 
Now we have data which defines
the space of weak stability conditions on 
$\dD$. The resulting complex manifold is 
\begin{align*}
\Stab_{\Gamma_{\bullet}}(\dD). 
\end{align*}
Let $\omega$ be an
ample divisor on $S$
and $t\in \mathbb{R}_{>0}$.  
We define the element, 
\begin{align*}
Z_{t\omega} \in \prod_{i=0}^{1}\Hom(\Gamma_{i}/\Gamma_{i-1}, \mathbb{C}), 
\end{align*}
to be the following: 
\begin{align*}
Z_{t\omega, 1}(R) & \cneq
R\sqrt{-1}, \quad R\in \Gamma_1/\Gamma_0 =\mathbb{Z},
\\
Z_{t\omega, 0}(v) &\cneq 
\int_{S} e^{-it\omega}v, 
\quad v\in \Gamma_0. 
\end{align*}
Let $\aA_{\omega} \subset \dD$ be 
the heart defined in Definition~\ref{defi:Ao}. 
We have the following lemma. 
\begin{lem}\label{lem:property}
For any ample divisor $\omega$
on $S$ and $t \in \mathbb{R}_{>0}$, we have 
\begin{align*}
\sigma_{t\omega} \cneq (Z_{t\omega}, \aA_{\omega}) 
\in \Stab_{\Gamma_{\bullet}}(\dD). 
\end{align*}
\end{lem}
\begin{proof}
For a non-zero object $E\in \aA_{\omega}$, 
suppose that $\rank(E) \neq 0$. 
Then $\rank(E)>0$ and 
\begin{align*}
Z_{t\omega}(E) \in \mathbb{R}_{>0} \sqrt{-1},
\end{align*}
by the definition of $Z_{t\omega}$. 
If $\rank (E)=0$, then $E\in \bB_{\omega}
=\bB_{t\omega}$
and we have 
\begin{align*}
Z_{t\omega}(E)=Z_{t\omega, 0}(E) \in \mathbb{H}, 
\end{align*}
by Lemma~\ref{lem:Bstab}, where 
 $\mathbb{H}$ is given 
by (\ref{def:mH}). 
Therefore the condition (\ref{cond}) is satisfied. 
The other properties
(Harder-Narasimhan property, local finiteness, 
support property,)
will be checked in Subsection~\ref{subsec:property}. 
\end{proof}
By~\cite[Lemma~2.17]{Tcurve1},
 the following map is a continuous map,
\begin{align*}
\mathbb{R}_{>0} \ni t \mapsto 
\sigma_{t\omega} 
\in \Stab_{\Gamma_{\bullet}}(\dD). 
\end{align*} 
\begin{rmk}\label{rmk:BA}
The subcategory $\bB_{\omega} \subset \aA_{\omega}$
is closed under subobjects and quotients. 
In particular, an object $E \in \bB_{\omega}$ is 
$Z_{t\omega}$-(semi)stable if and only if
$E$ is 
$Z_{t\omega, 0}$-(semi)stable with 
respect to the pair $(Z_{t\omega, 0}, \bB_{\omega})
\in \Stab_{\Gamma_{0}}(\dD_0)$. 
\end{rmk}

\begin{rmk}
By the construction, for 
an object $E\in \aA_{\omega}$ with 
$\rank(E) \neq 0$, we have 
\begin{align*}
\arg Z_{t\omega}(E)=\frac{\pi}{2}. 
\end{align*}
Therefore the $Z_{t\omega}$-semistability 
of $E$ is checked by comparing 
$\arg Z_{t\omega, 0}(F)$ with $\pi/2$
where $F \in \bB_{\omega}$
is a subobject or a quotient 
of $E$ in $\aA_{\omega}$. 
\end{rmk}

\subsection{Wall and chamber structure}
In this subsection, we see the wall
and chamber structure on 
the parameter space $t\in \mathbb{R}_{>0}$, 
and see what happens for small $t$. 
We introduce the following notation. 
\begin{align*}
M_{t\omega}(R, r, \beta, n)
&\cneq \left\{ 
E \in \aA_{\omega} : 
\begin{array}{l}
E\mbox{ is }
Z_{t\omega}\mbox{-semistable } 
\mbox{ with } \\
\cl(E)=(R, r, \beta, n). 
\end{array}\right\}.
\end{align*}
We have the following proposition. 
\begin{prop}\label{prop:wall}
For fixed $\beta \in \mathrm{NS}(S)$ and 
an ample divisor $\omega$ on $S$, there is a finite
sequence of real numbers, 
\begin{align*}
0=t_0 <t_1 < \cdots <t_{k-1} <t_k=\infty,
\end{align*}
such that the set of objects 
\begin{align*}
\bigcup_{\begin{subarray}{c} (R, r, n), \\
\arg Z_{t\omega}(R, r, \beta, n)=\pi/2
\end{subarray}}
M_{t\omega}(R, r, \beta, n),
\end{align*} 
is constant for each $t \in (t_{i-1}, t_i)$. 
\end{prop}
\begin{proof}
The proof will be given in Subsection~\ref{subsec:wall}. 
\end{proof}
For small $t$, we have the following proposition. 
\begin{prop}\label{prop:m2}
In the same situation of Proposition~\ref{prop:wall},
we have  
\begin{align*}
M_{t\omega}(R, r, \beta, n)=\emptyset,  
\end{align*}
for any $t\in (0, t_1)$ and $(R, r, n) \in \mathbb{Z}^{\oplus 3}$
 with $R \ge 1$ and $n\neq 0$. 
\end{prop}
\begin{proof}
The proof will be given in Subsection~\ref{subsec:finset2}. 
\end{proof}

\subsection{Comparison with $\mu_{i\omega}$-limit semistable objects}
Let $\aA(r) \subset \aA_{\omega}$ be the subcategory 
defined in Definition~\ref{defi:A(r)}. 
In this subsection,
we relate 
$Z_{t\omega}$-semistable objects in $\aA_{\omega}$
for $t\gg 0$ to certain semistable objects in 
 $\aA(r)$. 
\begin{defi}\label{def:lim}
 An object $E\in \aA(r)$
with $\rank(E)=1$ 
is $\mu_{i\omega}$-limit (semi)stable if
the following conditions hold:
\begin{itemize}
\item For any exact sequence 
$0 \to F \to E \to G \to 0$
in $\aA(r)$
with $F \in \Coh_{\pi}^{\le 1}(\overline{X})[-1]$, 
we have $\ch_3(F) \ge 0$. 
\item For any exact sequence 
$0 \to F \to E \to G \to 0$
in $\aA(r)$
with $G \in \Coh_{\pi}^{\le 1}(\overline{X})[-1]$, 
we have $\ch_3(G) \le 0$. 
\end{itemize}
\end{defi}
Note that if $E\in \aA(r)$ 
satisfies $\rank(E)=0$, then 
$E\in \Coh_{\pi}^{\le 1}(\overline{X})[-1]$. 
We also call $E\in \aA(r)$ with 
$\rank(E)=0$ to be $\mu_{i\omega}$-\textit{limit
(semi)stable} if $E[1] \in \Coh_{\pi}(\overline{X})$
is $\omega$-Gieseker (semi)stable. 

The $\mu_{i\omega}$-limit stability 
coincides with the 
same notion  
discussed in~\cite[Section~3]{Tolim2}, so we have 
employed the same notation here.  
To be more precisely, we have the following lemma: 
\begin{lem}\label{rmk:lim}
Take an object $E\in D^b \Coh(\overline{X})$
satisfying 
\begin{align}\notag
\ch(E)=(R, 0, -\beta, -n) \in \Gamma \subset H^{\ast}(\overline{X}, 
\mathbb{Q}),
\end{align}
for $R \le 1$. 
Then $E$ 
is an $\mu_{i\omega}$-limit semistable
object in $\aA(0)$  
iff $E[1]$ is an $\mu_{i\omega}$-limit semistable 
object 
in the sense of~\cite[Section~3]{Tolim2}. 
\end{lem}
\begin{proof}
The notion of $\mu_{i\omega}$-stability 
in~\cite[Section~3]{Tolim2}
together with the proof of this result 
will be given in Subsection~\ref{subsec:mulim}.
\end{proof}
In what follows, we use Definition~\ref{def:lim}
for the definition of $\mu_{i\omega}$-limit stability. 
For $R\le 1$, we set
\begin{align*}
M_{\rm{lim}}(R, r, \beta, n)
&\cneq \left\{ 
E \in \aA(r) : 
\begin{array}{l}
E\mbox{ is }
\mu_{i\omega} \mbox{-limit semistable } 
\\
\mbox{ with }
\cl(E)=(R, r, \beta, n). 
\end{array} \right\}.
\end{align*}
We have the following proposition. 
\begin{prop}\label{prop:m1}
In the same situation of Proposition~\ref{prop:wall},
we have   
\begin{align*}
M_{t\omega}(R, r, \beta, n)=M_{\rm{lim}}(R, r, \beta, n),  
\end{align*}
for any $t\in (t_{k-1}, \infty)$
and $R \le 1$ satisfying 
$\arg Z_{t\omega}(R, r, \beta, n)=\pi/2$. 
\end{prop}
\begin{proof}
The proof will be given in Subsection~\ref{subsec:finset}. 
\end{proof}
By the equivalence (\ref{equiv:A}), the following lemma is 
obvious. 
\begin{lem}\label{lem:Mr}
For $R\le 1$, 
we have the following bijection of objects, 
\begin{align*}
\otimes \pi^{\ast}\oO_{\mathbb{P}^1}(r) \colon 
M_{\rm{lim}}(R, 0, \beta, n)
\stackrel{1:1}{\to} M_{\rm{lim}}(R, Rr, \beta, n). 
\end{align*}
\end{lem}

\subsection{Abelian category $\aA_{\omega}(1/2)$}\label{subsec:abelian12}
In this subsection, we introduce
a certain abelian category generated by $Z_{t\omega}$-semistable 
objects for sufficiently small $t$. 
The following is an analogue 
of Bayer's polynomial stability condition~\cite{Bay}. 
\begin{defi}\label{defi:polynomial}
An object $E \in \aA_{\omega}$ is 
$Z_{0\omega}$-(semi)stable if for any exact sequence 
$0 \to F \to E \to G \to 0$ in $\aA_{\omega}$
with $F, G \neq 0$, we have 
\begin{align*}
\arg Z_{t\omega}(F)<(\le) \arg Z_{t\omega}(G), 
\end{align*}
for $0<t \ll 1$. 
\end{defi}
The same proof of Lemma~\ref{lem:property} shows that 
there are Harder-Narasimhan filtrations 
with respect to $Z_{0\omega}$-stability. 
For $\phi \in [0, 1]$, we set 
\begin{align*}
\aA_{\omega}(\phi)
\cneq \left\langle E \in \aA_{\omega}  :
\begin{array}{l}
E \mbox{ is }Z_{0\omega}\mbox{-semistable with }\\
\lim_{t\to 0}\arg Z_{t\omega}(E)=\phi \pi
\end{array} \right\rangle_{\ex}.
\end{align*}
By the definition of $Z_{t\omega}$, 
we have $\aA_{\omega}(\phi) \neq \{ 0\}$ 
only if $\phi \in \{0, 1/2, 1\}$. 
Here we focus on the case of $\phi=1/2$. 
We have the following lemma. 
\begin{lem}\label{lem:lem12}
(i) An object $E \in \aA_{\omega}$
is $Z_{0\omega}$-(semi)stable if and only 
if $E$ is $Z_{t\omega}$-(semi)stable
for $0<t \ll 1$. 
 
(ii)
Any object $E \in \aA_{\omega}(1/2)$ satisfies 
$\ch_3(E)=0$. 

(iii) The category $\aA_{\omega}(1/2)$ is an 
abelian subcategory of $\aA_{\omega}$.  
\end{lem}
\begin{proof}
The proof will be given in Subsection~\ref{subsec:lem12}
\end{proof}
We also use the following notation, 
\begin{align*}
\bB_{\omega}(1/2) \cneq 
\aA_{\omega}(1/2) \cap \bB_{\omega}. 
\end{align*}

\subsection{Weak stability conditions on $\aA_{\omega}(1/2)$}
\label{subsec:weakA12}
We construct weak stability conditions on 
the abelian category
$\aA_{\omega}(1/2)$. 
We define
finitely generated free abelian 
groups $\widehat{\Gamma}$, $\widehat{\Gamma}_0$
and group homomorphisms 
$\widehat{\cl}$, $\widehat{\cl}_0$, 
\begin{align*}
\widehat{\cl} &\colon 
 K(\aA_{\omega}(1/2)) \to \widehat{\Gamma} \cneq 
\mathbb{Z} \oplus \mathbb{Z} \oplus \mathrm{NS}(S), \\
\widehat{\cl}_0 &\colon 
 K(\bB_{\omega}(1/2)) \to \widehat{\Gamma}_0 \cneq 
 \mathbb{Z} \oplus \mathrm{NS}(S),
\end{align*}
 to be
\begin{align*}
\widehat{\cl}(E) &\cneq (\ch_0(E), \ch_1(E), \ch_2(E)), \\
\widehat{\cl}_0(E) &\cneq (\ch_1(E), \ch_2(E)).
\end{align*}
Here as in Subsection~\ref{subsec:Ch}, we have 
regarded $\ch_0(E)$, $\ch_1(E)$
as integers, and $\ch_2(E)$ as an element of 
$\mathrm{NS}(S)$. 
We take the following filtration of $\widehat{\Gamma}$, 
\begin{align*}
0=\widehat{\Gamma}_{-1} \subset 
\widehat{\Gamma}_0  \subset 
\widehat{\Gamma}_1 \cneq \widehat{\Gamma}. 
\end{align*}
Here the embedding $\widehat{\Gamma}_0 \subset \widehat{\Gamma}$
is given by $(r, \beta) \mapsto (0, r, \beta)$. 
For $\theta \in (0, 1)$, we construct the element 
\begin{align*}
\widehat{Z}_{\omega, \theta} \in \prod_{i=0}^{1}
\Hom(\widehat{\Gamma}_i/\widehat{\Gamma}_{i-1}, \mathbb{C}),
\end{align*}
as follows, 
\begin{align*}
\widehat{Z}_{\omega, \theta, 1}(R) & \cneq Re^{i\pi \theta}, \quad 
R \in \widehat{\Gamma}_1/\widehat{\Gamma}_0=\mathbb{Z}, \\
\widehat{Z}_{\omega, \theta, 0}(r, \beta)& \cneq
-r-(\beta \cdot \omega)\sqrt{-1}, \quad 
(r, \beta) \in \widehat{\Gamma}_0. 
\end{align*}
We have the following lemma. 
\begin{lem}\label{lem:weakA12}
For any ample divisor $\omega$ on $S$
and $0<\theta<1$, we have
\begin{align*}
(\widehat{Z}_{\omega, \theta}, \aA_{\omega}(1/2)) 
\in \Stab_{\widehat{\Gamma}_{\bullet}}(D^b(\aA_{\omega}(1/2))). 
\end{align*}
\end{lem}
\begin{proof}
The same proof of Lemma~\ref{lem:property} is applied, and we omit the detail. 
\end{proof}
The relationship between 
$\widehat{Z}_{\omega, 1/2}$-stability and $Z_{0\omega}$-stability 
is given as follows. 
\begin{lem}\label{lem:t0c}
An object $E \in \aA_{\omega}$ is 
$Z_{0\omega}$-semistable
satisfying 
\begin{align*}
\lim_{t\to 0}\arg Z_{t\omega}(E) = \pi/2,
\end{align*}
if and only if 
$E \in \aA_{\omega}(1/2)$ and $E$ is 
$\widehat{Z}_{\omega, 1/2}$-semistable.
\end{lem}
\begin{proof}
The proof will be given in Subsection~\ref{subsec:t0c}. 
\end{proof}

\subsection{Semistable objects in $\aA_{\omega}(1/2)$}
\label{subsec:semiA12}
We set 
\begin{align*}
\widehat{M}_{\omega, \theta}(R, r, \beta) \cneq 
\left\{E \in \aA_{\omega}(1/2): 
\begin{array}{l}
E \mbox{ is }\widehat{Z}_{\omega, \theta}\mbox{-semistable with }\\
\widehat{\cl}(E)=(R, r, \beta).
\end{array}  \right\}. 
\end{align*}
Similarly to Proposition~\ref{prop:wall} and Proposition~\ref{prop:m2}, 
we have the following proposition. 
\begin{prop}\label{prop:propA12}
For fixed $\beta \in \mathrm{NS}(S)$ and 
an ample divisor $\omega$ on $S$,
there is a finite sequence, 
\begin{align*}
0=\theta_k<\theta_{k-1}< \cdots <\theta_{1}<\theta_0=1/2,
\end{align*} 
such that the following holds. 

(i) The set of objects 
\begin{align*}
\bigcup_{(R, r), R \ge 1}\widehat{M}_{\omega, \theta}(R, r, \beta),
\end{align*}
is constant for $\theta \in (\theta_{i-1}, \theta_i)$. 

(ii) For $0<t \ll 1$ and any 
$(R, r, \beta) \in \widehat{\Gamma}$, we have 
\begin{align*}
\widehat{M}_{\omega, 1/2}(R, r, \beta)=M_{t\omega}(R, r, \beta, 0). 
\end{align*}

(iii) For $\theta \in (0, \theta_{k-1})$, we have 
\begin{align*}
\widehat{M}_{\omega, \theta}(1, r, \beta)=\left\{ \begin{array}{cc}
\{\pi^{\ast}\oO_{\mathbb{P}^1}(r) \}, & \mbox{ if } \beta=0, \\
\emptyset, & \mbox{ if } \beta \neq 0. 
\end{array} \right. 
\end{align*}
\end{prop}
\begin{proof}
The proof will be given in Subsection~\ref{subsec:propA12}
\end{proof}

\section{Counting invariants}\label{sec:inv}
In this section, we
discuss several
counting invariants on $X$
and $\overline{X}$, 
which appeared in the introduction.  
\subsection{Stable pairs}\label{subsec:stablepair}
In this subsection, we recall the notion of 
stable pairs introduced by Pandharipande-Thomas~\cite{PT}. 
Let $S$ be a K3 surface and 
$X=S \times \mathbb{C}$, 
as in Subsection~\ref{subsec:local}. 
Note that we have 
the subcategory, 
\begin{align}\label{Cohsub}
\Coh_{\pi}(X) \subset \Coh_{\pi}(\overline{X}), 
\end{align}
consisting of sheaves supported on fibers of 
the projection
$\pi|_{X} \colon X\to \mathbb{C}$. 
\begin{defi}\label{defi:stablepair}
A stable pair on 
$X$ is a pair $(F, s)$, 
\begin{align*}
F \in \Coh_{\pi}(X), \quad  s \colon \oO_{X} \to F, 
\end{align*}
satisfying the following conditions. 
\begin{itemize}
\item 
The sheaf $F$ is a pure one 
dimensional sheaf. 
\item The morphism $s$ is surjective in 
dimension one. 
\end{itemize}
\end{defi}
For $\beta \in H_2(X, \mathbb{Z})$
and $n\in \mathbb{Z}$, the moduli space 
of stable pairs $(F, s)$ satisfying 
\begin{align*}
[F]=\beta, \quad \chi(F)=n, 
\end{align*}
is denoted by 
\begin{align}\label{moduli:pair}
P_n(X, \beta). 
\end{align}
If we replace $X$ by 
$\overline{X}$ in 
Definition~\ref{defi:stablepair}, we 
also have the notion of stable pairs 
on $\overline{X}$. 
By regarding $\beta$
as an element of $H_2(\overline{X}, \mathbb{Z})$,
we also have 
the similar moduli space 
$P_n(\overline{X}, \beta)$
and an open embedding,  
\begin{align*}
P_n(X, \beta) \subset P_n(\overline{X}, \beta). 
\end{align*} 
The moduli space $P_n(\overline{X}, \beta)$
is proved to be a projective scheme in~\cite{PT}, 
hence in particular $P_n(X, \beta)$ is a 
quasi-projective scheme.
In what follows, 
we regard 
an algebraic class 
$\beta \in H_2(X, \mathbb{Z})$
 as an element of $\mathrm{NS}(S)$, 
by the natural isomorphism 
$H_2(X, \mathbb{Z}) \cong H_2(S, \mathbb{Z})$
and the Poincar\'e duality. 

We are interested in the generating series of 
the Euler characteristic of the moduli space (\ref{moduli:pair}).
\begin{defi}
We define the generating series $\PT^{\chi}(X)$ to be
\begin{align*}
\PT^{\chi}(X)\cneq
\sum_{\begin{subarray}{c}
\beta \in \mathrm{NS}(S), \\
n \in \mathbb{Z}
\end{subarray}}
\chi(P_n(X, \beta))y^{\beta}z^n. 
\end{align*}
\end{defi}
Let $(F, s)$ be stable pair on $\overline{X}$. 
We remark that
 if we regard
a pair $(F, s)$
 as a two term complex, 
\begin{align}\label{comp}
I^{\bullet}=(\oO_{\overline{X}} \stackrel{s}{\to} F) \in \dD,
\end{align}
then $P_n(\overline{X}, \beta)$
is also interpreted as a moduli space of 
two term complexes (\ref{comp}) satisfying 
\begin{align}\label{cl:stable:pair}
\cl(I^{\bullet})=(1, 0, -\beta, -n),
\end{align}
in the notation in Subsection~\ref{subsec:Ch}. 
(cf.~\cite[Section~2]{PT}.)

As we stated in the introduction, our 
goal is to give a formula relating 
$\PT^{\chi}(X)$ to other invariants. 
We first prove a formula for the 
generating series on $\overline{X}$, 
\begin{align*}
\PT^{\chi}(\overline{X})\cneq
\sum_{\begin{subarray}{c}
\beta \in \mathrm{NS}(S), \\
n \in \mathbb{Z}
\end{subarray}}
\chi(P_n(\overline{X}, \beta))y^{\beta}z^n. 
\end{align*}
These series are related as follows. 
\begin{lem}\label{lem:eqPT}
We have the equality, 
\begin{align}\label{eqPT}
\PT^{\chi}(\overline{X}) =\PT^{\chi}(X)^2. 
\end{align}
\end{lem}
\begin{proof}
A standard $\mathbb{C}^{\ast}$-action on 
$\mathbb{C}$ induces $\mathbb{C}^{\ast}$-actions 
on $X=S\times \mathbb{C}$
 and $\overline{X}=S\times \mathbb{P}^1$
via acting on the second factors. 
Hence $\mathbb{C}^{\ast}$ acts on the moduli
spaces $P_n(X, \beta)$ and $P_n(\overline{X}, \beta)$.  
A stable pair $(F, s)$ on $\overline{X}$
is $\mathbb{C}^{\ast}$-fixed if and only if 
\begin{align*}
F\cong F_0 \oplus F_{\infty}, \quad 
s=(s_0, s_{\infty}) \in \Gamma(F_0) \oplus \Gamma(F_{\infty}),
\end{align*} 
where $(F_0, s_0)$ and $(F_{\infty}, s_{\infty})$
determine $\mathbb{C}^{\ast}$-fixed 
stable pairs on $U_{0}=S\times \mathbb{C}$
and $U_{\infty}=S\times (\mathbb{P}^1 \setminus \{0\})$
respectively. 
Since both of $U_0$ and $U_{\infty}$ are 
$\mathbb{C}^{\ast}$-equivalently isomorphic to $X$, 
we have 
\begin{align*}
P_n(\overline{X}, \beta)^{\mathbb{C}^{\ast}}
\cong \coprod_{\begin{subarray}{c}
\beta_1 +\beta_2=\beta, \\
n_1 +n_2 =n
\end{subarray}}
P_{n_1}(X, \beta_1)^{\mathbb{C}^{\ast}} \times 
P_{n_2}(X, \beta_2)^{\mathbb{C}^{\ast}}. 
\end{align*}
Taking the Euler characteristic and the 
$\mathbb{C}^{\ast}$-localization, 
we obtain (\ref{eqPT}). 
\end{proof}
\subsection{Product expansion formula}\label{subsec:product}
In the paper~\cite{Tolim2}, the author 
essentially proved the following result. 
\begin{thm}{\bf \cite[Theorem~1.3]{Tolim2}}\label{thm:PNL}
For each $(\beta, n) \in \mathrm{NS}(S) \oplus \mathbb{Z}$, 
there are invariants, 
\begin{align}\label{NL}
N(0, \beta, n) \in \mathbb{Q}, \quad L(\beta, n) \in \mathbb{Q}, 
\end{align}
such that we have the following formula: 
\begin{align}\label{PNL}
\mathrm{PT}^{\chi}(\overline{X})
=\prod_{\beta>0, n>0}\exp\left(n N(0, \beta, n)y^{\beta}z^n  \right)
\left(\sum_{\beta, n}L(\beta, n)y^{\beta}z^n \right). 
\end{align}
\end{thm}
Roughly speaking, the invariants (\ref{NL})
are given in the following way. 
\begin{itemize}
\item The invariant $N(0, \beta, n)$ is a counting 
invariant of $\omega$-Gieseker semistable sheaves
$F \in \Coh_{\pi}(\overline{X})$, satisfying 
\begin{align*}
\cl_0(F)=(0, \beta, n).
\end{align*}
If we denote the 
moduli space of such sheaves by 
$\mM_{\omega, \overline{X}}(0, \beta, n)$, 
then $N(0, \beta, n)$ is given by 
\begin{align*}
N(0, \beta, n)=`\chi'(\mM_{\omega, \overline{X}}(0, \beta, n)). 
\end{align*}
\item The invariant $L(\beta, n)$ 
is a counting invariant of $\mu_{i\omega}$-limit semistable
 objects $E \in \aA(0)$, 
satisfying (cf.~Definition~\ref{def:lim},)
\begin{align*}
\cl(E)=(1, 0, -\beta, -n).
\end{align*} 
If we denote the moduli space of such objects by 
$\mM_{\rm{lim}}(1, 0, -\beta, -n)$, then 
$L(\beta, n)$ is given by 
\begin{align}\label{Lchi}
L(\beta, n)=`\chi'(\mM_{\rm{lim}}(1, 0, -\beta, -n)). 
\end{align}
\end{itemize}
The precise
definitions of $N(0, \beta, n)$
and $L(\beta, n)$ will be 
recalled in Definition~\ref{defi:Ninv} and 
Subsection~\ref{subsec:DT=L} respectively. 
If the moduli space $\mM_{\omega}(0, \beta, n)$
or $\mM_{\rm{lim}}(1, 0, -\beta, -\omega)$
consists of only $\omega$-Gieseker stable sheaves 
or $\mu_{i\omega}$-limit stable objects, 
then $`\chi'$ is the usual Euler 
characteristic of the moduli space. 
If there is a strictly semistable sheaves or objects, 
then the moduli space is an algebraic stack 
with possibly complicated stabilizers. In that case, 
we need to define $`\chi'$
so that the contributions of stabilizers are involved. 
This is worked out by Joyce~\cite{Joy4} using the 
Hall algebra, which we discuss in the next
subsection.

\begin{rmk}
The invariants $N(0, \beta, n)$
and $L(0, \beta, n)$ are denoted by 
$N_{n, \beta}^{eu}$ and $L_{n, \beta}^{eu}$
in~\cite[Theorem~1.3]{Tolim2}
respectively. 
\end{rmk}

\subsection{Hall algebra of $\aA_{\omega}$}
In this subsection, 
we recall the notion of Hall algebra 
associated to $\aA_{\omega}$. 
First Lieblich~\cite{LIE} constructs an
algebraic stack $\mM$ locally of finite type 
over $\mathbb{C}$, which parameterizes 
objects $E\in D^b\Coh(\overline{X})$ satisfying 
\begin{align*}
\Hom_{\overline{X}}(E, E[i])=0, \ i<0. 
\end{align*}
Then we have the substack, 
\begin{align}\label{stack:objA}
\oO bj(\aA_{\omega}) \subset \mM, 
\end{align}
which parameterizes objects $E\in \aA_{\omega}$. 
At this moment, 
we discuss under the assumption that 
$\oO bj(\aA_{\omega})$ is also an algebraic 
stack locally of finite type. 
A necessary result will be given in Lemma~\ref{lem:constructible}
below. 
\begin{defi}\label{def:HA}
We define the $\mathbb{Q}$-vector space $\hH(\aA_{\omega})$
to be spanned by the isomorphism classes of symbols, 
\begin{align*}
[\xX \stackrel{f}{\to} \oO bj(\aA_{\omega})], 
\end{align*}
where $\xX$ is an algebraic stack of finite type 
over $\mathbb{C}$ with affine stabilizers, and 
$f$ is a morphism of stacks. 
The relations are generated by 
\begin{align}\label{rel:gen}
[\xX \stackrel{f}{\to} \oO bj(\aA_{\omega})]
-[\yY \stackrel{f|_{\yY}}{\to}\oO bj(\aA_{\omega})]
-[\uU \stackrel{f|_{\uU}}{\to}\oO bj(\aA_{\omega})],
\end{align}
for a closed substack $\yY \subset \xX$
and $\uU=\xX \setminus \yY$. 
\end{defi}
Here 
two symbols $[\xX_i \stackrel{f_i}{\to} \oO bj(\aA_{\omega})]$
for $i=1, 2$ are isomorphic if there 
is an isomorphism 
$g \colon \xX_1 \stackrel{\cong}{\to} \xX_2$, which 
2-commutes with $f_i$. 

Let $\eE x(\aA_{\omega})$ be the stack of 
short exact sequences in $\aA_{\omega}$. 
There are morphisms of stacks, 
\begin{align*}
p_i \colon \eE x(\aA_{\omega}) \to \oO bj(\aA_{\omega}), 
\end{align*}
for $i=1, 2, 3$, 
sending a short exact sequence 
\begin{align*}
0 \to E_1 \to E_2 \to E_3 \to 0
\end{align*}
to the object $E_i$ respectively. 

There is an associative $\ast$-product on $\hH(\aA_{\omega})$, 
defined by 
\begin{align}\label{ast:prod}
[\xX \stackrel{f}{\to} \oO bj(\aA_{\omega})]
\ast [\yY \stackrel{g}{\to} \oO bj(\aA_{\omega})]
=[\zZ \stackrel{p_2 \circ h}{\to} \oO bj(\aA_{\omega})],
\end{align}
where $\zZ$ and $h$ fit into the Cartesian square, 
\begin{align*}
\xymatrix{
\zZ \ar[r]^{h}
\ar[d] & \eE x(\aA_{\omega}) \ar[r]^{p_2} \ar[d]^{(p_1, p_3)}
 & \oO bj(\aA_{\omega}) \\
\xX \times \yY \ar[r]^{f\times g} &
\oO bj(\aA_{\omega})^{\times 2}. & 
}
\end{align*}
The above $\ast$-product is associative by~\cite[Theorem~5.2]{Joy2}.

\subsection{Invariants via Hall algebra}\label{subsec:invvia}
In this subsection, 
we construct counting invariants of 
$Z_{t\omega}$-semistable objects
in $\aA_{\omega}$ via 
the algebra $(\hH(\aA_{\omega}), \ast)$.
For $v\in \Gamma$ with $\rank(v) \le 1$, 
let 
\begin{align}\label{stack:hall}
\mM_{t\omega}(v) \subset \oO bj(\aA_{\omega}),
\end{align}
be the substack 
which parameterizes
$Z_{t\omega}$-semistable objects $E\in \aA_{\omega}$
with $\cl(E)=v$.
For simplicity, we assume that (\ref{stack:hall}) 
is an algebraic stack of finite type over $\mathbb{C}$.
Again a necessary result will be 
given in Lemma~\ref{lem:constructible}.
We can define the element in $\hH(\aA_{\omega})$
to be  
\begin{align*}
\delta_{t\omega}(v) \cneq 
[\mM_{t\omega}(v) \hookrightarrow \oO bj(\aA_{\omega})]
\in \hH(\aA_{\omega}). 
\end{align*}
The `logarithm' of $\delta_{t\omega}(v)$ is defined as follows:
\begin{defi}
We define $\epsilon_{t\omega}(v) \in \hH(\aA_{\omega})$ to be
\begin{align}\label{sum:eps}
\epsilon_{t\omega}(v) \cneq 
\sum_{\begin{subarray}{c}l\ge 1,
 v_1 + \cdots +v_l=v, v_i \in \Gamma, \\
\arg Z_{t\omega}(v_i)=\arg Z_{t\omega}(v)
\end{subarray}}
\frac{(-1)^{l-1}}{l}\delta_{t\omega}(v_1) 
\ast \cdots \ast \delta_{t\omega}(v_l). 
\end{align}
\end{defi}
Note that the $v_i$ in a non-zero term of the sum (\ref{sum:eps})
satisfies $\rank(v_i)=0$ or $1$. 
Also we have the following lemma:
\begin{lem}\label{lem:finsum}
The sum (\ref{sum:eps}) is a finite sum, 
hence $\epsilon_{t\omega}(v)$ is well-defined. 
\end{lem}
\begin{proof}
The case of $\rank(v)=0$ is essentially proved in~\cite[Lemma~5.12]{Tst3}.
Suppose that $\rank(v)=1$ and
write $v=(1, r, \beta, n)$. 
Let $v_i \in \Gamma$ be an element 
which appears in a non-zero term of the sum (\ref{sum:eps}).  
Then there is unique $1\le e\le l$ such that 
$\rank(v_e)=1$ and $\rank(v_i)=0$
for $i\neq e$. 
We write $v_i=(0, r_i, \beta_i, n_i)$
for $i\neq e$. 
Since $0<-\beta_i \cdot \omega \le -\beta \cdot \omega$, 
the number $l$ in the sum (\ref{sum:eps})
is bounded, and $\beta_i^2$ is bounded above 
by the Hodge index theorem. 
By the condition $\arg Z_{t\omega}(v_i)=\pi/2$, 
we have 
\begin{align}\label{Re0}
\Ree Z_{t\omega}(v_i)=n_i -\frac{1}{2}r^2 t^2\omega^2=0,
\end{align}
for $i\neq e$. 
Also there is an $Z_{t\omega, 0}$-semistable
object $E\in \bB_{\omega}$
with $\cl(E_i)=v_i$, 
hence the same 
proof of Lemma~\ref{lem:Bog} shows the inequality, 
\begin{align}\label{Bog2}
\beta_i^2 +(\beta_i \cdot \omega)^2
 \ge 2r_i(r_i +n_i),
\end{align}
for $i\neq e$. 
Since $\beta_i^2$ is bounded above, 
the equality (\ref{Re0}) and the inequality (\ref{Bog2})
shows the boundedness of $r_i$ and $n_i$. 
Also $\beta_i^2$ and $\beta_i \cdot \omega$
are bounded, hence there is only a finite 
number of possibilities for $\beta_i$. 
\end{proof}
There is a map,
(cf.~\cite[Definition~2.1]{Joy5},)
\begin{align}\label{map:P}
P_q \colon \hH (\aA_{\omega}) \to \mathbb{Q}(q^{1/2}), 
\end{align}
such that if $G$ is a special algebraic group (cf.~\cite[Definition~2.1]{Joy5})
acting on a variety $Y$, then we have 
\begin{align*}
P_q \left(\left[[Y/G] \stackrel{f}{\to}\oO bj(\aA_{\omega})\right] \right)
=P_q(Y)/P_q(G), 
\end{align*}
where $P_q(Y)$ is the virtual 
Poincar\'e polynomial of $Y$. Namely if $Y$ is 
smooth and projective, $P_q(Y)$ is given by 
\begin{align*}
P_q(Y)=\sum_{i\in \mathbb{Z}}(-1)^i \dim H^i(Y, \mathbb{C})q^{i/2},
\end{align*}
and $P_q(Y)$ is defined for any $Y$
using the relation (\ref{rel:gen}) for varieties. 
\begin{thm}{\bf \cite[Theorem~6.2]{Joy5}}
The following limit exists, 
\begin{align*}
\lim_{q^{1/2} \to 1}
(q -1)P_q(\epsilon_{t\omega}(v)) \in \mathbb{Q}.
\end{align*}
\end{thm}
Using the above theorem, we can 
define the counting invariants. 
First we define the invariants 
of rank one objects. 
\begin{defi}\label{defi:DTchi}
For $v \in \Gamma_0$, we 
define 
$\DT^{\chi}_{t\omega}(v) \in \mathbb{Q}$
to be 
\begin{align}\label{def:DT}
\DT^{\chi}_{t\omega}(v)
\cneq 
\lim_{q^{1/2}\to 1}(q -1)P_q(\epsilon_{t\omega}(1, -v)).
\end{align}
Here $(1, -v) \in \Gamma=\mathbb{Z}\oplus \Gamma_0$.
\end{defi}
\begin{rmk}
As we remarked in~\cite[Remark~4.10]{Tcurve1},
if any object $E\in M_{t\omega}(1, -v)$
is $Z_{t\omega}$-stable, then 
the invariant (\ref{def:DT}) 
coincides with the Euler characteristic 
of the moduli space of
objects in $M_{t\omega}(1, -v)$. 
However if there is a strictly semistable 
object $E\in M_{t\omega}(1, -v)$, 
then the stabilizer group $\Aut(E)$ contributes 
to the denominator of the invariant (\ref{def:DT}). 
\end{rmk}
\begin{rmk}
The change of the sign of $v$
in (\ref{def:DT})
is to make the notation compatible 
with Chern characters of stable pairs (\ref{cl:stable:pair}). 
\end{rmk}

\subsection{Moduli stacks}
So far we have assumed that the 
stacks $\oO bj(\aA_{\omega})$ and $\mM_{t\omega}(v)$ 
are algebraic stacks locally of finite type, 
finite type respectively. 
However these are too strong 
assumptions
for our purpose.
In fact it is 
 enough to show the following lemma
by discussing with 
the framework of 
Kontsevich-Soibelman~\cite[Section~3]{K-S}.
We remark that, the proof here is 
technical, and use 
some of the results 
which will be proved in later sections. 
The readers may skip the proof here
at the first reading, and back after 
reading Sections~\ref{sec:tech}, \ref{sec:semistable}.  

\begin{lem}\label{lem:constructible}
(i) 
The set of $\mathbb{C}$-valued points of the 
substack $\oO bj(\aA_{\omega}) \subset \mM$
is a countable union of constructible 
subsets in $\mM$. 

(ii) 
For $v\in \Gamma$ with $\rank(v) \le 1$, 
the set of $\mathbb{C}$-valued points of 
the substack $\mM_{t\omega}(v) \subset \mM$
is a constructible subset in $\mM$. 
\end{lem}
\begin{proof}
(i)
We first note that the 
stack 
\begin{align}\label{stackB}
\oO bj(\bB_{\omega}) \subset \mM, 
\end{align}
which parameterizes objects $E \in \bB_{\omega}$
is an algebraic stack locally of finite type 
over $\mathbb{C}$. This result can 
be proved along with the same 
argument of~\cite[Lemma~4.7]{Tst3}.
Moreover
if $v\in \Gamma$ satisfies $\rank(v)=0$, then 
 the same proof of~\cite[Theorem~4.12]{Tst3}
shows that 
the substack, 
\begin{align*}
\mM_{t\omega}(v) \subset \oO bj(\bB_{\omega}),
\end{align*}
is an open substack of finite type over $\mathbb{C}$. 
Since any object $E\in \bB_{\omega}$ has a 
Harder-Narasimhan filtration with respect to 
$Z_{t\omega, 0}$-stability, the stack (\ref{stackB})
is a countable union of constructible subsets in 
$\mM$. 
Now we note that any object $E\in \aA_{\omega}$
has a filtration with each subquotient 
isomorphic to either an object in $\bB_{\omega}$
or in $\pi^{\ast}\Pic(\mathbb{P}^1)$. 
This fact easily implies that $\oO bj(\aA_{\omega})$
is a countable union of constructible subsets in $\mM$. 
(See the proof of~\cite[Lemma~3.2]{Tcurve3}.)

(ii) 
Take an element $v\in \Gamma$. 
As we discussed in the proof of (i), 
the result for the case of $\rank(v)=0$ 
is 
essentially proved in~\cite[Theorem~4.12]{Tst3}.
Suppose that $\rank(v)=1$, 
and we write it as 
\begin{align*}
v=(1, r, \beta, n) \in \Gamma. 
\end{align*}
We show that $\mM_{t\omega}(v)$ is a constructible 
subset in $\mM$. 
Let $E\in \aA_{\omega}$ be a $Z_{t\omega}$-semistable
object 
with $\cl(E)=v$. 
By Lemma~\ref{lem:rank10} below, there is a filtration in $\aA_{\omega}$, 
\begin{align*}
0=E_0 \subset E_1 \subset E_2 \subset E_3=E, 
\end{align*}
such that $K_i=E_{i}/E_{i-1}$ 
satisfies the condition 
(\ref{filt:filt}). 
Also by
the definition of $\aA(r)$,
 the object $K_2 \in \aA(r)$ has a filtration 
in $\aA(r)$, 
\begin{align*}
0=K_{2, 0} \subset K_{2,1} \subset K_{2,2} \subset K_{2,3} =K_2
\end{align*}
such that $M_i=K_{2,i}/K_{2,i-1}$ satisfies 
\begin{align*}
M_1 \in \Coh_{\pi}^{\le 1}(\overline{X})[-1], \ 
M_2 \in \pi^{\ast}\Pic(\mathbb{P}^1), \
M_3 \in \Coh_{\pi}^{\le 1}(\overline{X})[-1]. 
\end{align*}
We note that the moduli stack of $\omega$-Gieseker
semistable sheaves $F \in \Coh_{\pi}(\overline{X})$
 with fixed $\cl_0(F) \in \Gamma_0$
is an algebraic stack of finite type. 
Therefore
 it is enough to
 show that, 
for fixed $v \in \Gamma$, ample divisor $\omega$
and $t \in \mathbb{R}_{>0}$, 
there is only a finite number of 
possibilities for the 
numbers and 
numerical classes of 
Harder-Narasimhan factors of $K_1$, $K_3[1]$, $M_1[1]$
and $M_3[1]$.
(See the proof of~\cite[Lemma~3.2]{Tcurve3}.)

For simplicity we only show the
above finiteness
for $K_1$ and $M_1[1]$. 
The other cases are similarly discussed. 
We take Harder-Narasimhan filtrations 
of $K_1$, $M_1[1] \in \Coh_{\pi}(\overline{X})$ 
with respect to $\omega$-Gieseker stability, 
\begin{align*}
0&=A_0 \subset A_1 \subset A_2 \subset \cdots \subset A_{k}=K_1, \\
0&=B_0 \subset B_1 \subset B_2 \subset \cdots \subset B_{l}=M_1[1].
\end{align*}
We set $C_i =A_i/A_{i-1}$ and $D_i=B_i/B_{i-1}$, 
and write
 \begin{align*}
\cl_0(C_i)=(r_i, \beta_i, n_i), \ 
\cl_0(D_i)=(0, \beta_i', n_i').
\end{align*}
Since $0<-\omega \cdot \beta_i \le -\omega \cdot \beta$
and $0<\omega \cdot \beta_i' \le -\omega \cdot \beta$
for $i \ge 2$, the numbers $k$ and $l$ are bounded. 
Moreover since $\beta_i' \ge 0$, there is only 
a finite number of possibilities for $\beta_i'$. 

By the $Z_{t\omega}$-semistability of $E$, 
we have $\arg Z_{t\omega}(A_i) \le \pi/2$, or equivalently 
\begin{align*}
\sum_{j=1}^{i} \left( n_j-\frac{1}{2}r_j \omega^2 \right)  \ge 0, 
\end{align*}
for all $i$. 
Hence by the result of 
Lemma~\ref{finset} below, 
both of  
$r_1 + \cdots +r_i$ 
and $n_1 + \cdots +n_i$ are bounded. 
By the induction on $i$, we conclude that 
$r_i$ and $n_i$ are also bounded.
 Then noting that 
$0 \le -\omega \cdot \beta_i \le -\omega \cdot \beta$, 
Lemma~\ref{lem:Bog} implies that $\beta_i^2$ is bounded, 
hence the Hodge index theorem implies that
there is only a finite number of possibilities for $\beta_i$. 

It remains to show the boundedness of $n_i'$. 
Again using the $Z_{t\omega}$-semistability of $E$, we have 
\begin{align*}
\sum_{j=1}^{k} \left( n_j-\frac{1}{2}r_j \omega^2 \right)
 -\sum_{j=1}^{i}n_j' \ge 0,
\end{align*}
for all $i$. 
Since $k$, $l$, $r_j$ and $n_j$ are bounded, 
we see that all $n_i'$ are bounded above
for all $i$. 
On the other hand 
an argument of~\cite[Lemma~3.2]{Tcurve1} shows that
$M_1[1]$ is written as 
$\oO_{Z}$ for a subscheme $Z \subset \overline{X}$
with $\dim Z\le 1$.
Therefore $\ch_3(M_1[1])=\sum_{j=1}^{l}n_j'$
is bounded below by~\cite[Lemma~3.10]{Tolim}. 
Hence the boundedness of $n_i'$ follows. 
\end{proof}

\subsection{Invariants $L(\beta, n)$}\label{subsec:DT=L}

Let $L(\beta, n) \in \mathbb{Q}$ be 
the invariant discussed in Subsection~\ref{subsec:product}. 
Here we recall the definition of $L(\beta, n)$
in~\cite[Definition~4.1]{Tolim2}, and 
compare it with the invariant $\DT_{t\omega}^{\chi}(r, \beta, n)$. 

For $v=(R, r, \beta, n) \in \Gamma$
with $R\le 1$, 
let 
\begin{align*}
\mM_{\rm{lim}}(v) 
\end{align*}
be the moduli stack 
of $\mu_{i\omega}$-limit 
semistable objects $E\in \aA(r)$ with 
$\cl(E)=v$. 
(cf.~Definition~\ref{def:lim}.)
By Lemma~\ref{rmk:lim} and~\cite[Proposition~3.17]{Tolim2}, 
the stack $\mM_{\rm{lim}}(v)$ is an 
algebraic stack of finite type over $\mathbb{C}$. 
Hence we
 can define the element, 
\begin{align*}
\delta_{\rm{lim}}(v)
\cneq [\mM_{\rm{lim}}(v) \hookrightarrow \oO bj \aA_{\omega}] 
\in \hH(\aA_{\omega}), 
\end{align*}
and 
\begin{align*}
\epsilon_{\rm{lim}}(v) \cneq 
\sum_{\begin{subarray}{c}l\ge 1,
v_1, \cdots, v_l \in \Gamma, \\
1\le e \le l, 
v_i=(0, 0, \beta_i, 0) , i\neq e, \\
 v_1 + \cdots +v_l=v. \\
\end{subarray}}
\frac{(-1)^{l-1}}{l}\delta_{\rm{lim}}(v_1) 
\ast \cdots \ast \delta_{\rm{lim}}(v_l). 
\end{align*}
Then $L(\beta, n) \in \mathbb{Q}$ is defined by 
\begin{align}\label{defi:Lbn}
L(\beta, n) \cneq 
\lim_{q^{1/2} \to 0} (q-1) 
P_q(\epsilon_{\rm{lim}}(1, 0, -\beta, -n)). 
\end{align}
\begin{rmk}
We note that $L(\beta, n)$ is defined 
in the Hall algebra of $\aA_{1/2}^p$
in the notation of~\cite{Tolim2}. 
However all the elements defining 
$\epsilon_{\rm{lim}}(v)$
are contained in the Hall algebra of $\aA(0)^{\dag}$, 
and since $\aA(0)^{\dag}$ is an extension closed subcategory 
of $\aA_{1/2}^p$, the resulting invariant 
$L(\beta, n)$ 
coincides with the one 
defined in the Hall algebra of $\aA_{1/2}^p$. 
(The notation in this remark will be recalled in 
Subsection~\ref{subsec:mulim}.)
\end{rmk}
\begin{rmk}
Noting Lemma~\ref{rmk:lim}, it is easy 
to check that the invariant
$L(\beta, n)$
coincides with $L_{n, \beta}^{eu}$
introduced in~\cite[Definition~4.1]{Tolim2}.
\end{rmk}
We have the following proposition:
\begin{prop}\label{prop:DT=L}
(i) For $t\gg 0$, we have 
\begin{align*}
\DT_{t\omega}^{\chi}(r, \beta, n)=
L(\beta, n). 
\end{align*}
(ii) For $0<t\ll 1$ and $n\neq 0$, we have 
\begin{align*}
\DT_{t\omega}^{\chi}(r, \beta, n)=0. 
\end{align*}
\end{prop}
\begin{proof}
The result of (ii) is obvious from Proposition~\ref{prop:m2}, and 
we prove (i) below. 
Let us take $v=(1, -r, -\beta, -n) \in \Gamma$. 
By Lemma~\ref{lem:Mr} and (\ref{defi:Lbn}), 
 it follows that 
\begin{align}\label{LL1}
\lim_{q^{1/2} \to 1}
(q-1)P_q(\epsilon_{\rm{lim}}(1, -r, -\beta, -n))
=L(\beta, n).
\end{align}
Suppose that 
$\delta_{t\omega}(v_1) \ast \cdots \ast \delta_{t\omega}(v_l)$
appears as a non-zero term of (\ref{sum:eps}).
Then there is $1\le e \le l$ such that 
$\rank(v_i)=0$ for $i\neq e$ and $\rank(v_e)=1$.  
By Proposition~\ref{prop:m1}, 
we can take $t'>0$ so that 
each $\delta_{t\omega}(v_i)$
coincides with 
$\delta_{\rm{lim}}(v_i)$
for $t>t'$.  
Also note that if $v_i=(0, r_i, \beta_i, n_i)$
satisfies $\delta_{\rm{lim}}(v_i) \neq 0$
and $\arg Z_{t\omega}(v_i)=\pi/2$, 
then we have $r_i=n_i=0$. 
Therefore we have
\begin{align}\label{LL2}
\epsilon_{t\omega}(1, -r, -\beta, -n)
=\epsilon_{\rm{lim}}(1, -r, -\beta, -n), 
\end{align}
for $t\gg 0$. 
Then 
the result of (i) follows 
from (\ref{LL2}) and (\ref{LL1}). 
\end{proof}

\subsection{Counting invariants of rank zero}

Here we define invariants counting 
rank zero objects
in $\aA_{\omega}$ or $\aA_{\omega}[1]$, 
and study their property. 
We set $C(\bB_{\omega})$ 
as follows, 
\begin{align}\label{def:CB}
C(\bB_{\omega}) \cneq 
\Imm \left (\cl_0 \colon \bB_{\omega} \to \Gamma_{0} \right). 
\end{align}
\begin{defi}\label{defi:Ninv}
For $v \in \Gamma_0$, 
we define the invariant
$N(v) \in \mathbb{Q}$
as follows. 
\begin{itemize}
\item
If $v\in C(\bB_{\omega})$, then we define 
\begin{align}\label{inv:N}
N(v) \cneq 
\lim_{q^{1/2}\to 1}(q -1)P_q(\epsilon_{\omega}(0, v)).
\end{align}
\item If $-v \in C(\bB_{\omega})$, then we define
$N(v)\cneq N(-v)$.  
\item If $\pm v \notin C(\bB_{\omega})$, then 
we define $N(v)=0$. 
\end{itemize}
\end{defi}
\begin{rmk}
By Remark~\ref{rmk:BA}, 
the invariant (\ref{inv:N}) 
is also interpreted as a counting invariant of 
$Z_{\omega, 0}$-semistable 
objects $E\in \bB_{\omega}$ with $\cl_0(E)=v$. 
We also note that similar invariants 
on a K3 surface is already constructed
and studied in~\cite{Tst3}. 
\end{rmk}
\begin{rmk}
By comparing with~\cite[Definition~4.1]{Tolim2}, 
the invariant of the form $N(0, \beta, n)$
in Definition~\ref{defi:Ninv}
coincides with the one 
which appeared in the formula (\ref{PNL}). 
\end{rmk}

In defining (\ref{inv:N}),
we need to choose an ample divisor
 $\omega$. 
However 
it will turn out that $N(v)$
does not depend on a choice of $\omega$. 
This fact follows from the 
same arguments of~\cite[Theorem~6.24]{Joy4},
~\cite[Proposition-Definition~5.7]{Tcurve1} and
\cite[Theorem~1.2]{Tst3}.
Below we explain this by introducing more 
general invariants counting 
Bridgeland semistable objects in $\dD_0$, 
not necessary of the form $(Z_{\omega, 0}, \bB_{\omega})$. 

First we discuss the space of stability conditions
on $\dD_0$. 
Recall that
in Lemma~\ref{lem:Bstab},
 we constructed stability conditions 
$(Z_{t\omega, 0}, \bB_{\omega}) \in \Stab_{\Gamma_{0}}(\dD_0)$. 
These stability 
conditions are contained in a
same connected component, which 
we denote by 
\begin{align*}
\Stab_{\Gamma_0}^{\circ}(\dD_0) \subset \Stab_{\Gamma_0}(\dD_0).
\end{align*}
Next 
let $\Stab(S)$ be the space of 
stability conditions on $D^b \Coh(S)$. 
In~\cite{Brs2}, Bridgeland
describes a certain connected component 
of $\Stab(S)$, 
which we denote by 
\begin{align}\label{connB}
\Stab^{\circ}(S) \subset
\Stab(S). 
\end{align}
The space of stability conditions on $\dD_0$
and $D^b \Coh(S)$ are closely related. 
In fact, we have the following comparison result. 
\begin{thm}\label{thm:isomSS}
There is an isomorphism, 
\begin{align}\label{isom:stab}
\psi \colon 
\Stab_{\Gamma_{0}}^{\circ}(\dD_0) \stackrel{\sim}{\to}
\Stab^{\circ}(S). 
\end{align}
\end{thm}
\begin{proof}
The result is essentially proved in~\cite[Theorem~6.5, Lemma~5.3]{Tst2}. 
\end{proof}
In the paper~\cite{Tst3}, 
the author constructed
counting invariants 
of semistable objects 
in $D^b \Coh(S)$, 
motivated by Joyce's conjecture~\cite[Conjecture~6.25]{Joy4}. 
The construction itself relies on 
a choice of a stability condition in 
$\Stab^{\circ}(S)$, however
it turned out that the invariant does not 
depend on a choice of a stability condition. 
Although the categories $D^b \Coh(S)$ and $\dD_0$
are not equivalent, the 
arguments used for $D^b \Coh(S)$ in~\cite{Tst3}
is applied without any major modifications. 
A rough story of the arguments in~\cite{Tst3}
applied for $\dD_0$
is as follows:
for any element 
\begin{align*}
\sigma 
=(Z, \aA)\in \Stab_{\Gamma_{0}}^{\circ}(\dD_0),
\end{align*}
we can define the invariant generalizing $N(v)$, 
\begin{align}\label{invsigma}
N_{\sigma}(v) \in \mathbb{Q}, 
\end{align}
counting $Z$-semistable objects
$E \in \aA$ or $\aA[1]$, satisfying 
$\cl_0(E)=v$. 
Namely we can similarly define 
the algebra $(H(\aA), \ast)$,
by replacing the stack 
$\oO bj(\aA_{\omega})$ by 
$\oO bj(\aA)$, the stack 
of objects $E\in \aA$, 
in Definition~\ref{def:HA}. 
The stack of $Z$-semistable 
objects $E\in \aA$
with $\cl_0(E)=v$ defines 
an element, 
\begin{align*}
\delta_{\sigma}(v)
\cneq [\mM_{\sigma}(v) \hookrightarrow \oO bj(\aA)]
\in \hH(\aA). 
\end{align*}
Also the element $\epsilon_{\sigma}(v) \in \hH(\aA)$
can be defined in a way 
similar to (\ref{sum:eps}), 
by replacing $\delta_{t\omega}(0, v)$ by $\delta_{\sigma}(v)$. 
The invariant (\ref{invsigma})
is defined by replacing $\bB_{\omega}$, 
$\epsilon_{t\omega}(0, v)$ by 
$\aA$, $\epsilon_{\sigma}(v)$
respectively 
in Definition~\ref{defi:Ninv}.
All the technical details
in proving the existence of 
the invariant (\ref{invsigma}), 
e.g. the existence of 
moduli stacks, finiteness, etc.
 follow from 
the same arguments in~\cite{Tst3}.
Also if $\sigma$ is a stability condition 
constructed in Lemma~\ref{lem:Bstab}, 
then the invariant (\ref{invsigma}) 
coincides with the invariant 
defined in Definition~\ref{defi:Ninv}. 
We have the following result. 
\begin{thm}\label{sigma:depend}
The invariant $N_{\sigma}(v)$ does 
not depend on a choice of 
$\sigma \in \Stab_{\Gamma_{0}}^{\circ}(\dD_0)$. 
In particular, the invariant $N(v)$ is independent of $\omega$. 
\end{thm}
\begin{proof}
The proof
is same as in~\cite[Proposition-Definition~5.7]{Tcurve1}, 
\cite[Theorem~1.2]{Tst3}, so we 
just give a sketch of the proof. 
We take two elements $\sigma_i \in \Stab_{\Gamma_{\bullet}}^{\circ}(\dD_0)$
for $i=0, 1$. 
We may assume that 
$\sigma_1$ is sufficiently close 
to $\sigma_0$. 
Then we can essentially apply 
the wall-crossing formula 
in an abelian category~\cite[Theorem~6.28]{Joy4}, 
which describes
$N_{\sigma_1}(v)$
in terms of $N_{\sigma_0}(v)$. 
The wall-crossing formula is described as 
\begin{align*}
N_{\sigma_{1}}(v)=N_{\sigma_{0}}(v)+\sum_{v_1 +v_2=v}
a_{v_1, v_2} \chi(v_1, v_2)N_{\sigma_{0}}(v_1)N_{\sigma_0}(v_2)
+ \cdots,
\end{align*}
for some $a_{v_1, v_2} \in \mathbb{Q}$
and $\chi$ is the Euler pairing on $\Gamma_0$
defined in Subsection~\ref{subsec:gen:bi}.
All the other terms are also given by multiplications of
$\chi(v_i, v_j)$, $N_{\sigma_{0}}(v_i)$ and
some complicated coefficients. 
As we observed in Subsection~\ref{subsec:gen:bi}, 
 we have $\chi(v, v')=0$ for $v, v' \in \Gamma_0$, 
all the error terms vanish,
hence we have $N_{\sigma_{1}}(v)=N_{\sigma_{0}}(v)$. 
\end{proof}

\subsection{Sheaf counting invariants}\label{subsec:sheafcount}
As we discussed in the introduction, 
we are interested in the invariants 
counting semistable sheaves 
on the open 
Calabi-Yau 3-fold $X=S\times \mathbb{C}$. 
Let $\Coh_{\pi}(X)$ be the 
category given in (\ref{Cohsub}). 
The stack 
\begin{align}\label{objcoh}
\cC oh_{\pi}(X), 
\end{align}
which parameterizes objects 
in $\Coh_{\pi}(X)$ is known to 
be an 
algebraic stack locally of finite type 
over $\mathbb{C}$.
By just replacing (\ref{stack:objA}) by 
(\ref{objcoh}) in Definition~\ref{def:HA}, 
we can define 
the $\mathbb{Q}$-vector space, 
\begin{align*}
\hH(\Coh_{\pi}(X)), 
\end{align*}
with a $\ast$-product
similar to (\ref{ast:prod}). 
Also for each $v\in \Gamma_0$, 
let 
\begin{align}\label{Obj:coh}
 \mM_{\omega, X}(v) \subset \cC oh_{\pi}(X),
\end{align}
be the substack of
$\omega$-Gieseker semistable 
sheaves $E \in \Coh_{\pi}(X)$
satisfying 
\begin{align*}
v(E)=v. 
\end{align*}
Here $v(E)$ is the Mukai vector of 
$E$, defined by (\ref{DMukai}). 
The stack (\ref{Obj:coh})
is known to be an algebraic stack 
of finite type over $\mathbb{C}$.  
\begin{rmk}
Although we 
have constructed invariants (\ref{inv:N})
using the Chern character, 
the Mukai vector (not Chern character)
is used in this subsection. 
The reason is that Mukai vector 
is useful in describing the automorphic 
property of the invariants. 
(See Theorem~\ref{thm:aut} below.)
\end{rmk}
 The substack (\ref{Obj:coh})
defines an element, 
\begin{align*}
\delta_{\omega, X}(v) \cneq 
[\mM_{\omega, X}(v) \hookrightarrow \cC oh_{\pi}(X)]
\in \hH(\Coh_{\pi}(X)),
\end{align*}
and its `logarithm' defined by
\begin{align}\label{log:eps}
\epsilon_{\omega, X}(v) \cneq
\sum_{\begin{subarray}{c}l\ge 1, v_1 + \cdots +v_l=v,
v_i \in \Gamma_0 \\
\overline{\chi}_{\omega, v_i}(m)
=\overline{\chi}_{\omega, v}(m)
\end{subarray}}
\frac{(-1)^{l-1}}{l}\delta_{\omega, X}(v_1) 
\ast \cdots \ast \delta_{\omega, X}(v_l).
\end{align}
Here $\overline{\chi}_{\omega, v}(m)$
is the reduced Hilbert polynomial (\ref{red:hilb}).
Similarly to Lemma~\ref{lem:finsum}, 
the sum (\ref{log:eps})
is a finite sum and $\epsilon_{\omega, X}(v)$
is well-defined. The argument is standard, 
so we omit the detail. 
We set $C(X)$ as follows, 
\begin{align}\label{C(X)}
C(X) \cneq \Imm (v \colon \Coh_{\pi}(X) \to \Gamma_0).
\end{align}
We define the following sheaf counting invariant. 
\begin{defi}\label{defi:Jv}
For $v \in \Gamma_0$, we define 
the invariant 
$J(v) \in \mathbb{Q}$
as follows. 
\begin{itemize}
\item If $v \in C(X)$, we define 
\begin{align*}
J(v) \cneq 
\lim_{q^{1/2}\to 1}(q -1)P_q(\epsilon_{\omega, X}(v)). 
\end{align*}
\item If $-v \in C(X)$, we define 
$J(v)\cneq J(-v)$. 
\item If $\pm v \notin C(X)$, we define $J(v)=0$. 
\end{itemize}
\end{defi}
Here the map 
\begin{align*}
P_q \colon \hH(\Coh_{\pi}(X)) \to \mathbb{Q}(q^{1/2}),
\end{align*}
is defined similarly to (\ref{map:P}). 
We note that a similar invariant 
on $\Coh(S)$, (not $\Coh_{\pi}(X)$,)
is introduced and studied in~\cite{Joy4}. 
Similarly to~\cite[Theorem~6.24]{Joy4},
the invariant $J(v)$ does not depend on 
a choice of $\omega$. 
(Also see the proof of Theorem~\ref{sigma:depend}.)

In defining $J(v)$, we can also take 
$\omega$ to be an $\mathbb{R}$-ample divisor,
and show that it does not depend on 
$\omega$. If $v\in \Gamma_0$ is primitive
and $\omega$ is in a general position
in the ample cone, 
then 
 the moduli stack
$\mM_{\omega, X}(v)$ is written as
\begin{align*}
\mM_{\omega, X}(v)=
[M_{\omega, S}(v)/\mathbb{C}^{\ast}] \times \mathbb{C},
\end{align*}
where $M_{\omega, S}(v)$ is the
moduli space 
of $\omega$-Gieseker 
stable sheaves $E$ on $S$
satisfying $v(E)=v$, 
and $\mathbb{C}^{\ast}$ acts on $M_{\omega, S}(v)$
trivially. 
The space
$M_{\omega, S}(v)$ is known to 
be
a holomorphic symplectic manifold of 
dimension $(v, v)+2$, 
and deformation equivalent to the Hilbert scheme of 
$(v, v)/2 +1$-points on $S$. 
(cf.~\cite[Theorem~0.2]{Yoshi2}, \cite[Theorem~5.151]{KY}.)
Therefore we have 
\begin{align}\notag
J(v)&= \chi(M_{\omega, S}(v))  \\
\label{J=H}
&=\chi(\Hilb^{(v, v)/2+1}(S)), 
\end{align}
where $\Hilb^n(S)$ is the Hilbert 
scheme of $n$-points in $S$. 
The RHS of (\ref{J=H}) is 
given by the G$\ddot{\rm{o}}$ttsche's formula~\cite{Got},
\begin{align}\label{Gottscheform}
\sum_{n\ge 0}\chi(\Hilb^n(S))q^n 
=\prod_{n\ge 1}\frac{1}{(1-q^n)^{24}}. 
\end{align}
For a general $v\in \Gamma_0$, 
we will propose 
in Subsection~\ref{subsec:mult3}
a conjectural relationship between 
$J(v)$ and $\chi(\Hilb^n(S))$
in terms of a multiple cover formula. 

\subsection{Comparison of $N(v)$ and $J(v)$}

In~\cite[Theorem~6.6]{Tst3}, we discussed a relationship 
between invariants counting semistable 
objects in $D^b \Coh(S)$ and invariants 
counting Gieseker semistable sheaves in $\Coh(S)$. 
A similar result is also obtained for counting 
invariants
$N(v) \in \mathbb{Q}$ and 
counting invariants of Gieseker-semistable sheaves in 
$\Coh_{\pi}(\overline{X})$. 
Similarly to Definition~\ref{defi:Jv}, 
we can define 
the invariant, 
\begin{align}\label{Jv}
\overline{J}(v) \in \mathbb{Q}, 
\end{align}
counting $\omega$-Gieseker
semistable sheaves $E\in \Coh_{\pi}(\overline{X})$
with $v(E)=v \in \Gamma_0$.
 Namely we just replace 
$X$ by $\overline{X}$ for all the ingredients 
in defining the invariant $J(v)$ in Definition~\ref{defi:Jv}. 
By the arguments similar to the proofs of~\cite[Theorem~6.24]{Joy4}
and Theorem~\ref{sigma:depend},
we can show that $\overline{J}(v)$ does not depend
on $\omega$. We have the following result. 
\begin{thm}\label{bJ=N}
For any $v\in \Gamma_0$, we have 
\begin{align}\label{J=N2}
\overline{J}(v\sqrt{\td_S})=N(v).
\end{align}
\end{thm}
\begin{proof}
The proof is exactly same as in~\cite[Theorem~6.6]{Tst3}, 
so we just give a sketch of the proof.  
For an element $v=(r, \beta, n)$, 
suppose that $v\in C(X)$
where $C(X)$ is defined in (\ref{C(X)}).
If $v\in C(X)$, then
we can reduce the problem to 
the case of $\omega \cdot \beta>0$ or 
$r=\beta=0$. 
(See~\cite[Lemma~6.3]{Tst3} and the proof of~\cite[Theorem~6.6]{Tst3}.)  
In these cases, 
the same arguments as 
in~\cite[Proposition~6.4]{Tst3}, \cite[Lemma~6.5]{Tst3} 
show that an object $E\in \bB_{\omega}$ is 
$Z_{t\omega, 0}$-semistable
with $\cl_0(E)=v$ if and only if 
$E$ is an $\omega$-Gieseker 
semistable sheaf with $v(E)=v\sqrt{\td_S}$. 
This fact immediately implies the equality (\ref{J=N2}). 
A similar argument in the proof of~\cite[Theorem~6.6]{Tst3}
also proves the case of $v \in -C(X)$
and $v \notin \pm C(X)$. We 
leave the readers to check the detail. 
\end{proof}

Next we compare invariants $\overline{J}(v)$ with $J(v)$. 
By replacing $X$ by $\overline{X}$ in 
Subsection~\ref{subsec:sheafcount}, 
we have the element, 
\begin{align*}
\delta_{\omega, \overline{X}}(v)=[\mM_{\omega, \overline{X}}(v) \hookrightarrow \cC oh_{\pi}(\overline{X})] \in \hH(\Coh_{\pi}(\overline{X})),
\end{align*} 
where $\mM_{\omega, \overline{X}}(v)$ is the moduli stack 
of $\omega$-Gieseker semistable sheaves $E \in \Coh_{\pi}(\overline{X})$
with $\cl_0(E)=v$. 
For an open or closed subscheme $Z \subset \overline{X}$, we denote by 
$\mM_{\omega, Z}(v) \subset \mM_{\omega, \overline{X}}(v)$
the locus of $E \in \Coh_{\pi}(\overline{X})$ 
whose  
support is contained in $Z$. 
We set 
\begin{align*}
\delta_{\omega, Z}(v) \cneq [\mM_{\omega, Z}(v) \hookrightarrow \cC oh_{\pi}(\overline{X})] \in \hH(\Coh_{\pi}(\overline{X})),
\end{align*}
and define $\epsilon_{\omega, Z}(v) \in \hH(\Coh_{\pi}(\overline{X}))$
just by replacing $\delta_{\omega, X}(v_i)$ by 
$\delta_{\omega, Z}(v_i)$ in (\ref{log:eps}). 
Also for $p\in \mathbb{P}^1$, we set 
\begin{align*}
U_p \cneq \overline{X} \setminus X_p. 
\end{align*}
We have the following lemma. 
\begin{lem}\label{lem:eux}
We have 
\begin{align}\label{e=UX}
\epsilon_{\omega, \overline{X}}(v)=\epsilon_{\omega, U_p}(v) +\epsilon_{\omega, X_p}(v).
\end{align}
\end{lem}
\begin{proof} 
In order to simplify the notation, we omit
$\omega$ and write  
$\delta_{\omega, \overline{X}}(v)$
as $\delta_{\overline{X}}(v)$, etc. 
First we note that 
\begin{align*}
\delta_{\overline{X}}(v)=
\sum_{\begin{subarray}{c}v_1, v_2 \in \Gamma_0, v_1 +v_2=v, \\
\overline{\chi}_{\omega, v_1}(m)=\overline{\chi}_{\omega, v_2}(m)
\end{subarray}}
\delta_{U_p}(v_1) \ast \delta_{X_p}(v_2),
\end{align*}
since any object $E\in \Coh_{\pi}(\overline{X})$
decomposes as $E_1 \oplus E_2$ 
with $E_1$ supported on $U_p$ and $E_2$ supported on $X_p$. 
Since $\delta_{U_p}(v_1) \ast \delta_{X_p}(v_2)
=\delta_{X_p}(v_2) \ast \delta_{U_p}(v_1)$, we have 
\begin{align}\label{epsilon=prod}
&\epsilon_{\overline{X}}(v)
= \\
&\notag
\sum_{\begin{subarray}{c}l \ge 1, 
v_i, v_i' \in \Gamma_0, \\
\overline{\chi}_{\omega, v_i}(m)=\overline{\chi}_{\omega, v}(m), \\
\overline{\chi}_{\omega, v_i'}(m)=\overline{\chi}_{\omega, v}(m), \\
v_{1}+\cdots +v_{l}+ 
v_{1}'+ \cdots +v_{l}'=v
\end{subarray}}
\frac{(-1)^{l-1}}{l}
\delta_{U_p}(v_{1}) \ast \cdots \ast
\delta_{U_{p}}(v_{l}) 
\ast \delta_{X_p}(v_{1}') \ast \cdots \ast \delta_{X_p}(v_{l}').
\end{align}
Take $v_1, \cdots, v_a \in \Gamma_0$ and $v_1', \cdots, v_b' \in \Gamma_0$
with $v_i \neq 0$, $v_j' \neq 0$ for any $i$ and $j$
and satisfy
\begin{align*}
&\overline{\chi}_{\omega, v_i}(m)=\overline{\chi}_{\omega, v_i'}(m)
=\overline{\chi}_{\omega, v}(m), \\
&v_1 +\cdots +v_a + v_1' + \cdots +v_{b}'=v.
\end{align*}
If $a \ge 1$, $b\ge 1$
and $a \ge b$,  then 
the coefficient of $\delta_{U_p}(v_1) \ast \cdots \ast
\delta_{U_p}(v_a) \ast
\delta_{X_p}(v_1') \ast \cdots \ast \delta_{X_p}(v_b')$ in 
(\ref{epsilon=prod}) is 
\begin{align*}
&\sum_{l=a}^{a+b}\frac{(-1)^{l-1}}{l}\binom{l}{a} \binom{a}{a+b-l} \\
&= \frac{(-1)^{a-1}}{b}\sum_{m=0}^{b}
(-1)^m \binom{b}{m} \binom{m+a-1}{b-1} \\
&=0. 
\end{align*}
The last equality follows by 
taking the differentials
of $x^{a-1}(x-1)^{b}$ by $(b-1)$-times,
and substituting $x=1$.  
We can similarly show the vanishing 
of the coefficient when $a\le b$. Hence 
(\ref{e=UX}) follows. 
\end{proof}
We have the following lemma. 
\begin{lem}\label{Jlocarize}
We have $\overline{J}(v)=2J(v)$. 
\end{lem}
\begin{proof}
The proof essentially follows from 
$\mathbb{C}^{\ast}$-localization
for the invariants $\overline{J}(v)$ and 
$J(v)$. However a general localization formula 
for invariants defined via Hall algebra is 
not yet established. 
Here we give a proof
assuming the terminology of~\cite{JS}. 

As in the proof of Lemma~\ref{lem:eux}, we omit 
$\omega$ in the notation. 
Let $M_{\overline{X}}(v)$ be the coarse moduli 
scheme of $\omega$-Gieseker semistable sheaves 
$E \in \Coh_{\pi}(\overline{X})$ with $v(E)=v$.
There is a natural morphism,
\begin{align*}
\eta \colon \mM_{\overline{X}}(v) \to M_{\overline{X}}(v).
\end{align*}
sending
an $\omega$-Gieseker semistable 
sheaf $E$ to $\oplus_{i=1}^{N} F_i$, 
where $F_1, \cdots, F_N$ are
$\omega$-Gieseker stable factors of $E$. 
As in~\cite[Equation~(5.9)]{JS}, 
the invariant $\overline{J}(v)$ can be also expressed as 
\begin{align*}
\overline{J}(v)&=\chi(M_{\overline{X}}(v), \alpha) \\
&\cneq \sum_{m \in \mathbb{Z}}m \cdot \chi(\alpha^{-1}(m)), 
\end{align*}
for some constructible function $\alpha$ on 
$M_{\overline{X}}(v)$.
In the notation of~\cite[Equation~(5.9)]{JS}, 
the function $\alpha$ is given by  
\begin{align}\label{alpha:const}
\alpha=\mathrm{CF}^{\rm{na}}
(\eta)[\Pi_{\rm{CF}}\circ 
\overline{\Pi}^{\chi, \mathbb{Q}}_{\mM_{\overline{X}}(v)}
(\epsilon_{\overline{X}}(v))]. 
\end{align}
 Let
\begin{align*}
M_{\overline{X}}^{\dag}(v) \subset M_{\overline{X}}(v), 
\end{align*}
be the closed subscheme corresponding to 
semistable sheaves $E$ such that 
$\Supp(E) \subset X_p$ for some $p \in \mathbb{P}^1$. 
Since we have the formula (\ref{e=UX}) for any $p\in \mathbb{P}^1$, 
the construction of $\alpha$ in
(\ref{alpha:const}) easily implies that $\alpha$ is zero outside
$M_{\overline{X}}^{\dag}(v)$. 
On the other hand, 
we have the natural isomorphism, 
\begin{align}\label{nat:isom}
M_{\overline{X}}^{\dag}(v) \cong M_{X_0}(v) \times \mathbb{P}^1, 
\end{align}
where $M_{X_0}(v) \subset M_{\overline{X}}(v)$ is 
the closed subscheme corresponding to 
sheaves $E$ supported on $X_0$. 
Under the above isomorphism, 
we have 
\begin{align}\label{eta:equiv}
\alpha|_{M_{X_0}(v) \times \{p\}}
=\mathrm{CF}^{\rm{na}}
(\eta)[\Pi_{\rm{CF}}\circ 
\overline{\Pi}^{\chi, \mathbb{Q}}_{\mM_{\overline{X}}(v)}
(\epsilon_{X_p}(v))],
\end{align}
in the notation of~\cite[Equation~(5.9)]{JS}
by Lemma~\ref{lem:eux}. 

Let $\Coh_{X_p}(\overline{X}) \subset \Coh(\overline{X})$
be the subcategory consisting of sheaves 
supported on $X_p$. 
For $p, q \in \mathbb{P}^1$, 
choose $g \in \Aut(\mathbb{P}^1)$ such that 
$g(p)=q$. 
Then $g$ induces an equivalence
\begin{align*}
g_{\ast} \colon 
\Coh_{X_p}(\overline{X}) \stackrel{\sim}{\to} \Coh_{X_q}(\overline{X}),
\end{align*} 
and the induced isomorphism 
between the Hall algebras 
\begin{align}\label{isom:Hall:pq}
g_{\ast} \colon 
\hH(\Coh_{X_p}(\overline{X}))
 \stackrel{\sim}{\to} \hH(\Coh_{X_q}(\overline{X})).
\end{align}
The element $\epsilon_{X_p}(v)$ is 
regarded as an element of 
$\hH(\Coh_{X_p}(\overline{X}))$, 
which is mapped to 
$\epsilon_{X_q}(v)$ by 
the isomorphism (\ref{isom:Hall:pq}).
Therefore by (\ref{eta:equiv}), 
 we have $\alpha(x, p)=\alpha(x, q)$
for $x\in M_{X_0}(v)$
under the isomorphism (\ref{nat:isom}). 
Hence we have 
\begin{align*}
\overline{J}(v)=\chi(\mathbb{P}^1) \cdot 
\chi(M_{X_0}(v) \times \{0\}, \alpha). 
\end{align*}
Similarly we have 
\begin{align*}
J(v)=\chi(\mathbb{C}) \cdot \chi(M_{X_0}(v) \times \{0\}, \alpha). 
\end{align*}
Since $\chi(\mathbb{P}^1)=2$ and $\chi(\mathbb{C})=1$, 
we obtain the result. 
\end{proof}

We have the following corollaries: 
\begin{cor}\label{prop:inv}
For any $v\in \Gamma_0$, we have the following equality, 
\begin{align}\notag
J(v\sqrt{\td_S})=\frac{1}{2}N(v).
\end{align}
\end{cor}
\begin{proof}
The result follows by combining Theorem~\ref{bJ=N} 
and Lemma~\ref{Jlocarize} below. 
\end{proof}

\begin{cor}\label{cor:N}
For $v=(r, \beta, n) \in \Gamma_0$
and an ample divisor $\omega$ on $S$, 
 suppose that 
$\beta \cdot \omega \neq 0$. 
If $N(v) \neq 0$, then we have 
\begin{align*}
\beta^2 +2(\beta \cdot \omega)^2 \ge 2r(r+n). 
\end{align*}
Moreover if $\beta \cdot \omega >0$
and $rn\ge 0$, then we have 
$\beta>0$. 
\end{cor}
\begin{proof}
If $N(v)\neq 0$, 
then Corollary~\ref{prop:inv}
implies that 
there is an $\omega$-Gieseker 
semistable sheaf $E$ on 
$\overline{X}$ with $\cl_0(E)=(r, \beta, n)$
or $\cl_0(E)=-(r, \beta, n)$.  
Then the first statement follows Lemma~\ref{lem:Bog}.
Suppose that $\beta \cdot \omega >0$, 
and $rn \ge 0$. 
Let $E_1, \cdots, E_k$ be 
$\omega$-Gieseker stable factors of 
$E$. If we write $\cl_0(E_i)=(r_i, \beta_i, n_i)$, 
then $\beta_i \cdot \omega>0$, $r_i n_i \ge 0$. 
Applying the inequality (\ref{ineq:Bog})
to each $E_i$, 
we see that $\beta_i^2 \ge -2$, hence 
$\beta_i>0$ for all $i$
 by the Riemann-Roch theorem. 
Since $\beta$ is a sum $\sum_{i}\beta_i$, 
we have $\beta>0$.  
\end{proof}

\subsection{Automorphic property of $J(v)$}\label{subsec:Aut}

In Subsection~\ref{subsec:sheafcount}, 
we defined the invariant $J(v) \in \mathbb{Q}$. 
The invariant $J(v)$ is a counting invariant of 
$\omega$-Gieseker semistable sheaves on the
open Calabi-Yau 3-fold 
$X=S\times \mathbb{C}$.
The purpose here is to 
observe
that $J(v)$ has a certain automorphic property
with respect to the group $G$, 
\begin{align}\label{Hodge:isom}
G \cneq O_{\mathrm{Hodge}}(\widetilde{H}(S, \mathbb{Z}), (\ast, \ast))
\end{align}
consisting of isometries of the 
Mukai lattice $(\widetilde{H}(S, \mathbb{Z}), (\ast, \ast))$
preserving the Hodge structure on it.  
(See Subsection~\ref{subsec:K3}.)
Note that any $g\in G$ induces an isometry 
of the lattice $(\Gamma_0, (\ast, \ast))$, 
since $g$ preserves the Hodge structure on 
$\widetilde{H}(S, \mathbb{Z})$. 
Note that, in the previous subsection, 
we also defined the invariant $\overline{J}(v) \in \mathbb{Q}$
as a counting 
invariant of $\omega$-Gieseker semistable sheaves 
on the compactification $\overline{X}=S\times \mathbb{P}^1$. 
Our strategy is to prove the automorphic property 
of $\overline{J}(v)$, and then use the result of Lemma~\ref{Jlocarize}.  

The automorphic property of the invariants
essentially  
follows by investigating the effect of 
the invariants under Fourier-Mukai transforms. 
For two K3 surfaces $S$, $S'$, let 
$\Phi$ be a derived equivalence, 
\begin{align*}
\Phi \colon D^b \Coh(S') \stackrel{\sim}{\to}
D^b \Coh(S). 
\end{align*}
Recall that, 
by Orlov's theorem~\cite{Or1}, 
 any such an equivalence is written 
as
\begin{align*}
\Phi(-) \cong \dR p_{2\ast}p_1^{\ast}(- \dotimes \eE), 
\end{align*}
for some object $\eE \in D^b \Coh(S' \times S)$, 
called the \textit{kernel} of $\Phi$. 
Here $p_1 \colon S' \times S \to S'$ and 
$p_2 \colon S' \times S \to S$ are projections. 
The equivalence $\Phi$ induces the Hodge isometry, 
\begin{align}\label{Phiast}
\Phi_{\ast} \colon \widetilde{H}(S', \mathbb{Z}) \stackrel{\sim}{\to} 
\widetilde{H}(S, \mathbb{Z}), 
\end{align}
given by 
\begin{align*}
\Phi_{\ast}(-) =p_{2\ast} p_1^{\ast}(- \cdot
\ch(\eE) \sqrt{\td_{S' \times S}}), 
\end{align*}
and we have the 
commutative diagram, (cf.~\cite[Theorem~4.9]{Mu2}, ~\cite[Proposition~3.5]{Or1},)
\begin{align}\label{diagram:DDHH}
\xymatrix{
D^b \Coh(S') \ar[r]^{\Phi}\ar[d]_{v} & D^b \Coh(S)\ar[d]^{v} \\
\widetilde{H}(S', \mathbb{Z}) \ar[r]^{\Phi_{\ast}} & 
\widetilde{H}(S, \mathbb{Z}). 
}
\end{align}
Also the equivalence $\Phi$ 
induces the isomorphism, 
\begin{align*}
\Phi_{\mathrm{St}} \colon 
\Stab(S') \stackrel{\sim}{\to} \Stab(S). 
\end{align*}
In order to distinguish 
the notation for $S$ and $S'$, 
we write $\dD_0$, $\Gamma_0$ and 
$\overline{J}(v)$ as 
$\dD_{0, S}$, $\Gamma_{0, S}$ and 
$\overline{J}_S(v)$ respectively. 
We have the following proposition. 
\begin{prop}\label{prop:equivalence}
In the above situation, suppose that 
$\Phi_{\mathrm{St}}$ takes 
the connected component 
$\Stab^{\circ}(S')$ to $\Stab^{\circ}(S)$. 
Then for $v\in \Gamma_{0, S'}$, we have 
\begin{align*}
\overline{J}_{S'}(v)=\overline{J}_{S}(\Phi_{\ast}v). 
\end{align*}
\end{prop}
\begin{proof}
Let $\eE$ be the kernel of $\Phi$, 
and $\overline{X}'\cneq S' \times \mathbb{P}^1$. 
The equivalence $\Phi$ extends to the equivalence, 
(cf.~\cite[Assertion~1.7]{OrA},)
\begin{align*}
\Phi^{\dag} \colon 
D^b \Coh(\overline{X}') \stackrel{\sim}{\to}
D^b \Coh(\overline{X}), 
\end{align*}
with kernel given by 
\begin{align*}
\eE \boxtimes \oO_{\Delta_{\mathbb{P}^1}} 
\in D^b \Coh(S' \times S \times \mathbb{P}^1 \times \mathbb{P}^1). 
\end{align*}
Here we have identified $\overline{X}' \times \overline{X}$
with $S' \times S \times \mathbb{P}^1 \times \mathbb{P}^1$. 
It is easy to see that 
$\Phi^{\dag}$ restricts to the 
equivalence between 
$\dD_{0, S'}$ and $\dD_{0, S}$. 
Also note that 
$\Phi_{\ast}$ in (\ref{Phiast}) 
restricts to the isomorphism 
between $\Gamma_{0, S'}$ and $\Gamma_{0, S}$, 
since $\Phi_{\ast}$ preserves the Hodge structures. 
Therefore by the 
diagram (\ref{diagram:DDHH}),
we have the commutative diagram, 
\begin{align}\label{diag:D}
\xymatrix{
\dD_{0, S'} \ar[r]^{\Phi^{\dag}}\ar[d]_{\cl_0 \sqrt{\td_{S'}}} & \dD_{0, S}
\ar[d]^{\cl_0 \sqrt{\td_S}} \\
\Gamma_{0, S'} \ar[r]^{\Phi_{\ast}} & \Gamma_{0, S}.
}
\end{align}
Also by the assumption
and Theorem~\ref{thm:isomSS}, 
the equivalence $\Phi^{\dag}$ induces the isomorphism, 
\begin{align*}
\Phi_{\mathrm{St}} \colon 
\Stab_{\Gamma_{0, S'}}^{\circ}(\dD_{0, S'}) 
\stackrel{\sim}{\to} 
\Stab_{\Gamma_{0, S}}^{\circ}(\dD_{0, S}). 
\end{align*}
Take 
$\sigma \in \Stab_{\Gamma_{0, S}}^{\circ}(\dD_{0, S})$ and 
 $\sigma' \in \Stab_{\Gamma_{0, S'}}^{\circ}(\dD_{0, S'})$. 
Then for $v \in \Gamma_{0, S'}$, we have  
\begin{align} \label{JN1}
\overline{J}_{S}(\Phi_{\ast}v) &
=N_{\sigma}(\Phi_{\ast}v\cdot \sqrt{\td_S}^{-1}) \\
\label{JN2}
&= N_{\Phi_{\mathrm{St}}\sigma'}(\Phi_{\ast}v\cdot \sqrt{\td_{S}}^{-1}) \\
\label{JN3}
&= N_{\sigma'}(v \cdot \sqrt{\td_{S'}}^{-1}) \\
\label{JN4}
&= \overline{J}_{S'}(v). 
\end{align}
Here (\ref{JN1})
and (\ref{JN4}) follow from Theorem~\ref{bJ=N}, 
(\ref{JN2}) follows from Theorem~\ref{sigma:depend} 
and (\ref{JN3}) follows from the commutative 
diagram (\ref{diag:D}). 
\end{proof}
Recall that we defined the group $G$
in (\ref{Hodge:isom}) to be the 
group of Hodge isometries of $\widetilde{H}(S, \mathbb{Z})$. 
We have the following corollary of Proposition~\ref{prop:equivalence}.
\begin{cor}\label{prop:aut}
For $v \in \Gamma_0$ and 
$g\in G$, we have 
\begin{align}\label{Jbar:aut}
\overline{J}(gv)=\overline{J}(v).
\end{align}
\end{cor}
\begin{proof}
For a K3 surface $S$, 
let $\Auteq^{\circ}(S)$ be the
group of autoequivalences of 
$D^b \Coh(S)$, 
preserving the connected component 
$\Stab^{\circ}(S)$.  
Then the group homomorphism 
\begin{align*}
\Auteq^{\circ}(S) \ni \Phi \mapsto 
\Phi_{\ast} \in G^{+},
\end{align*}
is surjective
by~\cite[Corollary~4.10]{HMS2}, \cite[Proposition~7.9]{HH}. 
Here $G^{+}$ is the index two 
subgroup of $G$, consisting of $g\in G$
preserving the orientation of
the positive definite four plane in 
 $\widetilde{H}(S, \mathbb{R})$. 
Therefore (\ref{Jbar:aut}) holds 
for $g\in G^{+}$ by Proposition~\ref{prop:equivalence}. 

Finally let $\iota \in G$ be the involution, 
\begin{align*}
\iota=\id_{H^0(S, \mathbb{Z})} \oplus (-\id_{H^2(S, \mathbb{Z})})
\oplus \id_{H^4(S, \mathbb{Z})}.
\end{align*}
The equality
(\ref{Jbar:aut}) for $g=\iota$
follows 
by applying the derived dual on $\dD$. 
(Also see Proposition~\ref{prop:bij} below.) 
Since $G/G^{+}$ is generated by $\iota$, 
we obtain the result. 
\end{proof}

By combining the above results, we have the following theorem. 
\begin{thm}\label{thm:aut}
For any $g\in G$, we have 
\begin{align*}
J(gv)=J(v). 
\end{align*}
\end{thm}
\begin{proof}
The result follows by combining Corollary~\ref{prop:aut}
and Lemma~\ref{Jlocarize}. 
\end{proof}

\subsection{Invariants on $\aA_{\omega}(1/2)$}\label{subsec:invA12}
In this subsection, we use
notation introduced in 
Subsections~\ref{subsec:abelian12}, 
\ref{subsec:weakA12}, \ref{subsec:semiA12}. 
Recall that we constructed weak stability 
conditions $(\widehat{Z}_{\omega, \theta}, \aA_{\omega}(1/2))$
in Lemma~\ref{lem:weakA12}. 
Similarly to Subsection~\ref{subsec:invvia}, 
we can construct counting invariants of 
$\widehat{Z}_{\omega, \theta}$-semistable objects in $\aA_{\omega}(1/2)$. 
For $v\in \widehat{\Gamma}$, 
let 
\begin{align}\label{stackA12}
\widehat{\mM}_{\omega, \theta}(v), 
\end{align}
be the moduli stack of $\widehat{Z}_{\omega, \theta}$-semistable 
objects $E\in \aA_{\omega}(1/2)$ with 
$\widehat{\cl}(E)=v$. 
Similarly to the construction in Subsection~\ref{subsec:invvia}, 
we define the 
element,
\begin{align*}
\widehat{\delta}_{\omega, \theta}(v)
\cneq [ \widehat{\mM}_{\omega, \theta}(v) \hookrightarrow 
\oO bj(\aA_{\omega})] \in \hH(\aA_{\omega}).
\end{align*}
We replace $v_i \in \Gamma$, 
$\delta_{t\omega}(v_i)$, $Z_{t\omega}$
in the sum (\ref{sum:eps}) 
by $v_i \in \widehat{\Gamma}$, $\widehat{\delta}_{\omega, \theta}(v_i)$, 
$\widehat{Z}_{\omega, \theta}$ respectively. 
Then we can define the 
element 
\begin{align*}
\widehat{\epsilon}_{\omega, \theta}(v) \in \hH(\aA_{\omega}),
\end{align*}
for any $v\in \widehat{\Gamma}$, and the rank one invariant, 
\begin{align}\label{widehatDT(v)}
\widehat{\DT}_{\omega, \theta}^{\chi}(v)
\cneq \lim_{q^{1/2} \to 1}(q-1)P_q(\widehat{\epsilon}_{\omega, \theta}(1, -v)),
\end{align}
for $v\in \widehat{\Gamma}_0$. 
Also we replace $\epsilon_{t\omega}(0, v)$, 
$C(\bB_{\omega})$
in Definition~\ref{defi:Ninv}
by
$\widehat{\epsilon}_{\omega, \theta}(0, v)$
and 
\begin{align}\label{CB1/2}
C(\bB_{\omega}(1/2)) \cneq 
\Imm ( \widehat{\cl}_0 \colon 
\bB_{\omega}(1/2) \to \widehat{\Gamma}_{0} ),
\end{align}
respectively.
Then we have the rank zero invariant,
\begin{align}\label{widehatN(v)}
\widehat{N}(v) \in \mathbb{Q},
\end{align}
counting $\widehat{Z}_{\omega, \theta}$-semistable 
objects $E\in \aA_{\omega}(1/2)$
or $E \in \aA_{\omega}(1/2)[1]$
satisfying $\widehat{\cl}(E)=(0, v)$. 
All the details in defining these invariants
follow from the arguments in Subsection~\ref{subsec:invvia}, 
so we omit the detail. 
Also an argument similar to the proof of
Theorem~\ref{sigma:depend}
shows that $\widehat{N}(v)$ does not depend on 
$\omega$ and $\theta$. 
The invariants (\ref{widehatDT(v)}), (\ref{widehatN(v)})
are related to the invariants 
in Subsection~\ref{subsec:invvia} as follows. 
\begin{lem}\label{lem:N=N'}
For $v=(r, \beta) \in \widehat{\Gamma}_0$
and $0<t \ll 1$, we have 
\begin{align*}
\widehat{\DT}^{\chi}_{\omega, 1/2}(r, \beta)
&= \DT_{t\omega}^{\chi}(r, \beta, 0), \\
\widehat{N}(r, \beta) &=N(r, \beta, 0). 
\end{align*}
\end{lem}
\begin{proof}
The result follows from Proposition~\ref{prop:propA12} (ii). 
\end{proof}

\section{Wall-crossing formula}\label{sec:WCF}
In this section, we apply 
the wall-crossing formula for the 
invariants introduced in the 
previous section, and give a proof of Theorem~\ref{thm:main}. 
\subsection{Joyce's formula}\label{subsec:Bi}

Joyce's wall-crossing formula~\cite[Theorem~6.28]{Joy4}
enables us to see how the invariants 
$\DT_{t\omega}^{\chi}(v)$ vary
if we change $t \in \mathbb{R}_{>0}$. 
In general, the wall-crossing 
formula is described in terms of Euler pairing on 
the (numerical) Grothendieck 
group of the underlying 
Calabi-Yau 3-fold. 
 In our situation, 
the Euler pairing is not symmetric 
since $\overline{X}$ is not a Calabi-Yau 3-fold. 
Instead we use the bilinear pairing $\chi$ 
defined in Subsection~\ref{subsec:gen:bi}. 
The existence of $\chi$ satisfying the condition (\ref{chi})
is enough to establish the wall-crossing formula. 

If we apply the wall-crossing formula in~\cite[Equation~(130)]{Joy4}
to the invariants $\DT_{t\omega}^{\chi}(v)$, 
it immediately implies the following:  
for $t_1, t_2>0$ and $v\in \Gamma_0$,
we have 
\begin{align}\notag
\DT_{t_2 \omega}^{\chi}(v) &=
\sum_{\begin{subarray}{c}
l\ge 1, 1\le e \le l, v_i \in \Gamma_0 \\
v_1 + \cdots +v_l=v
\end{subarray}}
\sum_{\begin{subarray}{c}
G \emph{ \rm{ is a connected, simply connected} } \\
\emph{ \rm{graph with vertex} } \{1, \cdots, l\}, 
i\to j \emph{ \rm{implies} }
i<j
\end{subarray}} \\
\label{WCF:joy}
& \frac{1}{2^{l-1}}
U(\{v_1', \cdots, v_l'\}, t_1, t_2)
\prod_{i \to j \emph{\rm{ in }} G}
\chi(v_i', v_j') \prod_{k\neq e} N(v_k) \DT_{t_1 \omega}^{\chi}(v_e). 
\end{align}
Here 
\begin{align*}
v_i'=(0, -v_i) \mbox{ for } i\neq e, \quad 
v_e'=(1, -v_e),
\end{align*}
 and 
$U(\{v_1', \cdots, v_l'\}, t_1, t_2)$
is a certain rational number determined by 
the arguments of $\arg Z_{t_1\omega}(\ast)$ and $\arg Z_{t_2 \omega}(\ast)$
in a combinatorial way. (cf.~\cite[Definition~4.4]{Joy4}.)
Note that a non-zero term of the RHS of (\ref{WCF:joy})
satisfies 
either $v_i' \in \Gamma_0$ or $v_j' \in \Gamma_0$, so the 
Euler pairing $\chi(v_i', v_j')$ makes sense. 
The central results in~\cite{Tolim2} and~\cite{Tcurve1}
provide explicit computations of the combinatorial coefficients in the RHS. 
The result is formulated in terms of the limiting generating series 
discussed in the next subsection.

\subsection{Generating series}\label{subsec:genser}
For an ample divisor $\omega$ on $S$ and 
$t\in \mathbb{R}_{>0}$, we 
consider the following generating series, 
\begin{align*}
\DT_{t\omega}^{\chi}(\overline{X}) \cneq 
\sum_{(r, \beta, n) \in \Gamma_0}
\DT_{t\omega}^{\chi}(r, \beta, n)x^r y^{\beta} z^n. 
\end{align*}
The series $\DT_{t\omega}^{\chi}(\overline{X})$ is 
an element of the following 
vector space, 
\begin{align*}
\DT_{t\omega}^{\chi}(\overline{X}) \in 
\rR_{\omega} \cneq
\prod_{\begin{subarray}{c}\beta \in \mathrm{NS}(S), \\
\omega \cdot \beta \ge 0
\end{subarray}}
\mathbb{C}\db[ x^{\pm 1}, z^{\pm 1} \db] y^{\beta}.
\end{align*}
The vector space $\rR_{\omega}$ 
is a product of a countable number of copies of $\mathbb{C}$, 
and the Euclid topology on $\mathbb{C}$ induces 
a product topology on $\rR_{\omega}$. 
By the existence of wall and chamber 
structure in Lemma~\ref{prop:wall}, 
the following limiting series makes sense, 
\begin{align}\label{series:lim}
\lim_{t\to t_{0} \pm 0}
\DT_{t\omega}^{\chi}(\overline{X}) \in \rR_{\omega}, 
\end{align}
for any $t_0 \in \mathbb{R}_{>0}$. 

On the other hand, there is no ring 
structure on $\rR_{\omega}$, and 
we need to introduce a topological ring 
which acts on $\rR_{\omega}$. 
We set 
\begin{align*}
\rR_{0} \cneq \prod_{\begin{subarray}{c}\beta \in \mathrm{NS}(S), \\
\beta \ge 0
\end{subarray}}
\mathbb{C}[x^{\pm 1}, z^{\pm 1}] y^{\beta}. 
\end{align*}
Noting that 
the possible $\beta \ge 0$ with bounded 
$\omega \cdot \beta$ is finite, we have the natural 
product, 
\begin{align*}
\rR_0 \times \rR_{\omega} \to \rR_{\omega}, 
\end{align*}
which restricts to the ring structure on $\rR_0$. 
By the same reason,
the exponential for any $f\in \rR_{0}$
also makes sense, 
\begin{align*}
\exp(f) \cneq 
\sum_{k \ge 0}
\frac{1}{k!} f^k \in \rR_{0}. 
\end{align*}

\subsection{Wall-crossing formula of generating series}
The wall-crossing formula~\cite[Theorem~6.28]{Joy4}
describes a
difference of the two limiting series 
(\ref{series:lim}).
An argument used in~\cite[Theorem~5.8]{Tcurve1}
yields the following result:
\begin{thm}\label{thm:formula}
We have the following formula:
\begin{align}\notag
&\lim_{t\to t_{0} + 0}
\DT_{t\omega}^{\chi}(\overline{X}) \\
\label{thm:product}
&=\prod_{\begin{subarray}{c} \beta>0, \\
n=\frac{1}{2}rt_0^2 \omega^2
\end{subarray}} 
\exp\left( (n+2r)N(r, \beta, n)  x^r y^{\beta} z^n  \right)^{\epsilon(r)} 
\cdot
\lim_{t\to t_{0} - 0}
\DT_{t\omega}^{\chi}(\overline{X}).
\end{align}
Here $\epsilon(r)=1$ if $r>0$, $\epsilon(r)=-1$ if 
$r<0$ and $\epsilon(r)=0$ if $r=0$. 
\end{thm} 
\begin{proof}
First we note that 
\begin{align*}
\sum_{\begin{subarray}{c} \beta>0, \\
n=\frac{1}{2}rt_0^2 \omega^2
\end{subarray}} \epsilon(r)(n+2r)N(r, \beta, n)x^r y^{\beta} z^n 
\in \rR_{0},
\end{align*}
by Lemma~\ref{cor:N}. 
Therefore the infinite product (\ref{thm:product}) makes sense by 
the argument in the previous subsection. 

Next we note that
the wall-crossing formula (\ref{WCF:joy})
describes the difference between two
limiting series (\ref{series:lim}) 
in terms of $\chi$ and 
invariants of rank zero, i.e. 
$N(v) \in \mathbb{Q}$.  
Also the bilinear map $\chi$ restricts 
to zero on $\Gamma_0$, and  
this is exactly the same situation
 as in~\cite[Theorem~5.8]{Tcurve1}, 
\cite[Theorem~4.7]{Tolim2}. 
Hence the same arguments are applied
to our situation. 
More precisely, let $W_{t_0}$ be the subset
of $\Gamma_0$ defined by  
\begin{align*}
W_{t_0} \cneq\{ v\in \Gamma_0 : Z_{t_0 \omega, 0}(v) \in 
\mathbb{R}_{>0} \sqrt{-1}\}.
\end{align*}
Then $W_{t_0}$ is written as $W_{t_0}^{+} \cup W_{t_0}^{-} \cup W_{t_0}^{0}$, 
\begin{align*}
W_{t_0}^{+}& \cneq \left\{(r, \beta, n) \in \Gamma_0 : 
n=\frac{1}{2}rt_0^2 \omega^2, \ r< 0, \omega \cdot \beta<0 \right\}, \\
W_{t_0}^{-}& \cneq \left\{(r, \beta, n) \in \Gamma_0 : 
n=\frac{1}{2}rt_0^2 \omega^2, \ r> 0, \omega \cdot \beta<0\right\}, \\
W_{t_0}^{0}& \cneq \left\{ (r, \beta, n) \in \Gamma_0 : r=n=0, 
\omega \cdot \beta<0 \right\}.
\end{align*}
For $v\in W_{t_0}^{+}$, we have 
\begin{align*}
\arg Z_{(t_0+\varepsilon)\omega, 0}(v) <
\frac{\pi}{2} <\arg Z_{(t_0 -\varepsilon) \omega, 0}(v),
\end{align*}
for $0<\epsilon \ll 1$. 
The above inequalities are reversed for $v\in W_{t_0}^{-}$
and are equalities for $v\in W_{t_0}^{0}$. 
Also noting the formula 
 (\ref{chi:form}) for $\chi$, 
the arguments in~\cite[Theorem~5.8]{Tcurve1}, 
\cite[Theorem~4.7]{Tolim2}
imply
\begin{align*}
&\lim_{t\to t_{0} + 0}
\DT_{t\omega}^{\chi}(\overline{X}) \\
&=\prod_{-(r, \beta, n) \in W_{t_0}^{+}}
\exp\left( (n+2r)N(r, \beta, n) x^r y^{\beta} z^n  \right) \\
&\qquad \cdot
\prod_{-(r, \beta, n) \in W_{t_0}^{-}}
\exp\left( (n+2r)N(r, \beta, n) x^r y^{\beta} z^n  \right)^{-1}
\cdot 
\lim_{t\to t_{0} - 0}
\DT_{t\omega}^{\chi}(\overline{X}) \\
&=
\prod_{-(r, \beta, n) \in W_{t_0}}
\exp\left( (n+2r)N(r, \beta, n) x^r y^{\beta} z^n  \right)^{\epsilon(r)}
\cdot
\lim_{t\to t_{0} - 0}
\DT_{t\omega}^{\chi}(\overline{X}).
\end{align*}
If $-(r, \beta, n) \in W_{t_0}^{+} \cup W_{t_0}^{-}$ satisfies 
$N(r, \beta, n) \neq 0$, then 
$\beta>0$ follows from Corollary~\ref{cor:N}. 
Therefore we obtain the formula (\ref{thm:formula}). 
\end{proof}
Let $L(\beta, n) \in \mathbb{Q}$
be the invariant, discussed in Subsection~\ref{subsec:product}. 
By applying the wall-crossing formula 
from $t\to 0$ to $t\to \infty$, 
we obtain the following corollary.  
\begin{cor}\label{cor:DTN}
We have the formula, 
\begin{align*}
\sum_{(r, \beta, n) \in \Gamma_0}L(\beta, n)x^r y^{\beta}z^n 
=\prod_{\beta>0, rn>0} 
&\exp\left( (n+2r)N(r, \beta, n) x^r y^{\beta} z^n  \right)^{\epsilon(r)} \\
&\cdot \lim_{t\to 0}\sum_{(r, \beta, 0) \in \Gamma_0}
\DT_{t\omega}^{\chi}(r, \beta, 0)x^r y^{\beta}. 
\end{align*}
\end{cor}
\begin{proof}
By Proposition~\ref{prop:DT=L}, we have 
\begin{align*}
\lim_{t\to \infty}\DT_{t\omega}^{\chi}(\overline{X})=
\sum_{(r, \beta, n)\in \Gamma_0}L(\beta, n)x^r y^{\beta}z^n.
\end{align*}
On the other hand by Proposition~\ref{prop:m2}, we have 
\begin{align*}
\lim_{t\to 0}\DT_{t\omega}^{\chi}(\overline{X})
=\lim_{t\to 0}\sum_{(r, \beta, 0) \in \Gamma_0}
\DT_{t\omega}^{\chi}(r, \beta, 0)x^r y^{\beta}. 
\end{align*}
Therefore applying the formula (\ref{thm:product})
from $0<t\ll 1$ to $t\gg 1$, 
and using the same argument of~\cite[Corollary~5.11]{Tcurve1}, 
we obtain the formula. 
\end{proof}

\subsection{Wall-crossing in $\aA_{\omega}(1/2)$}
In this subsection, we use the notation 
given in Subsection~\ref{subsec:invA12}. 
Our next step is to apply the wall-crossing formula 
in the subcategory $\aA_{\omega}(1/2) \subset \aA_{\omega}$
to prove a formula for the series 
$\lim_{t\to 0}\DT_{t\omega}^{\chi}(\overline{X})$. 
For $0<\theta<1$, we set 
\begin{align*}
\widehat{\DT}_{\omega, \theta}^{\chi}(\overline{X}) \cneq 
\sum_{(r, \beta) \in \widehat{\Gamma}_0}
\widehat{\DT}_{\omega, \theta}^{\chi}(r, \beta)x^r y^{\beta}. 
\end{align*}
We note that 
\begin{align}\label{lim:DT}
\widehat{\DT}_{\omega, 1/2}^{\chi}(\overline{X})=
\lim_{t\to 0}\sum_{(r, \beta, 0) \in \Gamma_0}
\DT_{t\omega}^{\chi}(r, \beta, 0)x^r y^{\beta}
\end{align}
by Proposition~\ref{prop:propA12} (iii). 
By the same arguments of Theorem~\ref{thm:formula}
and Corollary~\ref{cor:DTN}, we obtain the 
following proposition.
\begin{prop}\label{prop:DT=N'}
We have the formula, 
\begin{align}\label{formula:DT=N'}
\widehat{\DT}_{\omega, 1/2}^{\chi}(\overline{X})
=\prod_{r>0, \beta>0}
\exp\left(2r\widehat{N}(r, \beta)x^r y^{\beta} \right) \cdot 
\sum_{r \in \mathbb{Z}}x^r. 
\end{align}
\end{prop}
\begin{proof}
Note that the bilinear map $\chi$
given in (\ref{def:chi}) restricts to a bilinear 
map on $\widehat{\Gamma} \times \widehat{\Gamma}_0$, 
given by 
\begin{align*}
\chi((R, r, \beta), (r', \beta'))=2Rr'. 
\end{align*}
The above bilinear map  
satisfies the same condition as in (\ref{chi})
for $E \in \aA_{\omega}(1/2)$
and $F \in \bB_{\omega}(1/2)$.  
Therefore the same argument of Theorem~\ref{thm:formula}
shows that, for each $\theta_0 \in (0, 1/2)$, we have 
\begin{align}\notag
&\lim_{\theta \to \theta_0 +0}
\widehat{\DT}_{\omega, \theta}^{\chi}(\overline{X}) \\
\label{formula:A12}
&=\prod_{-(r, \beta) \in \widehat{W}_{\theta_0}}
\exp\left(2r\widehat{N}(r, \beta)x^r y^{\beta} \right) \cdot 
\lim_{\theta \to \theta_0 -0}
\widehat{\DT}_{\omega, \theta}^{\chi}(\overline{X}).
\end{align}
Here $\widehat{W}_{\theta_0}$ is defined by 
\begin{align*}
\widehat{W}_{\theta_0} \cneq \{ v\in \widehat{\Gamma}_{0} : 
\widehat{Z}_{\omega, \theta_0, 0}(v) \in \mathbb{R}_{>0}e^{i\pi \theta_0}\}.
\end{align*}
For $(r, \beta) \in \widehat{\Gamma}_0 =\mathbb{Z} \oplus \mathrm{NS}(S)$, 
we have $-(r, \beta) \in \widehat{W}_{\theta_0}$ if and only if 
\begin{align*}
r=\frac{\beta \cdot \omega}{\tan \pi \theta_0} >0. 
\end{align*}
Also 
if $\widehat{N}(r, \beta) \neq 0$ in the formula (\ref{formula:A12}), 
then $\beta>0$ by Corollary~\ref{cor:N} and Lemma~\ref{lem:N=N'}. 
By applying the formula (\ref{formula:A12})
 from $\theta \to 0$ to $\theta \to 1/2$, 
we obtain 
\begin{align}\notag
&\lim_{\theta \to 1/2-0}\widehat{\DT}_{\omega, \theta}^{\chi}(\overline{X}) \\
\label{formula:result}
&=\prod_{r>0, \beta>0}
\exp\left(2r\widehat{N}(r, \beta)x^r y^{\beta} \right) \cdot 
\lim_{\theta \to 0}
\widehat{\DT}_{\omega, \theta}^{\chi}(\overline{X}).
\end{align}
Hence the formula (\ref{formula:DT=N'}) 
follows from (\ref{formula:result}) and 
the following equalities, 
\begin{align}\label{Aequal1}
\lim_{\theta \to 1/2 -0}
\widehat{\DT}_{\omega, \theta}^{\chi}(\overline{X}) &=
\widehat{\DT}_{\omega, 1/2}^{\chi}(\overline{X}), \\
\label{Aequal2}
\lim_{\theta \to 0}
\widehat{\DT}_{\omega, \theta}^{\chi}(\overline{X}) &=
\sum_{r\in\mathbb{Z}}x^r. 
\end{align}
To see the equality (\ref{Aequal1}), 
note that if $v=(r, \beta)\in \widehat{W}_{1/2}$, then 
$r=0$ and $\chi((1, v'), v)=0$ for any $v' \in \widehat{\Gamma}_0$. 
This implies that, 
by the formula given in~\cite[Theorem~6.28]{Joy4}, 
there is no wall-crossing from 
$\theta \to 1/2 -0$ to $\theta =1/2$, and the 
generating series does not change. 

Also note that 
Proposition~\ref{prop:propA12} (iii)
implies that 
\begin{align*}
\widehat{\mM}_{\omega, \theta}(1, r, \beta)
=\left\{ \begin{array}{cc}
[\Spec \mathbb{C}/\mathbb{C}^{\ast}], & \mbox{ if } \beta=0, \\
\emptyset, & \mbox{ if } \beta \neq 0,
\end{array}   \right.
\end{align*}
for $0<\theta \ll 1$. 
Then the equality (\ref{Aequal2}) follows from the 
definition of $\widehat{\DT}_{\omega, \theta}^{\chi}(r, \beta)$. 
\end{proof}

\subsection{Generating series of stable pairs}
Let $\PT^{\chi}(\overline{X})$ and $\PT^{\chi}(X)$
be the generating series of stable pair invariants, 
introduced in Subsection~\ref{subsec:stablepair}.
By combining the results in the previous subsections, 
we prove formulas for these generating series. 
\begin{thm}\label{thm:PTb=N}
We have the formula, 
\begin{align}\label{formula:PTb=N}
\PT^{\chi}(\overline{X})=
\prod_{\beta>0, (r, n) \in \mathbb{S}}
\exp\left( (n+2r)N(r, \beta, n) y^{\beta} z^n  \right)^{\epsilon(r+n)}. 
\end{align}
Here $\mathbb{S} \subset \mathbb{Z}^{\oplus 2}$ is given by 
\begin{align*}
\mathbb{S} \cneq \{ (r, n) \in \mathbb{Z}^{\oplus 2} :
rn>0 \mbox{ or } r=0, n>0, \mbox{ or } r>0, n=0\}. 
\end{align*}
\end{thm} 
\begin{proof}
By Theorem~\ref{thm:PNL}, Corollary~\ref{cor:DTN},
the equality (\ref{lim:DT}),
 Proposition~\ref{prop:DT=N'}
and Lemma~\ref{lem:N=N'},  
we obtain 
\begin{align*}
&\PT^{\chi}(\overline{X}) \cdot 
\sum_{r \in \mathbb{Z}} x^r \\
&=\prod_{\beta>0, (r, n) \in \mathbb{S}}
\exp\left( (n+2r)N(r, \beta, n) x^r y^{\beta} z^n  \right)^{\epsilon(r+n)}
\cdot \sum_{r \in \mathbb{Z}}x^r.  
\end{align*}
By comparing the $x^{0}$-term, we obtain the formula (\ref{formula:PTb=N}). 
\end{proof}
Finally, we prove
our main theorem  
which relates $\PT^{\chi}(X)$to 
sheaf counting invariants $J(v) \in \mathbb{Q}$
for $v\in \Gamma_0$
introduced in Subsection~\ref{subsec:sheafcount}.
\begin{thm}\label{maintheorem:PTJ}
We have the formula, 
\begin{align}\notag
\PT^{\chi}(X)=\prod_{r\ge 0, \beta>0, n\ge 0} &
\exp\left((n+2r)J(r, \beta, r+n)y^{\beta}z^n \right) \\
\label{main:formula2}
&\cdot \prod_{r> 0, \beta>0, n>0}
\exp\left((n+2r)J(r, \beta, r+n) y^{\beta} z^{-n} \right).
\end{align}
\end{thm}
\begin{proof}
The formula (\ref{main:formula2}) follows 
from Theorem~\ref{thm:PTb=N}, Lemma~\ref{lem:eqPT}, 
Corollary~\ref{prop:inv}, and noting that 
\begin{align*}
J(-r, \beta, -n)=J(r, \beta, n), 
\end{align*}
by Theorem~\ref{thm:aut}. 
\end{proof}

\section{Discussion toward Katz-Klemm-Vafa conjecture}\label{sec:KKV}
In this section we discuss how Theorem~\ref{maintheorem:PTJ}
is related to the conjecture by Katz-Klemm-Vafa (KKV)~\cite{KKV}. 
\subsection{KKV conjecture}\label{subsec:KKVconj}
Let $S$ be a K3 surface, and $X=S\times \mathbb{C}$
as before. 
Let $\overline{M}_{g}(X, \beta)$
be the moduli stack of 
stable maps from 
genus $g$ connected nodal 
curves to $X$ with 
curve class $\beta \in \mathrm{NS}(S)$. 
Note that $S$ has a holomorphic symplectic 
form, and 
there is a $\mathbb{C}^{\ast}$-action
on $X$ by multiplying the second factor. 
Therefore 
 $\overline{M}_{g}(X, \beta)$ admits an 
equivariant  
reduced obstruction theory and 
an equivariant 
reduced virtual class,
(see~\cite[Section~1]{MPT},)
\begin{align*}
[\overline{M}_{g}(X, \beta)]^{\rm{red}} \in 
A_1^{\mathbb{C}^{\ast}}(\overline{M}_g(X, \beta), \mathbb{Q}). 
\end{align*}
Since 
$\overline{M}_g(X, \beta)^{\mathbb{C}^{\ast}}$ is compact, 
we can define the 
integration of the reduced virtual class by 
\begin{align*}
\int_{[\overline{M}_g(X, \beta)]^{\rm{red}}} 1 \cneq 
\int_{{[\overline{M}_g(X, \beta)^{\mathbb{C}^{\ast}}]^{\rm{red}}}}
\frac{1}{e(\mathrm{Nor}^{\rm{vir}})} \in \mathbb{Q}(u),
\end{align*}
where $\mathrm{Nor}^{\rm{vir}}$ is the virtual normal 
bundle of the embedding 
$\overline{M}_g(X, \beta)^{\mathbb{C}^{\ast}} \subset
\overline{M}_g(X, \beta)$, 
and $u$ is the equivariant 
parameter 
for the $\mathbb{C}^{\ast}$-action on $X$.  
The reduced GW invariant $R_{g, \beta} \in \mathbb{Q}$ is defined by  
\begin{align}\label{reduced:GW}
R_{g, \beta} 
\cneq \mathrm{Res}_{u=0}\int_{[\overline{M}_g(X, \beta)]^{\rm{red}}} 
1.
\end{align}
The invariant (\ref{reduced:GW}) is 
unchanged under deformations of $S$ 
which preserves $\beta$ to be
an algebraic class. The Gromov-Witten partition function
is 
\begin{align*}
\mathrm{GW}(X) &\cneq 
\sum_{g\ge 0, \beta}
R_{g, \beta}\lambda^{2g-2}y^{\beta} \\
&=\sum_{\beta}\mathrm{GW}_{\beta}(X)y^{\beta},
\end{align*} 
where $\mathrm{GW}_{\beta}(X)$ is a series of 
$\lambda$. 
The BPS number $r_{g, \beta}$ is uniquely defined 
by the equation, 
\begin{align*}
\mathrm{GW}(X)
=\sum_{\beta, g, k}
\frac{r_{g, \beta}}{k}
\left(2\sin\left(\frac{k\lambda}{2} \right) \right)^{2g-2}
y^{k\beta}. 
\end{align*}
The following conjecture is a 
mathematical formulation of KKV conjecture~\cite{KKV}
by Maulik-Pandharipande~\cite{MP1}
in terms of reduced Gromov-Witten invariants. 
\begin{conj}{\bf~\cite[Section~6]{KKV}, 
\cite[Conjecture~1, 2]{MP1}}\label{conj:KKV}

(i) The BPS count $r_{g, \beta}$ depends only on 
$g$ and $\beta^2$.  
If $\beta^2 =2h-2$, then 
we set $r_{g, h} \cneq r_{g, \beta}$. 

(ii) The numbers $r_{g, h}$ are determined 
by the following equation, 
\begin{align}\label{KKV:ii}
\sum_{g=0}^{\infty}
\sum_{h=0}^{\infty}
(-1)^{g} r_{g, h}
\left(\sqrt{z}-\frac{1}{\sqrt{z}} \right)^{2g}
q^{h-1}=
\frac{1}{\Delta(z, q)},
\end{align} 
where $\Delta(z, q)$ is 
\begin{align*}
\Delta(z, q)=q\prod_{n=1}^{\infty}(1-q^n)^{20}(1-zq^n)^2 (1-z^{-1}q^n)^2. 
\end{align*}
\end{conj}
The following result is obtained 
by Maulik-Pandharipande-Thomas~\cite{MPT}. 
\begin{thm}{\bf \cite[Theorem~1]{MPT}}\label{thm:MPT}
The invariants $r_{g, h}$ for primitive curve 
classes satisfy the equation (\ref{KKV:ii}). 
\end{thm}

\subsection{Reduced PT invariants}\label{subsec:reduced2}
Similarly to the reduced GW theory, 
we can define the reduced PT invariants. 
Namely there is an equivariant reduced virtual class in dimension one, 
(cf.~\cite[Section~1]{MPT},) 
\begin{align*}
[P_n(X, \beta)]^{\rm{red}} \in A_1^{\mathbb{C}^{\ast}}
(P_n(X, \beta), \mathbb{Z}),
\end{align*}
and the reduced PT invariant $P_{n, \beta} \in \mathbb{Z}$
is defined by,
\begin{align*}
P_{n, \beta} \cneq 
\mathrm{Res}_{u=0} \int_{[P_n(X, \beta)]^{\rm{red}}}
1. 
\end{align*}
The generating series is defined by,
\begin{align*}
\PT(X) &\cneq \sum_{\beta, n}P_{n, \beta}y^{\beta}z^n \\
&= \sum_{\beta}\PT_{\beta}(X)y^{\beta},
\end{align*}
where $\PT_{\beta}(X)$ is a series of $z$. 
If $\beta$ is an irreducible curve class, 
then $P_{n, \beta}$ 
coincides with the Euler characteristic
invariant, 
\begin{align}\label{pair/chi}
P_{n, \beta}=(-1)^{n-1}\chi(P_n(X, \beta)), 
\end{align}
by~\cite[Lemma~8]{MPT}.
In this case,  
$P_n(X, \beta)$ depends
only on $n$ and the norm $\beta^2$ up to 
deformation equivalence. 
We write $P_n(X, \beta)$ as $P_n(X, h)$
when $\beta^2 =2h-2$. 
The following result is given by 
Kawai-Yoshioka~\cite{KY}. 
(Also see~\cite{Bakker} for higher 
rank generalization.)
\begin{thm}{\bf~\cite[Theorem~5.80]{KY} }\label{KKV:KY}
We have the formula, 
\begin{align}\label{KYcomp}
\sum_{h=0}^{\infty}\sum_{n=1-h}^{\infty}
\chi(P_n(X, h))z^n q^{h-1}
=\left(\sqrt{z}-\frac{1}{\sqrt{z}} \right)^{-2}
\frac{1}{\Delta(z, q)}.
\end{align}
\end{thm} 
Our formula (\ref{main:formula}) reconstructs 
the above result by Kawai-Yoshioka. 
In fact 
when $\beta$ is irreducible
and $n\ge 0$, the formula (\ref{main:formula})
implies that 
\begin{align*} 
\chi(P_n(X, \beta)) &= \sum_{r \ge 0}(n+2r)J(r, \beta, r+n),  \\
\chi(P_{-n}(X, \beta)) &= \sum_{r >0}(n+2r)J(r, \beta, r+n). 
\end{align*}
The above formulas are nothing but 
specializations of~\cite[Equations~(5.168), (5.170)]{KY} respectively. 
Using (\ref{J=H}), we obtain 
\begin{align*}
\chi(P_n(X, h)) &=
 \sum_{r\ge 0}(n+2r)\chi(\Hilb^{h-r(r+n)}(S)), \\
\chi(P_{-n}(X, h)) &=
\sum_{r> 0}(n+2r)\chi(\Hilb^{h-r(r+n)}(S)),
\end{align*}
for $n\ge 0$. 
Together with 
some calculations involving 
G$\ddot{\rm{o}}$ttsche's formula (\ref{Gottscheform}),
we obtain the formula (\ref{KYcomp}). 
(See~\cite[Equations~(5.171), (5.172), (5.173), (5.174)]{KY}.)
Note that Theorem~\ref{thm:MPT}
can be reduced to the case of irreducible 
curve classes by a deformation argument. 
Then Theorem~\ref{thm:MPT} follows 
from Theorem~\ref{KKV:KY}, the formula (\ref{pair/chi})
and the following
reduced version of 
GW/PT correspondence. 
\begin{thm}{\bf \cite[Theorem~9]{MPT}}
Suppose that $\beta$ is a primitive 
curve class. Then 
after the variable change 
$-e^{i\lambda}=z$, we have 
\begin{align*}
\mathrm{GW}_{\beta}(X)=\PT_{\beta}(X). 
\end{align*}
\end{thm}
  
\subsection{Speculation on KKV conjecture}\label{spec:KKV}
As we discussed in the introduction, 
the strategy of the proof of Theorem~\ref{thm:MPT}
in~\cite{MPT}
consists of two steps: 
prove reduced GW/PT correspondence and 
compute reduced PT theory. 
Suppose that we try to prove 
Conjecture~\ref{conj:KKV}
for arbitrary curve classes, 
 along with the same
strategy as in the case of primitive
curve classes~\cite{MPT}. 
Then one might expect the following: 
\begin{itemize}
\item The reduced GW/PT correspondence 
for arbitrary curve classes 
may hold.  Namely we may have 
\begin{align}\label{expect1}
\exp\left(\mathrm{GW}(X) \right)=\PT(X),
\end{align}
by the variable change $-e^{i\lambda}=z$. 
\item 
The series $\PT(X)$ may be written 
as a similar product expansion to (\ref{main:formula}).
For instance, 
looking at the equation (\ref{pair/chi}), 
one may expect the 
following formula:  
\begin{align}\notag
\PT(X)=&\prod_{r\ge 0, \beta>0, n\ge 0} 
\exp\left((-1)^{n-1}(n+2r)J(r, \beta, r+n)y^{\beta}z^n \right) \\
\label{KKV:formula2}
&\cdot \prod_{r> 0, \beta>0, n>0}
\exp\left((-1)^{n-1}(n+2r)J(r, \beta, r+n) y^{\beta} z^{-n} \right).
\end{align}
\end{itemize}
Although $J(v)$ does not involve the virtual 
cycle, it seems likely that $J(v)$ is 
invariant under deformations of $S$
preserving $v$ to be an algebraic class. 
(See Subsection~\ref{subsec:mult3} below.) 
Hence the formula 
(\ref{KKV:formula2}) seems to make sense.  
At this moment we do not know 
whether (\ref{expect1}), (\ref{KKV:formula2}) hold
or not. In particular it might be too strong
to assume (\ref{KKV:formula2}). 
However 
even if (\ref{KKV:formula2}) is not true, 
a similar formula 
should be obtained if one could 
involve the reduced virtual cycles 
in the wall-crossing formula. 
Namely, for instance, suppose that we could
relate reduced PT invariants 
to the weighted Euler characteristic with respect 
to the Behrend function~\cite{Beh}.
Then 
by combining the argument in proving Theorem~\ref{thm:main}, 
work of Joyce-Song~\cite{JS} and the
announced result by Behrend-Getzler~\cite{BG}, 
it should be possible to prove a formula 
similar to (\ref{KKV:formula2}),
possibly by replacing  
$J(v)$ by another counting invariant
which has similar properties to $J(v)$. 
The arguments below may also be applied 
after such an replacement. 
The following result reduces 
Conjecture~\ref{conj:KKV} to the above 
expectations.   
\begin{thm}\label{thm:KKV}
Suppose that the formulas (\ref{expect1}) and (\ref{KKV:formula2}) 
hold
for any K3 surface $S$. 
 Then Conjecture~\ref{conj:KKV} is true. 
Furthermore we have the formula, 
\begin{align}\notag
\PT(X)
=& \prod_{r\ge 0, \beta>0, n\ge 0}
(1+(-1)^{n-1}y^{\beta}z^n)^{(n+2r)\chi(\Hilb^{\beta^2/2 -r(n+r)+1}(S))} \\
\label{product:form}
&\cdot \prod_{r> 0, \beta>0, n> 0}(1+(-1)^{n-1}y^{\beta}z^{-n})^{(n+2r)\chi(\Hilb^{\beta^2/2 -r(n+r)+1}(S))}.
\end{align}
\end{thm}
\begin{proof}
By a deformation argument as 
in~\cite[Section~4]{BrLe}, \cite[Section~2]{MPT}, 
we may assume that $S$ is an elliptically fibered K3
surface $S \to \mathbb{P}^1$ with a section
and $\mathrm{NS}(S)$ is rank two. 
Let 
\begin{align*}
{\bf s}, {\bf f} \in \mathrm{NS}(S),
\end{align*}
be the classes of the section and the elliptic 
fiber. Any $\beta \in \mathrm{NS}(S)$ is written as 
\begin{align}\label{beta:write}
\beta=a{\bf s}+b{\bf f},
\end{align}
for some $a, b \in \mathbb{Z}$. 
Suppose that (\ref{expect1}) and (\ref{KKV:formula2}) hold. 
Then $\PT(X)$ can be described by 
the following Gopakumar-Vafa form, (cf.~\cite[Equation~(18)]{Katz2},)
\begin{align}\notag
\PT(X)=\prod_{\beta >0}\prod_{n=1}^{\infty} &
(1+(-1)^{n-1}y^{\beta}z^n)^{nr_{0, \beta}} \\
\label{GVform}
& \cdot \prod_{g=1}^{\infty}\prod_{k=0}^{2g-2}
(1+(-1)^{g-k}y^{\beta}z^{g-1-k})^{(-1)^{g-k}r_{g, \beta}
\left(\begin{subarray}{c}
2g-2 \\
k
\end{subarray} \right)}. 
\end{align}
By~\cite[Theorem~6.4]{Tsurvey}, 
the series $\PT(X)$ is expressed by
a Gopakumar-Vafa form (\ref{GVform}) if and only 
we have the following multiple cover formula, 
\begin{align}\notag
J(0, \beta, n) &=\sum_{k\ge 1, k|(\beta, n)}
\frac{1}{k^2}J(0, \beta/k, 1) \\
\label{mult1}
&=\sum_{k \ge 1, k|(\beta, n)}
\frac{1}{k^2}
\chi(\Hilb^{\beta^2 /2k^2 +1}(S)). 
\end{align}
Here the second equality follows from (\ref{J=H}). 
We claim that for any $v=(r, \beta, n) \in \Gamma_0$,
we have the multiple cover formula, 
\begin{align}\label{mult2}
J(v)=\sum_{k\ge 1, k|v}\frac{1}{k^2}\chi(\Hilb^{(v/k, v/k)/2 +1}(S)).
\end{align}
In order to prove (\ref{mult2}), 
we write $\beta$ as (\ref{beta:write})
for $a, b\in \mathbb{Z}$, and set 
\begin{align*}
(r, a)=d(\overline{r}, \overline{a}), 
\end{align*}
where $d=\mathrm{GCD}(r, a)>0$. 
By Theorem~\ref{thm:aut}, we may assume that 
$r>0$, hence $\overline{r}>0$. 
Let $S' \to \mathbb{P}^1$ be the relative 
moduli space of stable sheaves on the fibers
of the elliptic fibration $S \to \mathbb{P}^1$
with rank $\overline{r}$ and degree $\overline{a}$ on fibers. 
Then $S'$ is also an elliptically fibred 
smooth K3 surface with a section, 
and we denote by ${\bf s'}, {\bf f'} \in \mathrm{NS}(S')$
the classes of the section and the elliptic fiber. 
The universal sheaf on 
$S \times_{\mathbb{P}^1} S'$ induces a derived equivalence, 
(cf.~\cite[Theorem~3.11]{Or1}, \cite[Theorem~5.3]{Brell},)
\begin{align*}
\Phi \colon D^b \Coh(S') \stackrel{\sim}{\to} D^b \Coh(S). 
\end{align*}
As we will recall in Subsection~\ref{subsec:Aut},
the equivalence $\Phi$ fits into a commutative diagram, 
(cf.~\cite{Mu2}, \cite{Or1},)
\begin{align*}
\xymatrix{
 D^b \Coh(S') \ar[r]^{\Phi} \ar[d]_{\ch \sqrt{\td_S}} & 
D^b \Coh(S) \ar[d]^{\ch \sqrt{\td_S}} \\
\widetilde{H}(S', \mathbb{Z}) \ar[r]^{\Phi_{\ast}} & 
\widetilde{H}(S, \mathbb{Z}), 
}
\end{align*}
for an isomorphism $\Phi_{\ast}$. 
By the construction of $S'$ and $\Phi$, we have 
\begin{align*}
\Phi_{\ast}^{-1}(\overline{r}, \overline{a}{\bf s}, 0)
=(0, {\bf s'} +b'{\bf f}, m),
\end{align*}
for some $b', m \in \mathbb{Z}$. 
Also since 
\begin{align*}
\Phi_{\ast}(\mathbb{Z}[{\bf f'}] \oplus H^4(S', \mathbb{Z}))
=\mathbb{Z}[{\bf f}] \oplus H^4(S, \mathbb{Z}) , 
\end{align*}
by the construction of $\Phi$, it follows that 
\begin{align*}
\Phi_{\ast}^{-1}(r, \beta, n)=(0, \beta', n'), 
\end{align*}
for some $\beta' \in \mathrm{NS}(S')$
and $n' \in \mathbb{Z}$. 
It is easy to see that $\Phi$
satisfies the assumption in 
Proposition~\ref{prop:equivalence} below.
Hence combined with Lemma~\ref{Jlocarize}, we have 
\begin{align*}
J_{S}(r, \beta, n)=J_{S'}(0, \beta', n'). 
\end{align*}
Here we have written $J(v)$ as $J_{S}(v)$
in order to distinguish the invariants on 
$S$ and $S'$. Then (\ref{mult2}) follows 
from (\ref{mult1}) for $S'$. 

By the formula (\ref{mult2}), we have 
\begin{align*}
&\exp\left((-1)^{n-1}(n+2r)J(r, \beta, r+n)y^{\beta}z^n   \right) \\
&=\exp\left((-1)^{n-1}(n+2r)\sum_{ \begin{subarray}{c}
k \ge 1, \\
k|(r, \beta, n)
\end{subarray}}
\frac{1}{k^2}
\chi(\Hilb^{(r/k, \beta/k, r/k +n/k)^2 /2+1}(S))y^{\beta}z^n   \right) \\
&= \exp \left( \sum_{k \ge 1}  \frac{(-1)^{kn-1}}{k}(n+2r)
\chi(\Hilb^{(r, \beta, r+n)^2 /2 +1}(S)) y^{k\beta} z^{kn} \right) \\
&=\left(1+(-1)^{n-1}y^{\beta} z^n  
\right)^{(n+2r)\chi(\Hilb^{\beta^2 /2 -r(r+n)+1}(S))}.
\end{align*}
Therefore the formula (\ref{product:form})
follows. Comparing (\ref{product:form})
and (\ref{GVform}), we see that $r_{g, \beta}$
depends only on $g$ and $\beta^2$, i.e. 
Conjecture~\ref{conj:KKV} (i) follows. 
Moreover the formula (\ref{product:form}) implies
that $r_{g, \beta} \neq 0$ only if 
$\beta^2 \ge -2$. 
If we write $\beta^2 =2h-2$ for $h\ge 0$, we have 
\begin{align*}
({\bf s} +h {\bf f})^2 =\beta^2. 
\end{align*}
Therefore the computation of $r_{g, \beta}$
can be reduced to the primitive case.
Hence Conjecture~\ref{conj:KKV} (ii) follows from 
Theorem~\ref{thm:MPT}. 
\end{proof}

\subsection{Multiple cover formula}\label{subsec:mult3}
In the proof of Theorem~\ref{thm:KKV}, we have
observed the following conjectural multiple cover 
formula: 
\begin{conj}\label{conj:mult3}
For $v\in \Gamma_0$, we have  
\begin{align}\label{mult3}
J(v)=\sum_{k\ge 1, k|v}\frac{1}{k^2}
\chi(\Hilb^{(v/k, v/k)/2 +1}(S)).
\end{align}
\end{conj}
The above conjecture also indicates that 
$J(v)$ is invariant under deformations of $S$
preserving $v$ to be an algebraic class. 
If we assume the formula (\ref{mult3}), then 
the same computation in the proof of Theorem~\ref{thm:KKV}
shows that 
\begin{align}\notag
\PT^{\chi}(X)
=& \prod_{r\ge 0, \beta>0, n\ge 0}(1-y^{\beta}z^n)^{-(n+2r)\chi(\Hilb^{\beta^2/2 -r(n+r)+1}(S))} \\
\label{Euform}
&\cdot \prod_{r> 0, \beta>0, n> 0}(1-y^{\beta}z^{-n})^{-(n+2r)\chi(\Hilb^{\beta^2/2 -r(n+r)+1}(S))}.
\end{align}
The formula (\ref{Euform}) may be 
interpreted as an Euler characteristic version of 
KKV conjecture for stable pairs. Namely if we 
define $r_{g, \beta}'$ by the formula, 
\begin{align}\notag
&\PT^{\chi}(X)=\\
&\notag \prod_{\beta >0}\prod_{n=1}^{\infty} 
(1-y^{\beta}z^n)^{-nr_{0, \beta}'} 
 \prod_{g=1}^{\infty}\prod_{k=0}^{2g-2}
(1-y^{\beta}z^{g-1-k})^{(-1)^{g-k-1}r_{g, \beta}'
\left(\begin{subarray}{c}
2g-2 \\
k
\end{subarray} \right)},  
\end{align}
then $r_{g, \beta}'$ satisfies the same
conditions in Conjecture~\ref{conj:KKV}. 
In what follows, we give some evidence of
 the conjectural formula (\ref{mult3})
in some examples. 
\begin{lem}\label{lem:00n}
For $v=(0, 0, n)$, we have 
\begin{align*}
J(0, 0, n)=24\sum_{k\ge 1, k|n} \frac{1}{k^2}. 
\end{align*}
In particular the formula (\ref{mult3}) holds. 
\end{lem}
\begin{proof}
Since $\chi(X)=24$, 
the result follows from~\cite[Example~6.2]{JS}, 
\cite[Remark~5.14]{Tcurve1}.
\end{proof}
Another evidence is as follows: 
\begin{lem}\label{lem:r0r}
For $v=(r, 0, r)$, we have 
\begin{align*}
J(r, 0, r)=\frac{1}{r^2}. 
\end{align*}
In particular the formula (\ref{mult3}) holds. 
\end{lem}
\begin{proof}
Let $E \in \Coh_{\pi}(X)$ be an $\omega$-Gieseker
semistable sheaf with $v(E)=(r, 0, r)$, 
and $E_1, \cdots, E_k$ be $\omega$-Gieseker
stable factors of $E$. By changing $\omega$
if necessary, we may assume that 
$v(E_i)=(r_i, 0, r_i)$ for some $r_i \in \mathbb{Z}$. 
Then Lemma~\ref{lem:Bog} implies that 
$r_i=1$, hence all the $E_i$ is isomorphic to 
$\oO_{X_p}$ for some $p\in \mathbb{C}$. 
By the localization argument given in Lemma~\ref{Jlocarize} below, 
we may assume that $p=0 \in \mathbb{C}$. 
Then $J(r, 0, r)$ is a counting invariant
of objects in $\langle \oO_{X_0} \rangle_{\ex}$, 
given by $r$-times extensions of $\oO_{X_0}$. 
Noting that the category $\langle \oO_{X_0} \rangle_{\ex}$
resembles the category of representations of 
a quiver with one vertex and one arrow, 
we can apply the same argument of~\cite[Example~7.27]{JS}
to compute $J(r, 0, r)$. 
We leave the readers to check the detail. 
\end{proof}
Next we focus on the following situation. 
Let 
\begin{align*}
S \to \mathbb{P}^2,
\end{align*}
be a double cover branched along 
a general sextic. 
Let $H \in \mathrm{NS}(S)$ be a pull-back 
of a hyperplane in $\mathbb{P}^2$ to $S$. 
Note that $H^2 =2$ and 
$\mathrm{NS}(S)=\mathbb{Z}[H]$. 
We have the following lemma. 
\begin{prop}\label{lem:176}
In the above situation, 
take $v=(0, 2H, -2)$.
Then 
we have 
\begin{align}\label{176}
J(v)=176337. 
\end{align}
In particular the formula (\ref{mult3}) holds. 
\end{prop}
\begin{proof}
Note that the RHS of (\ref{mult3}) is 
\begin{align*}
\chi(\Hilb^5(S))+\frac{1}{4}\chi(\Hilb^{2}(S)) 
&=176256+\frac{1}{4} \cdot 324 \\
&=176337,
\end{align*}
from the formula (\ref{Gottscheform}). 
We compute $J(v)$ directly from its definition, 
in the same way as in~\cite[Section~5]{Trk2}. 
In order to simplify the notation, we omit 
$\omega$ and $X$ in the notation of Subsection~\ref{subsec:sheafcount}. 

We fist note that, 
since $v' \cneq (0, H, -1)$ is primitive, the 
moduli stack $\mM(v')$ is written as 
\begin{align*}
\mM(v') =[M(v')/\mathbb{C}^{\ast}], 
\end{align*}
for a holomorphic symplectic manifold 
$M(v')$ of 
dimension $(v', v')+2 =4$.
Therefore $M(v')$ is deformation equivalent to 
$\Hilb^2(S)$ and 
\begin{align*}
\chi(M(v')) &=\chi(\Hilb^2(S)) \\
&=324. 
\end{align*}
Next we observe that 
the moduli stack $\mM(v)$ 
has a stratification, 
\begin{align*}
\mM(v)^{(0)} \sqcup \mM(v)^{(1)} \sqcup \mM(v)^{(2)} \sqcup
\mM(v)^{(3)} \sqcup \mM(v)^{(4)},
\end{align*}
where each $\mM(v)^{(i)}$ is the following: 
\begin{itemize}
\item $\mM(v)^{(0)}$ corresponds to $\omega$-Gieseker stable sheaves. 
\item $\mM(v)^{(1)}$ corresponds to sheaves $E$ which fits into 
a non-split exact sequence 
$0 \to E_1 \to E \to E_2 \to 0$
with $[E_i] \in M(v')$,  
and $E_1$ is not isomorphic to $E_2$. 
\item $\mM(v)^{(2)}$ corresponds to sheaves $E$
which is isomorphic to $E_1 \oplus E_2$
with $[E_i] \in M(v')$, 
and $E_1$ is not isomorphic to $E_2$. 
\item $\mM(v)^{(3)}$ corresponds to sheaves $E$ which fits into 
a non-split exact sequence 
$0 \to E' \to E \to E' \to 0$ 
with $[E'] \in M(v')$. 
\item $\mM(v)^{(4)}$ corresponds to sheaves $E$
which is isomorphic to $E^{'\oplus 2}$
with $[E'] \in M(v')$. 
\end{itemize}
We compute the contributions of each 
strata to the invariant $J(v)$. 
First the strata
$\mM(v)^{(0)}$ is written as 
\begin{align*}
\mM(v)^{(0)} =
[M(v)^{(0)}/\mathbb{C}^{\ast}],
\end{align*}
for a smooth variety $M(v)^{(0)}$
of dimension $(v, v)+2=10$, 
with a trivial $\mathbb{C}^{\ast}$-action.  
The variety $M(v)^{(0)}$ is 
birational to  
O'Grady's 10-dimensional symplectic manifold~\cite{Ogra1}, 
 and its 
Euler characteristic 
is computed by Mozgovoy~\cite[Subsection~4.3.1]{Moz},
\begin{align}\label{709}
\chi(M(v)^{(0)})=70956. 
\end{align}
Next the contribution of 
$\mM(v)^{(1)}$ to $\epsilon(v)$ is 
\begin{align*}
&\frac{1}{2}
\bigcup_{\begin{subarray}{c}(E_1, E_2) \in M(v')^{\times 2}, \\
 E_1 \neq E_2
\end{subarray}}
\left[ \left[\frac{\Ext_{X}^1(E_2, E_1) \setminus \{0\}}{\Hom(E_2, E_1) \rtimes (\mathbb{C}^{\ast})^{2}} \right] \to \cC oh_{\pi}(X)  \right]  \\
&=\frac{1}{2}
\bigcup_{\begin{subarray}{c}(E_1, E_2) \in M(v')^{\times 2}, \\
 E_1 \neq E_2
\end{subarray}}
\left[[\mathbb{P}^1/\mathbb{C}^{\ast}] \to \cC oh_{\pi}(X) \right]. 
\end{align*}
Applying $(q-1)P_t(\ast)$ and taking the limit 
$q^{1/2} \to 1$, the contribution to $J(v)$ is 
\begin{align}\notag
\chi(M(v')^2)-\chi(M(v')) &=324^2 -324 \\
\label{104}
&= 104652.
\end{align}
The contribution of $\mM(v)^{(2)}$ to $J(v)$
can be shown to be zero by a similar argument of~\cite[Lemma~5.6]{Trk2}. 
The contribution of $\mM(v)^{(3)}$ to $\epsilon(v)$ is 
\begin{align*}
&\frac{1}{2} \bigcup_{E' \in M(v')} \left[
\left[ \frac{\Ext_{X}^1(E', E') \setminus \{0\}}{\Hom(E', E') \rtimes 
(\mathbb{C}^{\ast})^{2}} \right]  \to \cC oh_{\pi}(X) \right] \\
&=\frac{1}{2} \bigcup_{E' \in M(v')} \left[
\left[ \frac{\mathbb{P}^4}{\mathbb{A}^1 \rtimes \mathbb{C}^{\ast}} \right]  \to \cC oh_{\pi}(X) \right]. 
\end{align*}
Hence the contribution to $J(v)$ is 
\begin{align}\notag
\frac{1}{2}\cdot 5 \cdot \chi(M(v')) &=\frac{5}{2} \cdot 324 \\
\label{810}
&= 810.
\end{align}
Finally the contribution of $\mM(v)^{(4)}$ to $J(v)$ 
can be computed similarly to~\cite[Lemma~5.8]{Trk2}, 
\begin{align}\notag 
-\frac{1}{4}\chi(M(v'))&=-\frac{1}{4} \cdot 324 \\
\label{81}
&=-81. 
\end{align}
Summing up, we obtain 
\begin{align*}
J(v) &=(\ref{709}) +(\ref{104}) +(\ref{810}) +(\ref{81}) \\
&=70956+104652+810-81 \\
&=176337,
\end{align*}
as expected. 
\end{proof}

\begin{rmk}
By Theorem~\ref{thm:aut}, 
if $J(v)$ satisfies the formula (\ref{mult3}),
then $J(gv)$ for a Hodge isometry $g \in G$ also 
satisfies (\ref{mult3}). 
In particular, if 
$v$ is given in either Lemma~\ref{lem:00n} 
or Lemma~\ref{lem:r0r} or Proposition~\ref{lem:176}, 
then $J(gv)$
satisfies (\ref{mult3})
for any Hodge isometry $g$. 
For instance: 
\begin{itemize}
\item The map 
$(r, \beta, n) \mapsto (n, \beta, r)$
is a Hodge isometry. 
In particular by Lemma~\ref{lem:00n}, 
$J(n, 0, 0)$ also satisfies (\ref{mult3}). 
\item The map $(r, \beta, n) \mapsto (-r, \beta, -n)$
is a Hodge isometry. 
In particular in the situation of 
Proposition~\ref{lem:176}, 
$J(0, 2H, 2)$ satisfies (\ref{mult3}). 
\item For $v \in \Gamma_0$ with 
$(v, v)=-2$, 
the map on $\widetilde{H}(S, \mathbb{Z})$,
\begin{align*}
r_v \colon x \mapsto x+(x, v)v,
\end{align*}
is a Hodge isometry. 
In particular in the situation of 
Proposition~\ref{lem:176}, by 
applying $r_v$ where $v$ is 
\begin{align*}
v=v(\oO_S(H))=(1, H, 2), 
\end{align*}
the invariant $J(2, 0, -2)$ 
can be shown to satisfy (\ref{mult3}). 
\end{itemize}
\end{rmk}

\section{Results on the category $\aA_{\omega}$}\label{sec:tech}
In this section, we prove several
properties on the category 
\begin{align*}
\aA_{\omega} =\langle \pi^{\ast}\Pic(\mathbb{P}^1), 
\bB_{\omega} \rangle_{\ex} \subset \dD,
\end{align*}
defined in Definition~\ref{defi:Ao}.
Especially we will prove Lemma~\ref{instead}
which is used in
the proof of Proposition~\ref{heart}. 
\subsection{Properties of $\aA_{\omega}$}
\label{tech:t-st}
First we construct the heart
of a certain bounded t-structure on $D^b \Coh(\overline{X})$. 
Let $\fF_{\omega}'$ be the subcategory of $\Coh(\overline{X})$, 
defined by 
\begin{align*}
\fF_{\omega}' \cneq \{ E \in \Coh(\overline{X}) : 
\Hom(\tT_{\omega}, E)=0\}. 
\end{align*}
Here $\tT_{\omega}$ is defined in 
(\ref{def:Tomega}). 
Since $\Coh(\overline{X})$ is noetherian, 
the pair $(\tT_{\omega}, \fF_{\omega}')$
is a torsion pair on $\Coh(\overline{X})$.
Also we have
\begin{align}\label{F'=F}
\fF_{\omega}' \cap \Coh_{\pi}(\overline{X})
=\fF_{\omega}, 
\end{align}
where $\fF_{\omega}$ is defined in (\ref{def:Fomega}). 
\begin{defi}\label{defi:Ao'}
We define $\aA_{\omega}'$ to be 
\begin{align*}
\aA_{\omega}' \cneq \langle
\fF_{\omega}', \tT_{\omega}[-1]
 \rangle_{\ex}
\subset D^b \Coh(\overline{X}). 
\end{align*}
\end{defi}
The category $\aA_{\omega}'$ is the heart of a bounded 
t-structure on $D^b \Coh(\overline{X})$. 
It contains any line bundle on $\overline{X}$
and objects in $\bB_{\omega}$. 
\begin{lem}\label{twolem}
(i) The subcategory $\bB_{\omega} \subset \aA_{\omega}'$
is closed under subobjects and quotients. 

(ii) We have 
\begin{align}\label{vanish}
\Hom(E, \pi^{\ast}\oO_{\mathbb{P}^1}(r))=0,
\end{align}
for any $E \in \bB_{\omega}$ and $r\in \mathbb{Z}$. 

(iii) Any non-zero morphism $u \colon \pi^{\ast}\oO_{\mathbb{P}^1}(r)
\to \pi^{\ast}\oO_{\mathbb{P}^1}(r')$ fits into an exact sequence 
in $\aA_{\omega'}$, 
\begin{align}\label{mm'}
0 \to \pi^{\ast}\oO_{\mathbb{P}^1}(r) \to \pi^{\ast}\oO_{\mathbb{P}^1}(r')
\to T \to 0, 
\end{align}
with $T \cong \pi^{\ast}Q \in \fF_{\omega}$
for a zero dimensional sheaf $Q$ on $\mathbb{P}^1$.  

(iv) For any morphism $u \colon \pi^{\ast}\oO_{\mathbb{P}^1}(r) \to E$
with $E \in \bB_{\omega}$, 
we have $\Ker(u) \in \pi^{\ast}\Pic(\mathbb{P}^1)$. 
\end{lem}
\begin{proof}
(i) Take an object $E\in \bB_{\omega}$ and an exact sequence
in $\aA_{\omega}'$, 
\begin{align}\label{FEG}
0 \to F \to E \to G \to 0. 
\end{align}
We need to show that $F, G \in \bB_{\omega}$. 
By the definition of $\aA_{\omega}'$, 
we have 
\begin{align}\label{H1FG}
\hH^1(F), \hH^1(G) \in \tT_{\omega} \subset \Coh_{\pi}(\overline{X}). 
\end{align}
Also since $\hH^0(F)$ is a subsheaf of $\hH^0(E)$, 
we have  
$\hH^0(F) \in \Coh_{\pi}(\overline{X})$. 
Hence by the long exact sequence of cohomologies
associated to (\ref{FEG}), we have 
$\hH^0(G) \in \Coh_{\pi}(\overline{X})$. 
Since (\ref{F'=F}) holds, 
we have 
\begin{align}\label{H0FG}
\hH^0(F), \hH^0(G) \in \fF_{\omega}. 
\end{align}
By (\ref{H1FG}) and (\ref{H0FG}), 
we conclude $F, G \in \bB_{\omega}$. 

(ii) By the definition of $\bB_{\omega}$, we
 may assume that $E\in \fF_{\omega}$ or $E\in \tT_{\omega}[-1]$. 
If $E\in \fF_{\omega}$, then (\ref{vanish}) is 
obviously follows. 
Suppose that $E \in \tT_{\omega}[-1]$. We may assume that, 
as in (\ref{ip}),  
$E$ is isomorphic to 
$i_{p\ast}E'[-1]$ for an $\omega$-Gieseker stable 
sheaf $E'$ on $X_p$ for some $p\in \mathbb{P}^1$. 
We have 
\begin{align*}
&\Hom_{\overline{X}}(i_{p\ast}E'[-1], \pi^{\ast}\oO_{\mathbb{P}^1}(r)) \\
&\cong \Hom_{X_p}(E', i_{p}^{!}\oO_{\overline{X}}[1]) \\
& \cong \Hom_{X_p}(E', \oO_{X_p}). 
\end{align*}
Since $E'$ is $\mu_{\omega}$-stable
sheaf on $X_p$ with 
positive slope, we have the vanishing 
 $\Hom_{X_p}(E', \oO_{X_p})=0$. 

(iii) If $u$ is non-zero, then $u$ is an
injection of sheaves and the cokernel is written as 
$\pi^{\ast}Q$ for a zero dimensional sheaf $Q$ 
on $\mathbb{P}^1$. Since $\pi^{\ast}Q \in \fF_{\omega}$, 
we have the exact sequence (\ref{mm'}). 

(iv) The morphism $u$ factors through 
the subobject $\hH^0(E) \subset E$ in 
$\bB_{\omega}$. Let $F$ be the image subsheaf of $u$
in $\hH^0(E)$, 
\begin{align*}
\pi^{\ast}\oO_{\mathbb{P}^1}(r) \stackrel{j}{\twoheadrightarrow} 
F \hookrightarrow \hH^0(E). 
\end{align*}
Since $\hH^0(E) \in \fF_{\omega}$
and $F$ is a subsheaf of $\hH^0(E)$, we have 
$\mu_{\omega+}(F)\le 0$. 
On the other hand, the surjection $j$
factors through the surjection 
\begin{align*}
\pi^{\ast}\oO_{\mathbb{P}^1}(r) \twoheadrightarrow
\pi^{\ast}\oO_W \twoheadrightarrow F,
\end{align*}
for some zero dimensional subscheme $W \subset \mathbb{P}^1$. 
This implies that $\mu_{\omega-}(F) \ge 0$, 
and hence 
$F$ is $\mu_{\omega}$-semistable with $\mu_{\omega}(F)=0$. 
If $F$ is $\mu_{\omega}$-stable, the surjection 
$\pi^{\ast} \oO_{W} \twoheadrightarrow F$ 
implies that $F\cong \oO_{X_p}$ for some $p\in \mathbb{P}^1$
and the kernel of $j$ is 
isomorphic to $\pi^{\ast}\oO_{\mathbb{P}^1}(r-1)$. 
In general by
 the induction on the number of Jordan-H$\ddot{\rm{o}}$lder
factors of $F$, we can easily see that 
any Jordan-H$\ddot{\rm{o}}$lder factor of 
$F$ is isomorphic to $\oO_{X_p}$ for some 
$p\in \mathbb{P}^1$ and the kernel of 
$j$ is isomorphic to $\pi^{\ast}\oO_{\mathbb{P}^1}(r')$
for some $r' \in \mathbb{Z}$. 
\end{proof}
\begin{lem}\label{instead}
Take $E, E' \in \aA_{\omega}$ and 
a morphism $u \colon E \to E'$ 
in $\aA_{\omega}'$. Then we have 
\begin{align}\label{KIC}
\Ker(u), \ \Cok(u) \in \aA_{\omega}. 
\end{align}
Here the kernel and the cokernel
are taken in the abelian category $\aA_{\omega}'$. 
\end{lem}
We divide the proof into 2 steps. 
\begin{step}\label{s1}
We have (\ref{KIC}) when 
$E' \in \bB_{\omega}$
or $E'\in \pi^{\ast}\Pic(\mathbb{P}^1)$. 
\end{step}
\begin{proof}
We show the case of $E' \in \bB_{\omega}$. 
The other case is similarly discussed. 
By the definition of $\aA_{\omega}$, there 
is a filtration in $\aA_{\omega}'$, 
\begin{align}\label{filtA}
0=E_0 \subset E_1 \subset \cdots \subset E_N=E, 
\end{align}
such that each $F_i=E_i/E_{i-1}$
is either an object in $\bB_{\omega}$ or of the 
form $\pi^{\ast}\oO_{\mathbb{P}^1}(r)$ for 
some $r\in \mathbb{Z}$. 
We prove (\ref{KIC}) by the induction on $N$. 
If $N=1$, then (\ref{KIC}) follows from 
Lemma~\ref{twolem}. 
Suppose that (\ref{KIC}) holds for 
$E=F'$, and take an exact sequence 
in $\aA_{\omega}'$, 
\begin{align*}
0 \to F'' \to F \to F' \to 0, 
\end{align*}
with $F''$ an object in either $\bB_{\omega}$
or $\pi^{\ast}\Pic(\mathbb{P}^1)$.  
For a morphism $u \colon F \to E'$, 
let $A$ be the image of the composition
\begin{align*}
F'' \to F \to E',
\end{align*}
 in $\aA_{\omega}'$. 
By setting $B=E'/A$ in $\aA_{\omega}'$, we obtain the 
commutative diagram of exact sequences in $\aA_{\omega}'$, 
\begin{align*}
\xymatrix{
0 \ar[r] & F'' \ar[r] \ar[d]_{u''} & F \ar[r] \ar[d]_{u} & 
F' \ar[r] \ar[d]_{u'} & 0 \\
0 \ar[r] & A \ar[r] & E' \ar[r] & B \ar[r] & 0.
}
\end{align*}
Note that $u''$ is surjective in $\aA_{\omega}'$, 
and $A, B \in \bB_{\omega}$ by Lemma~\ref{twolem} (i).  
By the assumption of the induction,
we have 
\begin{align*}
\Ker(u''), \ \Ker(u'), \ \Cok(u') \in \aA_{\omega}. 
\end{align*}
Therefore (\ref{KIC}) holds
for $u\colon F \to E'$ by the snake lemma. 
\end{proof}
\begin{step}
We have (\ref{KIC}) for any $E' \in \aA_{\omega}$. 
\end{step}
\begin{proof}
We take a $N$-step 
filtration of $E'$ as in (\ref{filtA})
and prove (\ref{KIC}) by the induction 
on $N$. The case of $N=1$
is proved in Step~\ref{s1}. 
Suppose that $N\ge 2$ and take 
an exact sequence in $\aA_{\omega}'$, 
\begin{align*}
0 \to A \to E' \to B \to 0, 
\end{align*}
with 
$A \in \aA_{\omega}$ and 
$B$ is either an object in $\bB_{\omega}$
or in $\pi^{\ast}\Pic(\mathbb{P}^1)$. 
Let $D$ be image of the
 composition in $\aA_{\omega}'$, 
\begin{align*}
E \stackrel{u}{\to} E' \to B.
\end{align*}
We also denote its kernel in 
$\aA_{\omega}'$ 
by $C$. 
By Step~\ref{s1}, we have 
$C\in \aA_{\omega}$. We have the 
morphism of the exact sequences in $\aA_{\omega}'$, 
\begin{align*}
\xymatrix{
0 \ar[r] & C \ar[r] \ar[d]_{u''} & E \ar[r] \ar[d]_{u} & 
D \ar[r] \ar[d]_{u'} & 0 \\
0 \ar[r] & A \ar[r] & E' \ar[r] & B \ar[r] & 0.
}
\end{align*}
Note that $u'$ is injective in $\aA_{\omega}'$. 
Similarly to Step~\ref{s1}, 
(\ref{KIC}) follows from the inductive 
assumption, Step~\ref{s1} and the 
snake lemma. 
\end{proof}

Now we have proved Lemma~\ref{instead}, 
so the proof of Proposition~\ref{heart} is completed. 
In particular $\aA_{\omega}$ is an abelian category, 
and we use this fact in what follows. 

\subsection{Filtrations in $\aA_{\omega}$}
In this subsection, we collect some 
results which will be used
in later sections. 
In particular, the results here will be used 
in proving Proposition~\ref{prop:m2} in Subsection~\ref{subsec:finset2}, and 
proving
Proposition~\ref{prop:m1} in Subsection~\ref{subsec:finset}. 
Here we use the notation in Subsections~\ref{subsec:heartD0}, 
\ref{subsec:heartD}, \ref{sub:R1}. 
Let $\tT_{\omega}^{\rm{pure}}$ be
the following subcategory of $\tT_{\omega}$, 
\begin{align*}
\tT_{\omega}^{\rm{pure}} \cneq 
\{ E\in \tT_{\omega} : E \mbox{ is pure two dimensional } \}
\cup \{0\}.
\end{align*}
Note that $\tT_{\omega}^{\rm{pure}}$
is a right orthogonal complement of $\Coh_{\pi}^{\le 1}(\overline{X})$
in $\tT_{\omega}$. 
\begin{lem}\label{lem:rank10}
For any object $E\in \aA_{\omega}$ with $\rank(E)\le 1$, 
there is a filtration in $\aA_{\omega}$, 
\begin{align}\label{filt:more}
E_1 \subset E_2 \subset E_3=E, 
\end{align}
such that we have 
\begin{align}\label{filt:filt}
K_1 \cneq E_1 \in \fF_{\omega}, \  
K_2 \cneq E_2/E_1 \in \aA(r), \ 
K_3 \cneq E/E_2 \in \tT_{\omega}^{\rm{pure}}[-1],
\end{align}
for some $r\in \mathbb{Z}$. 
If $\rank (E)=0$, we can take 
$K_2 \in \Coh^{\le 1}_{\pi}(\overline{X})[-1]$. 
\end{lem}
\begin{proof}
When $\rank(E)=0$, then $E \in \bB_{\omega}$
and the statement is obvious by the definition of 
$\bB_{\omega}$. 
Suppose that $\rank(E)=1$. 
Because $E\in \aA_{\omega}'$ and $\aA_{\omega}'$
is obtained as a tilting of the torsion pair 
$(\tT_{\omega}, \fF_{\omega}')$, 
(cf.~Definition~\ref{defi:Ao'},)
we can find
a filtration in $\aA_{\omega}'$, 
\begin{align}\label{filt123}
E_1' \subset E_2' \subset E_3'=E
\end{align}
satisfying 
\begin{align}\label{filt:sat}
E_1'=\hH^0(E)_{\rm{tor}}, \quad E_2'/E_1'=\hH^0(E)_{\rm{fr}}, \quad 
E/E_2' =\hH^1(E)[-1]. 
\end{align} 
Here $\hH^0(E)_{\rm{tor}}$ is the maximal 
torsion subsheaf of $\hH^0(E)$, and 
$\hH^0(E)_{\rm{fr}} \cneq \hH^0(E)/\hH^0(E)_{\rm{tor}}$. 
Let $F$ be a torsion sheaf on $\overline{X}$
whose support is irreducible, and not 
contained in fibers of $\pi$. Then 
by the definition of $\aA_{\omega}$, it follows that  
\begin{align*}
\Hom(F, \hH^0(E)_{\rm{tor}}) \subset \Hom(F, E)=0. 
\end{align*}
Therefore $\hH^0(E)_{\rm{tor}}
 \in \Coh_{\pi}(\overline{X})$, 
hence $E_1' \in \fF_{\omega}$ 
by (\ref{F'=F}). 
As for $E_2'/E_1'$, because $\hH^0(E)_{\rm{fr}}$
is a torsion free sheaf of rank one, it is written as 
\begin{align}\label{LIZ}
\hH^0(E)_{\rm{fr}} \cong L\otimes I_Z, 
\end{align}
for a line bundle $L$ on $\overline{X}$ and a 
subscheme $Z \subset \overline{X}$ with $\dim Z \le 1$. 
Since $E\in \aA_{\omega}$,
the definition of $\aA_{\omega}$ yields that 
 $L \in \pi^{\ast} \Pic(\mathbb{P}^1)$
 and $Z$ is supported on fibers of $\pi$. 
Therefore we have $E_2'/E_1' \in \aA(r)$
for some $r \in \mathbb{Z}$. 
Finally by
 the definition of $\aA_{\omega}'$, we have 
$\hH^1(E) \in \tT_{\omega}$. 
By combining the filtration (\ref{filt123})
 with the exact sequence in 
$\aA_{\omega}$, 
\begin{align*}
0 \to T_1[-1] \to E/E_2' \to 
T_2[-1] \to 0,
\end{align*}
where $T_1 \in \Coh_{\pi}^{\le 1}(\overline{X})$ and 
$T_2 \in \tT_{\omega}^{\rm{pure}}$,
we obtain a desired filtration (\ref{filt:more}).  
\end{proof}

Another lemma we need is the following. 
\begin{lem}\label{lem:rank1}
For any object $E\in \aA_{\omega}$, 
there is 
 an exact sequence
in $\aA_{\omega}$, 
\begin{align}\label{BO}
0 \to A \to E \to B \to 0, 
\end{align}
such that $A \in \bB_{\omega}$ and 
$B\in \langle \pi^{\ast} \Pic(\mathbb{P}^1) \rangle_{\ex}$.
\end{lem}
\begin{proof}
Take an object $E \in \aA_{\omega}$. 
If $\rank(E)=0$, then $E\in \bB_{\omega}$ and 
the result is satisfied with $B=0$. 
If $\rank(E)>0$, then $E$ is written as a successive 
extensions of rank one objects.
Hence we may assume that $\rank(E)=1$. 

Suppose that $\rank(E)=1$. 
Below
we use the notation in the proof of Lemma~\ref{lem:rank10}.
As in (\ref{filt123}), we
can take a filtration 
$E_{\bullet}'$ of $E$ satisfying the condition 
(\ref{filt:sat}).
As in (\ref{LIZ}), 
the object $E_2'/E_1'$ is isomorphic to
$L\otimes I_Z$
for $L \in \pi^{\ast}\Pic(\mathbb{P}^1)$
and $Z \subset \overline{X}$
with $\dim Z \le 1$, 
contained in fibers of $\pi$. 
By combining the filtration (\ref{filt123}) with 
an exact sequence in $\aA_{\omega}$, 
\begin{align*}
0 \to L\otimes \oO_Z[-1] \to L \otimes I_Z \to L \to 0, 
\end{align*} 
we obtain a filtration
\begin{align}\label{filt123''}
E_1'' \subset E_2 '' \subset E_3'' =E, 
\end{align}
satisfying 
\begin{align*}
E_1'' \in \bB_{\omega}, \quad E_2''/E_1''
 \cong \pi^{\ast}\oO_{\mathbb{P}^1}(r), 
\quad
E/E_2'' \in \tT_{\omega}[-1]. 
\end{align*}
  We write $E/E_2''=A[-1]$ for $A \in \tT_{\omega}$, and take a filtration 
\begin{align*}
0=A_0 \subset A_1 \subset A_2 \subset \cdots \subset A_N=A, 
\end{align*}
such that each subquotient $B_i=A_i/A_{i-1}$
is $\omega$-Gieseker stable with 
\begin{align*}
\overline{\chi}_{B_1, \omega}(m) \succeq 
\overline{\chi}_{B_2, \omega}(m) \succeq \cdots
\succeq \overline{\chi}_{B_i, \omega}(m) \succeq 
\overline{\chi}_{B_{i+1}, \omega}(m) \succeq \cdots. 
\end{align*}
(See Subsection~\ref{subsec:classical}.)
We inductively replace the filtration (\ref{filt123''})
by another filtration 
\begin{align}\label{123i}
E_1^{(j)} \subset E_2^{(j)} \subset E_3^{(j)}=E, 
\end{align} 
satisfying 
\begin{align}\label{Ei}
E_1^{(j)} \in \bB_{\omega}, \quad 
E_2^{(j)}/E_1^{(j)} \in \pi^{\ast}\Pic(\mathbb{P}^1),
\quad 
E/E_2^{(j)} \cong \left(A/A_{j-1}\right)[-1].
\end{align}
A desired exact sequence (\ref{BO}) is obtained 
by putting $j=N+1$. 

When $j=1$, we can take a filtration (\ref{Ei}) 
to be (\ref{filt123''}). 
For $j\ge 1$, suppose that we have a filtration (\ref{123i})
satisfying (\ref{Ei}). We construct $E_2^{(j+1)}$ to be the
kernel of the 
composition of the surjections in $\aA_{\omega}$, 
\begin{align*}
E \twoheadrightarrow E/E_{2}^{(j)} =\left(A/A_{j-1}\right)
[-1] \twoheadrightarrow 
\left(A/A_j\right)[-1]. 
\end{align*}
Note that we have 
\begin{align}\label{isom:EE2A}
E/E_{2}^{(j+1)} \cong 
\left(A/A_{j}\right)[-1],
\end{align}
 by the construction. 

Next we construct $E_{1}^{(j+1)}$. 
By the diagram, 
\begin{align*}
\xymatrix{
  & E_2^{(j+1)} \ar[d] & B_j[-1] \ar[d] \\
E_2^{(j)} \ar[r] & E \ar[r] \ar[d] & A/A_{j-1}[-1] \ar[d] \\
  & A/A_{j}[-1] \ar[r]^{\id} & A/A_{j}[-1], }
\end{align*}
we have the exact sequence in $\aA_{\omega}$, 
\begin{align}\notag
0 \to E_2^{(j)} \to E_2^{(j+1)} \to B_j[-1] \to 0. 
\end{align}
Since $E_1^{(j)} \subset E_2^{(j)}$, 
we also have the exact sequence in $\aA_{\omega}$, 
\begin{align}\label{extj}
0 \to E_2^{(j)}/E_{1}^{(j)} \to E_2^{(j+1)}/E_1^{(j)} \to B_j[-1] \to 0. 
\end{align}
We denote by $\xi$ the extension class of (\ref{extj}). 
There are two cases: 

{\bf The case of $\xi=0$: }
In this case, we have a splitting surjection of (\ref{extj}), 
\begin{align*}
E_2^{(j+1)}/E_1^{(j)} \twoheadrightarrow E_2^{(j)}/E_1^{(j)}.
\end{align*}
We define $E_1^{(j+1)}$ to be the kernel of the composition 
\begin{align*}
E_2^{(j+1)} \twoheadrightarrow 
E_2^{(j+1)}/E_1^{(j)} \twoheadrightarrow E_2^{(j)}/E_1^{(j)}.
\end{align*}
Then we have the exact sequence in $\aA_{\omega}$, 
\begin{align*}
0 \to E_1^{(j)} \to E_1^{(j+1)} \to B_j[-1] \to 0.
\end{align*}
Hence $E_1^{(j+1)} \in \bB_{\omega}$. 
Noting (\ref{isom:EE2A}), 
the filtration  
$E_{\bullet}^{(j+1)}$ 
satisfies the condition (\ref{Ei}) for $j+1$. 

{\bf The case of $\xi \neq 0$: }
By the inductive assumption, $E_2^{(j)}/E_1^{(j)}$ is isomorphic 
to $\pi^{\ast}\oO_{\mathbb{P}^1}(r)$ for some 
$r \in \mathbb{Z}$. Also 
since $B_j$ is $\omega$-Gieseker stable, 
as in (\ref{ip}), 
there 
is $p\in \mathbb{P}^1$ and an $\omega$-Gieseker
stable sheaf $B_j'$ on $X_p$ such that 
$B_j \cong i_{p\ast}B_j'$. 
Hence the extension class $\xi$ lies in 
\begin{align}\notag
\xi &\in 
\Ext_{\overline{X}}^1(i_{p\ast}B_j'[-1], \pi^{\ast}\oO_{\mathbb{P}^1}(r)) \\
\notag
&\cong \Ext_{X_p}^2(B_j', i_{p}^{!}\oO_{\overline{X}}) \\
\label{isom:xiEx}
& \cong \Ext_{X_p}^1(B_j', \oO_{X_p}). 
\end{align}
Let 
\begin{align}\label{B''}
0 \to \oO_{X_p} \to B_j'' \to B_j' \to 0, 
\end{align}
be the extension in $X_p$ corresponding to $\xi$
via the isomorphism (\ref{isom:xiEx}). 
By Sublemma~\ref{below} below, 
we have
\begin{align}\label{sublem}
i_{p\ast}B_j''[-1] \in \tT_{\omega}[-1] \subset \bB_{\omega}. 
\end{align}
We have the commutative diagram, 
\begin{align*}
\xymatrix{
& E_2^{(j+1)}/E_1^{(j)} \ar[d]  & \pi^{\ast}\oO_{\mathbb{P}^1}(r+1)
\ar[d] \\	
i_{p\ast}B_j''[-1] \ar[r] & B_j[-1] \ar[r]\ar[d]_{\xi}
 & \oO_{X_p} \ar[d] \\
0 \ar[r] & \pi^{\ast}\oO_{\mathbb{P}^1}(r)[1] \ar[r]^{\id} & 
\pi^{\ast}\oO_{\mathbb{P}^1}(r)[1].
}
\end{align*}
By the above diagram and (\ref{sublem}), we obtain the 
exact sequence in $\aA_{\omega}$,  
\begin{align*}
0 \to i_{p\ast}B_j''[-1] \to E_2^{(j+1)}/E_{1}^{(j)} \to 
\pi^{\ast}\oO_{\mathbb{P}^1}(r+1) \to 0. 
\end{align*}
We construct $E_1^{(j+1)}$ to be the kernel of the 
composition of surjections in $\aA_{\omega}$, 
\begin{align*}
E_{2}^{(j+1)} \twoheadrightarrow E_2^{(j+1)}/E_1^{(j)} \twoheadrightarrow 
\pi^{\ast}\oO_{\mathbb{P}^1}(r+1). 
\end{align*}
Then we have the exact sequence in $\aA_{\omega}$, 
\begin{align*}
0 \to E_1^{(j)} \to E_1^{(j+1)} \to i_{p\ast}B_j''[-1] \to 0. 
\end{align*}
Therefore $E_1^{(j+1)} \in \bB_{\omega}$. 
Noting (\ref{isom:EE2A}), 
the filtration  
$E_{\bullet}^{(j+1)}$ 
satisfies the condition (\ref{Ei}) for $j+1$. 
\end{proof}
We have used the following sublemma. 
\begin{sublem}\label{below}
Let $B_j''$ be the sheaf on $X_p$ defined by 
(\ref{B''}). Then we have 
\begin{align*}
i_{p\ast}B_j'' \in \tT_{\omega}. 
\end{align*}
\end{sublem}
\begin{proof}
It is enough to show that 
\begin{align}\label{BF0}
\Hom(B_j'', F)=0, 
\end{align}
for any $\mu_{\omega}$-stable sheaf $F$
on $X_p$ with $\mu_{\omega}(F) \le 0$. 
Applying $\Hom(\ast, F)$ to the exact sequence (\ref{B''}), 
we have the exact sequence, 
\begin{align*}
\Hom(B_j', F) \to \Hom(B_j'', F) \to 
\Hom(\oO_{X_p}, F) \stackrel{\iota}{\to} \Ext_{X_p}^1(B_j', F). 
\end{align*}
Since $B_j'$ is $\mu_{\omega}$-stable with 
$\mu_{\omega}(B_j')>0$, we have 
$\Hom(B_j', F)=0$. 
Therefore by the above sequence, 
(\ref{BF0})
follows if $\Hom(\oO_{X_p}, F)=0$. 
Suppose that $\Hom(\oO_{X_p}, F)$ is non-zero. 
Then $F$ must be isomorphic to $\oO_{X_p}$, and 
under the isomorphism $F \cong \oO_{X_p}$, 
the image of $1$ under $\iota$ is the extension 
class corresponding to (\ref{B''}).
Since (\ref{B''}) does not split by the assumption, 
the map $\iota$ is injective. 
Hence (\ref{BF0}) follows. 
\end{proof}

\section{Results on weak stability conditions}\label{sec:reswe}
In this section, we recall some properties of 
weak stability conditions and complete a 
proof of Lemma~\ref{lem:property} in Subsection~\ref{subsec:property}. 
\subsection{Properties of weak stability conditions}
In this subsection, we recall some 
technical properties of weak stability conditions. 
We discuss in a general 
situation, 
and use the same notation 
in Subsection~\ref{subsec:general}. 
Let $(Z, \aA)$ be a weak stability condition on 
a triangulated category $\tT$. 
For $0<\phi \le 1$,
 the subcategory 
$\pP(\phi)\subset \tT$ is defined
 to be the category of $Z$-semistable
objects $E\in \aA$
satisfying 
\begin{align}\label{exp}
Z(E) \in \mathbb{R}_{>0}\exp(i\pi \phi). 
\end{align}
For other $\phi \in \mathbb{R}$, the subcategory 
$\pP(\phi)$ is determined by the rule, 
\begin{align*}
\pP(\phi+1)=\pP(\phi)[1]. 
\end{align*}
The family of subcategories 
$\pP(\phi)$ for $\phi \in \mathbb{R}$
determines a \textit{slicing} introduced in~\cite[Definition~3.3]{Brs1}. 
As in~\cite[Proposition~2.13]{Tcurve1}, 
giving a weak stability condition is equivalent 
to giving a data, 
\begin{align}\label{pasli}
\sigma=(Z=\{Z_i\}_{i=0}^N, \pP),
\end{align}
where $Z$ is
as in Definition~\ref{defi:weak}
 and $\pP$ is a slicing, 
satisfying the condition (\ref{exp})
for any non-zero $E\in \pP(\phi)$.
The subcategory $\pP(\phi) \subset \tT$
is called the category of \textit{$\sigma$-semistable objects
of phase $\phi$.} 

 For an interval 
$I\subset \mathbb{R}$, we set 
\begin{align*}
\pP(I)\cneq \langle \pP(\phi) : \phi \in I \rangle_{\ex}.
\end{align*}
The following properties are required 
in constructing the space $\Stab_{\Gamma_{\bullet}}(\tT)$. 
\begin{itemize}
\item {\bf (Support property):}
There is a constant $C>0$ such that 
for any $E\in \pP(\phi)$
with $\cl(E) \in \Gamma_i \setminus \Gamma_{i-1}$, we have 
\begin{align*}
\lVert [\cl(E)] \rVert_{i} \le C\cdot \lvert Z(E) \rvert. 
\end{align*}
Here $\lVert \ast \rVert_{i}$ is a fixed 
norm on $(\Gamma_i/\Gamma_{i-1}) \otimes_{\mathbb{Z}} \mathbb{R}$. 
\item {\bf (Local finiteness):} 
There is $\varepsilon>0$ such that the quasi-abelian category 
$\pP((\phi-\varepsilon, \phi+\varepsilon))$ is of finite length
for any $\phi \in \mathbb{R}$. 
\end{itemize}
Here we refer~\cite[Definition~4.1, Definition~5.7]{Brs1}
for the detail on the notion of quasi-abelian categories and 
their finite length property. 
The set $\Stab_{\Gamma_{\bullet}}(\tT)$
in Subsection~\ref{subsec:general}
is defined to be the set of weak stability 
conditions satisfying 
the above two properties.

\subsection{Proof of Lemma~\ref{lem:property}}\label{subsec:property}
In this subsection, we complete a proof of 
Lemma~\ref{lem:property}. Namely 
we prove the existence of Harder-Narasimhan filtrations, 
Support property and local finiteness
for the pair $(Z_{t\omega}, \aA_{\omega})$. 
We divide the proof into 4 steps. 
\begin{ssstep}\label{ss1}
The abelian category $\aA_{\omega}$ is noetherian. 
\end{ssstep}
\begin{proof}
Suppose that there is an infinite sequence of surjections 
in $\aA_{\omega}$, 
\begin{align}\label{sequ}
E_1 \twoheadrightarrow E_2 \twoheadrightarrow \cdots \twoheadrightarrow 
E_i \twoheadrightarrow E_{i+1} \twoheadrightarrow \cdots. 
\end{align}
We check that the sequence (\ref{sequ}) terminates. 
By Lemma~\ref{lem:pos}, we may assume that 
$\rank(E_i)$ and $\ch_2(E_i) \cdot \omega$
are constant. 
Also since we have surjections 
$\hH^1(E_i) \twoheadrightarrow \hH^1(E_{i+1})$
for all $i$, we may assume that $\hH^1(E_i) \cong \hH^1(E_{i+1})$
for all $i$. Let us take an exact sequence in $\aA_{\omega}$, 
\begin{align*}
0 \to K_i \to E_1 \to E_i \to 0. 
\end{align*}
We have the sequence of subsheaves, 
\begin{align*}
\hH^0(K_i) \subset \hH^0(K_{i+1}) \subset \cdots \subset \hH^0(E_1), 
\end{align*}
so we may assume that $\hH^0(K_i) \cong \hH^0(K_{i+1})$
for all $i$. 
Also since $\rank(K_i)=0$, we have $K_i \in \bB_{\omega}$. 
Furthermore since $\ch_2(K_i) \cdot \omega=0$, we have 
\begin{align*}
\dim \Supp \hH^1(K_i)=0. 
\end{align*}
Hence it is enough to bound the length of $\hH^1(K_i)$. 
Setting 
\begin{align*}
A=\hH^0(E_1)/\hH^0(K_1),
\end{align*} we have the exact sequence
of sheaves, 
\begin{align*}
0 \to A \to \hH^0(E_i) \to \hH^1(K_i) \to 0. 
\end{align*}
Let $A', \hH^0(E_i)'$ be the torsion parts 
and $A'', \hH^0(E_i)''$ the 
free parts of $A, \hH^0(E_i)$ respectively. 
We have the exact sequences of sheaves, 
\begin{align}
\label{AT1}
&0 \to A' \to \hH^0(E_i)' \to T_i' \to 0, \\
\label{AT2}
&0 \to A'' \to \hH^0(E_i)'' \to T_i'' \to 0, \\
\label{AT3}
&0 \to T_i' \to \hH^1(K_i) \to T_i'' \to 0,
\end{align}
where $T_i'$ and $T_i''$ are zero dimensional 
sheaves. 
By (\ref{AT1}) and (\ref{AT2}), we have the inclusions, 
\begin{align*}
T_i' \subset (A')^{\vee \vee}/A', \quad
T_i'' \subset (A'')^{\vee \vee}/A''. 
\end{align*}
Here for a pure two dimensional sheaf $F$, we set 
\begin{align*}
F^{\vee} \cneq \eE xt^1_{\overline{X}}(F, \oO_{\overline{X}}). 
\end{align*}
Therefore the length of $T_i'$ and $T_i''$ are bounded. 
By (\ref{AT3}), the length of $\hH^1(K_i)$ is also bounded. 
\end{proof}

\begin{ssstep}\label{ss2}
There exist Harder-Narasimhan filtrations 
for the pair $(Z_{t\omega}, \aA_{\omega})$. 
\end{ssstep}
\begin{proof}
By~\cite[Proposition~2.12]{Tcurve1}
and Step~\ref{ss1},
it is enough to check that there 
is no infinite sequence of subobjects in $\aA_{\omega}$, 
\begin{align}\label{seq:ar}
\cdots \subset E_{j+1} \subset E_j \subset \cdots \subset E_2 \subset E_1
\end{align}  
with $\arg Z_{t\omega}(E_{j+1})> \arg Z_{t\omega}(E_{j}/E_{j+1})$
for all $j$. Suppose that such a sequence exists. 
By Lemma~\ref{lem:pos}, we may assume that 
$\rank(E_j)$ and $\ch_2(E_j) \cdot \omega$ are constant, 
hence 
\begin{align*}
\rank(E_j/E_{j+1})=0, \quad \ch_2(E_j/E_{j+1}) \cdot \omega =0,
\end{align*}
for all $j$. By the definition of $\aA_{\omega}$
and $Z_{t\omega}$, the above condition is equivalent 
to $Z_{t\omega}(E_j/E_{j+1}) \in \mathbb{R}_{<0}$. 
This contradicts to 
$\arg Z_{t\omega}(E_{j+1})> \arg Z_{t\omega}(E_{j}/E_{j+1})$, 
hence there is no such a sequence. 
\end{proof}

\begin{ssstep}
The pair $(Z_{t\omega}, \aA_{\omega})$ satisfies 
the support property. 
\end{ssstep}
\begin{proof}
Let $E \in \aA_{\omega}$ be a $Z_{t\omega}$-semistable 
object
with $\cl(E)=(R, r, \beta, n)$. If $R\neq 0$, we have 
\begin{align*}
\frac{\lVert [\cl(E)] \rVert}{\lvert Z_{t\omega}(E) \rvert}=1. 
\end{align*}
If $R=0$, then $E \in \bB_{\omega}$ 
and $E$ is a $Z_{t\omega, 0}$-semistable 
object. (cf.~Remark~\ref{rmk:BA}.)
Hence the support property for such $E$ 
follows from that of the pair 
\begin{align}\label{pair}
(Z_{t\omega, 0}, \bB_{\omega}) \in \Stab_{\Gamma_0}(\dD_0).
\end{align}
The support property of the pair
(\ref{pair}) follows from the same 
argument for the surface case. (cf.~\cite[Section~4]{BaMa}.)
\end{proof}
\begin{sssstep}
The pair $(Z_{t\omega}, \aA_{\omega})$ satisfies 
the local finiteness. 
\end{sssstep}
\begin{proof}
For a pair $(Z_{t\omega}, \aA_{\omega})$, let 
$\{\pP(\phi)\}_{\phi \in \mathbb{R}}$ be
the corresponding slicing. 
For each $\phi \in \mathbb{R}$, we need to 
find $\varepsilon>0$ so that 
$\pP((\phi-\varepsilon, \phi+\varepsilon))$
is of finite length. 
Note that if $\phi \notin 1/2 + \mathbb{Z}$, then 
any $E \in \pP(\phi)$ is a semistable object
with respect to the pair (\ref{pair}). 
Hence the local finiteness in this 
case follows from that of (\ref{pair}), 
which can be proved along with the same 
argument for the surface case. (cf.~\cite[Lemma~4.4]{Brs2}.)
Suppose that $\phi \in 1/2 + \mathbb{Z}$. 
We may assume $\phi=1/2$. In this case, 
it is enough to show that $\pP((0, 1))$ is 
of finite length. 
Since $\pP((0, 1))$ is a subcategory of $\aA_{\omega}$, 
$\pP((0, 1))$ is noetherian by Step~\ref{ss1}. 
The proof that $\pP((0, 1))$ is artinian 
follows from the same argument
in Step~\ref{ss2} that there are no infinite sequence (\ref{seq:ar}). 
\end{proof}

\section{Results on semistable objects}\label{sec:semistable}
In this section, we give proofs of several
results on semistable objects in $\aA_{\omega}$. 
In particular we prove some of the results stated in
Section~\ref{section:wstab}: 
Proposition~\ref{prop:wall} in Subsection~\ref{subsec:wall}, 
Proposition~\ref{prop:m2} in Subsection~\ref{subsec:finset2}, 
Lemma~\ref{rmk:lim} in Subsection~\ref{subsec:mulim}
and 
Proposition~\ref{prop:m1} in Subsection~\ref{subsec:finset}. 
 
\subsection{Duality of semistable objects}
In this subsection, we discuss 
a duality of $Z_{t\omega}$-semistable 
objects in $\aA_{\omega}$. 
For an object $E \in \dD$, note that 
\begin{align*}
\mathbb{D}(E) \cneq \dR \hH om_{\overline{X}}
(E, \oO_{\overline{X}}) \in \dD. 
\end{align*}
Also note that 
$\aA_{\omega}$ contains the following subcategory, 
\begin{align*}
\cC_{\omega} \cneq \left\langle F, \oO_x[-1] :
F \in \fF_{\omega} \mbox{ with } 
\mu_{\omega}(F)=0, \
 x\in \overline{X} 
\right\rangle_{\ex}. 
\end{align*}
We have the following lemma:
\begin{lem}\label{lem:equiv:C}
We have the autoequivalence, 
\begin{align*}
\mathbb{D} \circ [1] \colon \cC_{\omega} 
\stackrel{\sim}{\to} \cC_{\omega}. 
\end{align*}
\end{lem}
\begin{proof}
It is enough to show that 
\begin{align}\label{duality:cond1}
\mathbb{D}(\oO_x[-1])[1] &\in \cC_{\omega}, \\
\label{duality:cond2}
\mathbb{D}(F)[1] &\in \cC_{\omega}, 
\end{align}
for $x \in \overline{X}$ and 
 a $\mu_{\omega}$-stable sheaf 
$F \in \fF_{\omega}$
with $\mu_{\omega}(F)=0$.
The condition (\ref{duality:cond1}) follows
from $\mathbb{D}(\oO_x)=\oO_x[-3]$. 
For the sheaf $F$ as above, we write 
$F=i_{p\ast}F'$ for a $\mu_{\omega}$-stable 
sheaf on $X_p$ as in (\ref{ip}). 
We have the distinguished triangle, 
\begin{align*}
Q[-1] \to F \to i_{p\ast}F^{'\vee \vee},
\end{align*} 
for some zero dimensional sheaf $Q$. 
Note that $F^{'\vee \vee}$ is a locally 
free sheaf on $X_p$. 
Then the condition (\ref{duality:cond2})
follows from (\ref{duality:cond1}) and 
the fact that 
\begin{align*}
\mathbb{D}(i_{p\ast}F^{'\vee \vee})[1] 
\cong i_{p\ast} F^{'\vee}, 
\end{align*} 
which is $\mu_{\omega}$-stable with 
$\mu_{\omega}=0$. 
\end{proof}
We also consider the 
right orthogonal complement 
of $\cC_{\omega}$,  
\begin{align}\label{ROCom}
\cC_{\omega}^{\perp}\cneq \{ E \in \aA_{\omega} : 
\Hom(\cC_{\omega}, E)=0 \}. 
\end{align}
The following result implies that $\cC_{\omega}^{\perp}$ is 
also self dual. 
\begin{lem}\label{lem:dual}
We have the autoequivalence, 
\begin{align*}
\mathbb{D} \colon \cC_{\omega}^{\perp}
\stackrel{\sim}{\to} \cC_{\omega}^{\perp}. 
\end{align*}
\end{lem}
\begin{proof}
For an object $E \in \cC_{\omega}^{\perp}$
and $K \in \cC_{\omega}$, 
we have  
\begin{align*}
\Hom(K, \mathbb{D}(E)) &\cong 
\Hom(E, \mathbb{D}(K)[1][-1]) \\
& \cong 0,
\end{align*}
since 
$\mathbb{D}(K)[1] \in \cC_{\omega}$
by Lemma~\ref{lem:equiv:C}. 
Therefore it is enough to show that 
$\mathbb{D}(E) \in \aA_{\omega}$. 
By Lemma~\ref{lem:rank1}, 
there is an exact sequence in $\aA_{\omega}$
\begin{align}\label{seq:AEB}
0 \to A \to E \to B \to 0, 
\end{align}
such that $A \in \bB_{\omega}$ and 
$B\in \langle\pi^{\ast}\Pic(\mathbb{P}^1)\rangle_{\ex}$. 
Since $\mathbb{D}(B) \in \aA_{\omega}$, 
it is enough to show that 
\begin{align*}
\mathbb{D}(A) \in \bB_{\omega}. 
\end{align*}
We take an exact sequence in $\bB_{\omega}$, 
\begin{align*}
0 \to F \to A \to T[-1] \to 0, 
\end{align*}
with $F \in \fF_{\omega}$ and $T \in \tT_{\omega}$. 
Since $F$ is a subobject of $E$ in $\aA_{\omega}$
and $E \in \cC_{\omega}^{\perp}$, we have 
$F \in \cC_{\omega}^{\perp}$. 
Taking the dual of the above sequence, we obtain 
the distinguished triangle, 
\begin{align*}
\mathbb{D}(T)[1] \to \mathbb{D}(A) \to \mathbb{D}(F).  
\end{align*}
By Lemma~\ref{lem:sub} below
and the long exact sequence of cohomologies,  
we have 
\begin{align*}
&\hH^i \mathbb{D}(A) =0, \ i\neq 0, 1, 2, \\
&\hH^0\mathbb{D}(A) \cong \eE xt_{\overline{X}}^1(T, \oO_{\overline{X}}), \\
&\dim \hH^2 \mathbb{D}(A) =0,
\end{align*}
and 
the exact sequence of sheaves, 
\begin{align*}
0 \to \eE xt_{\overline{X}}^2(T, \oO_{\overline{X}})
\to \hH^1 \mathbb{D}(A) \to \eE xt_{\overline{X}}^1(F, \oO_{\overline{X}})
\to Q \to 0,
\end{align*}
for some zero dimensional sheaf $Q$. 
Applying Lemma~\ref{lem:sub} again, we have 
\begin{align*}
\hH^0 \mathbb{D}(A) \in \fF_{\omega}, \ 
\hH^1 \mathbb{D}(A) \in \tT_{\omega}. 
\end{align*}
Suppose that $\hH^2 \mathbb{D}(A) \neq 0$. 
Then there is $x \in \overline{X}$ such that 
\begin{align*}
\Hom(\mathbb{D}(A), \oO_x[-2]) \neq 0. 
\end{align*}
Applying $\mathbb{D}$, we have 
\begin{align*}
\Hom(\oO_x[-1], A) \neq 0, 
\end{align*}
which contradicts to $E \in \cC_{\omega}^{\perp}$. 
\end{proof}
We have used the following lemma. 
\begin{lem}\label{lem:sub}
(i) For $F\in \cC_{\omega}^{\perp} \cap \fF_{\omega}$, we have 
\begin{align}\label{sub1}
&\eE xt^{i}_{\overline{X}}(F, \oO_{\overline{X}})=0, \quad 
i\neq 1, 2, \\
\label{sub2}
&\eE xt^{1}_{\overline{X}}(F, \oO_{\overline{X}}) \in \tT_{\omega}, \\
\label{sub3}
&\dim \eE xt_{\overline{X}}^2(F, \oO_{\overline{X}})=0. 
\end{align}
(ii) For $T \in \tT_{\omega}$, we have 
\begin{align}\label{sub4}
&\eE xt^{i}_{\overline{X}}(T, \oO_{\overline{X}})=0, \quad 
i\neq 1, 2, 3, \\
\label{sub5}
&\eE xt^{1}_{\overline{X}}(T, \oO_{\overline{X}}) \in \fF_{\omega}, \\
\label{sub6}
&\dim \eE xt_{\overline{X}}^i(T, \oO_{\overline{X}})=3-i, \ i=2, 3.  
\end{align}
\end{lem}
\begin{proof}
The properties (\ref{sub1}), (\ref{sub3}), (\ref{sub4}), (\ref{sub6})
are well-known and the proofs are standard. See~\cite{Hu} for instance. 
We show the property 
(\ref{sub2}). The property (\ref{sub5}) is similarly proved.
Let us take $F\in \cC_{\omega}^{\perp} \cap \fF_{\omega}$. 
By taking Harder-Narasimhan filtration and Jordan-H$\ddot{\rm{o}}$lder
filtration with respect to $\mu_{\omega}$-stability, we
may assume that $F \cong i_{p\ast}F'$ for 
some $\mu_{\omega}$-stable sheaf $F'$ on $X_p$
as in (\ref{ip}). 
The condition $F \in \cC_{\omega}^{\perp}$ implies that
$\mu_{\omega}(F')<0$. 
Then by the adjunction, we have 
\begin{align*}
\eE xt_{\overline{X}}^1(F, \oO_{\overline{X}}) & \cong 
i_{p\ast}\hH om_{X_p}(F', \oO_{X_p}) \\
& \cong i_{p\ast}F^{'\vee}. 
\end{align*} 
Since $F^{'\vee}$ is $\mu_{\omega}$-stable 
with $\mu_{\omega}(F^{'\vee})>0$, 
we have $i_{p\ast}F^{'\vee} \in \tT_{\omega}$.  
\end{proof}
In order to see the duality of 
semistable objects, we show 
the following lemma. 
\begin{lem}\label{lem:cri}
An object $E\in \aA_{\omega}$ 
with $\Imm Z_{t\omega}(E)>0$ is $Z_{t\omega}$-semistable 
if and only if 
$E \in \cC_{\omega}^{\perp}$ and
for any exact sequence in $\aA_{\omega}$
\begin{align}\label{seqC}
0 \to F \to E \to G \to 0
\end{align}
with $F, G \in \cC_{\omega}^{\perp}$, the 
inequality 
\begin{align}\label{ineq:C}
\arg Z_{t\omega}(F) \le \arg Z_{t\omega}(G)
\end{align}
is satisfied. 
\end{lem}
\begin{proof}
Take $E\in \aA_{\omega}$ with 
$\Imm Z_{t\omega}(E)>0$, and 
suppose that $E$ is $Z_{t\omega}$-semistable. 
Since $\arg Z_{t\omega}(E)<\pi$ and 
\begin{align}\label{Cle}
Z_{t\omega}(\cC_{\omega}) \subset \mathbb{R}_{\le 0},
\end{align}
we have $E\in \cC_{\omega}^{\perp}$
by the $Z_{t\omega}$-semistability of $E$.  
The inequality 
(\ref{ineq:C})
with respect to the sequence (\ref{seqC}) 
follows from the $Z_{t\omega}$-semistability of $E$. 

Conversely, suppose that $E \in \cC_{\omega}$
satisfies the inequality (\ref{ineq:C})
with respect to 
any sequence (\ref{seqC}). 
We take an exact sequence in $\aA_{\omega}$, 
\begin{align*}
0 \to F' \to E \to G' \to 0. 
\end{align*}
Since $\aA_{\omega}$ is 
noetherian, (see Subsection~\ref{subsec:property},)
there is an exact sequence in $\aA_{\omega}$, 
\begin{align*}
0 \to G''' \to G' \to G'' \to 0, 
\end{align*}
with $G''' \in \cC_{\omega}$ and $G'' \in \cC_{\omega}^{\perp}$. 
By composing the above sequences, we obtain the
exact sequence in $\aA_{\omega}$, 
\begin{align*}
0 \to F'' \to E \to G'' \to 0,
\end{align*}
with $F'' , G'' \in \cC_{\omega}^{\perp}$. 
Using the assumption and (\ref{Cle}), we obtain 
\begin{align*}
\arg Z_{t\omega}(F') &\le \arg Z_{t\omega}(F'') \\
& \le \arg Z_{t\omega}(G'') \\
&\le \arg Z_{t\omega}(G'). 
\end{align*}
Hence $E$ is $Z_{t\omega}$-semistable. 
\end{proof}
Summarizing
the above results, we obtain the following result. 
\begin{prop}\label{prop:bij}
Suppose that $R\ge 1$ or $R=0$, $\beta \cdot \omega \neq 0$. 
Then we have the bijection, 
\begin{align}\label{bij1}
\mathbb{D}: M_{t\omega}(R, r, \beta, n) 
\stackrel{1:1}{\to} M_{t\omega}(R, -r, \beta, -n), 
\end{align}
If $R=\beta \cdot \omega =0$, we have 
\begin{align}\label{bij2}
\mathbb{D} \circ [1] \colon 
M_{t\omega}(0, r, \beta, n) \stackrel{1:1}{\to}
M_{t\omega}(0, r, -\beta, n). 
\end{align}
\end{prop}
\begin{proof}
Take an object $E\in M_{t\omega}(R, r, \beta, n)$
and suppose that $R\ge 1$ or $R=0$, $\beta \cdot \omega \neq 0$. 
Then $\Imm Z_{t\omega}(E)>0$, hence 
 noting Lemma~\ref{lem:dual}, Lemma~\ref{lem:cri} and 
\begin{align*}
Z_{\omega}(\mathbb{D}(E))=-\overline{Z}_{\omega}(E),
\end{align*}
we easily see that $\mathbb{D}(E)$ is a
$Z_{t\omega}$-semistable object in $\aA_{\omega}$. 
Therefore (\ref{bij1}) follows. 
If $R=\beta \cdot \omega =0$, 
then we have $E\in \cC_{\omega}$
and (\ref{bij2}) follows 
from Lemma~\ref{lem:equiv:C}. 
\end{proof}
By applying the dualizing functor, we
can also prove the following lemma. 
\begin{lem}\label{lem:Ostable}
For any $r\in \mathbb{Z}$, there is no 
non-trivial exact sequence in $\aA_{\omega}$, 
\begin{align}\label{no-nontriv}
0 \to A \to \pi^{\ast}\oO_{\mathbb{P}^1}(r) \to B \to 0, 
\end{align}
with $A, B \in \cC_{\omega}^{\perp}$. 
In particular, 
the object $\pi^{\ast}\oO_{\mathbb{P}^1}(r) \in \aA_{\omega}$
is $Z_{t\omega}$-stable for any $t\in \mathbb{R}_{>0}$. 
\end{lem}
\begin{proof}
Since $\pi^{\ast}\oO_{\mathbb{P}^1}(r) \in \cC_{\omega}^{\perp}$, 
the $Z_{t\omega}$-stability of $\pi^{\ast}\oO_{\mathbb{P}^1}(r)$
follows from the first statement and Lemma~\ref{lem:cri}. 
Suppose that a non-trivial 
sequence (\ref{no-nontriv}) exists. 
Then we have $\rank(A)=0$ or $\rank(B)=0$,
and by the duality in Lemma~\ref{lem:dual}, we may 
assume that $\rank(B)=0$, i.e. $B\in \bB_{\omega}$. 
Then by Lemma~\ref{twolem} (iii), (iv), 
the object $B$ is written as $\pi^{\ast}Q$
for a zero dimensional sheaf $Q$ on $\mathbb{P}^1$.
Since $\pi^{\ast}Q \in \cC_{\omega}$, 
this contradicts to $B \in \cC_{\omega}^{\perp}$. 
\end{proof}

\subsection{Proof of Proposition~\ref{prop:wall}}\label{subsec:wall}
In this subsection, we give a proof of Proposition~\ref{prop:wall}, 
that is the existence of wall and chamber structure on $t\in \mathbb{R}_{>0}$.
For the reader's convenience, we restate the proposition. 
\begin{prop}\label{prop:wall2}
For fixed $\beta \in \mathrm{NS}(S)$ and 
an ample divisor $\omega$ on $S$, there is a finite
sequence of real numbers, 
\begin{align*}
0=t_0 <t_1 < \cdots <t_{k-1} <t_k=\infty,
\end{align*}
such that the set of objects 
\begin{align*}
\bigcup_{\begin{subarray}{c} (R, r, n), \\
\arg Z_{t\omega}(R, r, \beta, n)=\pi/2
\end{subarray}}
M_{t\omega}(R, r, \beta, n),
\end{align*} 
is constant for each $t \in (t_{i-1}, t_i)$. 
\end{prop}
\begin{proof}
We fix $\beta$, $\omega$ and 
take an object, 
\begin{align}\label{objR}
E \in \bigcup_{\begin{subarray}{c} (R, r, n), \\
\arg Z_{t\omega}(R, r, \beta, n)=\pi/2
\end{subarray}}M_{t\omega}(R, r, \beta, n).
\end{align} 
Suppose that $A \in \bB_{\omega}$ is a subobject or a quotient of 
$E$ in $\aA_{\omega}$ and satisfies 
\begin{align}\label{condA}
\arg Z_{t\omega}(A)=\frac{\pi}{2}. 
\end{align}
If we
 write $\cl_0(A)=(r', \beta', n')$, 
then we have 
\begin{align}\label{eq:wall}
\Ree Z_{t\omega}(A)=n'-\frac{1}{2}r't^2 \omega^2=0. 
\end{align}
By Lemma~\ref{lem:rank10}, 
there is a filtration in $\bB_{\omega}$, 
\begin{align}\label{filt:w1}
0=A_0 \subset A_1 \subset A_2 \subset A_3=A,
\end{align}
such that each subquotient $K_i \cneq A_i/A_{i-1}$
satisfies the condition (\ref{filt:filt}).
We write $\cl_0(K_i)=(r_i, \beta_i, n_i)$. 
By the $Z_{t\omega}$-semistability of $E$
and the condition (\ref{condA}),
we have 
\begin{align}\label{wall2}
\Ree Z_{t\omega}(K_1) &= n_1 -\frac{1}{2}r_1 t^2 \omega^2 \ge 0, \\
\label{wall22}
\Ree Z_{t\omega}(K_3) &= n_3 -\frac{1}{2}r_3 t^2 \omega^2 \le 0. 
\end{align} 
Since $r_1\ge0$ and $r_3 \le0$, 
the inequalities (\ref{wall2}), (\ref{wall22})
imply that $n_1 \ge 0$ and $n_3 \le 0$. 
Also by Lemma~\ref{lem:pos}, 
we have $\beta \cdot \omega \le \beta_i \cdot \omega \le 0$.  
Therefore we can apply Lemma~\ref{finset} below and 
conclude that $(r_1, n_1)$,
 $(r_3, n_3)$,  
hence $r'=r_1 +r_3$, have only
a finite number of possibilities. 

Suppose that $K_1=K_3=0$. 
Then the equality (\ref{eq:wall}) is satisfied 
only if $(r', n')=(0, 0)$. 
Otherwise, for instance if $K_1 \neq 0$, then 
$r_1 >0$ and 
the inequality (\ref{wall2})
implies that 
\begin{align*}
n_1 -\frac{1}{2}t^2 \omega^2 r_1 \ge 0. 
\end{align*} 
Therefore such $t$ is bounded above.
A similar argument shows the boundedness of $t$
under the assumption $K_3 \neq 0$. 
Therefore the set of possible 
$t \in \mathbb{R}$ 
satisfying the equation (\ref{eq:wall})
for some $A \in \bB_{\omega}$, 
which is a subobject or a quotient 
of some object (\ref{objR})
with $(r', n') \neq 0$, is a finite set. 
If we denote this finite set 
by $0=t_0<t_1< \cdots <t_k<\infty$, 
then $t_{\bullet}$ satisfies the desired condition. 
\end{proof}
We have used the following lemma.
\begin{lem}\label{finset}
For fixed ample divisor
$\omega$ on $S$ and 
$a, b \in \mathbb{R}$, 
the following 
subsets in $\mathbb{Z}^{\oplus 2}$ are finite sets:
\begin{align}\label{fs1}
&
\left\{ (r', n') : \begin{array}{l}
\mbox{ there is }
T \in \tT_{\omega}^{\rm{pure}}, 
\ \cl_0(T)=(r', \beta', n'), \\
\mbox{ satisfying }
\beta' \cdot \omega \le a, \ 
n' \ge b
\end{array} \right\},  \\
\label{fs2}
&
\left\{ (r', n') : \begin{array}{l}
\mbox{ there is }
F \in \fF_{\omega}, 
\ \cl_0(F)=(r', \beta', n'), \\
\mbox{ satisfying }
\beta' \cdot \omega \ge a, \ 
n' \ge b
\end{array} \right\}.
\end{align}
\end{lem}
\begin{proof}
For simplicity, we prove the 
finiteness of (\ref{fs1}). 
The finiteness of (\ref{fs2}) is similarly proved. 
Take $T\in \tT_{\omega}^{\rm{pure}}$ 
with $\cl_0(T)=(r', \beta', n')$, 
$\beta' \cdot \omega \le a$
and $n' \ge b$. 
Taking the Harder-Narasimhan filtrations and 
Jordan-H$\ddot{\rm{o}}$lder filtrations 
of $T$ with respect to $\mu_{\omega}$-stability, 
we have a filtration of coherent sheaves,
\begin{align*}
0=T_0 \subset T_1 \subset \cdots \subset T_N=T, 
\end{align*}
such that each $M_i \cneq T_i/T_{i-1}$
is $\mu_{\omega}$-stable. 
We write $\cl_0(M_i)=(r_i, \beta_i, n_i)$. 
Since $\beta_i \cdot \omega >0$
for all $i$, 
we have 
\begin{align*}
\beta_i \cdot \omega \le \beta' \cdot \omega 
\le a, \quad
0<N\le a.
\end{align*}
By the Hodge index theorem, there is 
a constant $s(a, \omega)>0$
which depends only on $a$ and $\omega$ such that 
\begin{align}\notag
\beta_i^2 \le s(a, \omega). 
\end{align}
Also note that $r_i>0$ for all $i$, 
since $T$ is a pure two dimensional sheaf. 
Therefore applying Lemma~\ref{lem:Bog}, we have 
\begin{align}\notag
n_i &\le \frac{\beta_i^2 +2}{2r_i} -r_i \\
\notag
&\le \frac{1}{2} \left(s(a, \omega)+2  \right) -r_i.
\end{align}
Taking the sum from $i=1$ to $i=N$, we obtain 
\begin{align}\label{n7}
n'  &\le \frac{N}{2}(s(a, \omega)+2)-r' \\
\notag
&\le \frac{a}{2}(s(a, \omega)+2)-r'. 
\end{align}
Combined with $n' \ge b$,
we have 
\begin{align*}
0< r'
\le \frac{a}{2}(s(a, \omega)+2)-b. 
\end{align*}
Therefore there is only a finite number of possibilities 
for $r'$. 
By (\ref{n7}) and $n'\ge b$,
there is also a finite number of possibilities for $n'$. 
\end{proof}

\subsection{Proof of Proposition~\ref{prop:m2}}\label{subsec:finset2}
In this subsection, we prove Proposition~\ref{prop:m2}, 
which is restated as follows: 
\begin{prop}\label{prop:m23}
In the same situation of Proposition~\ref{prop:wall2},
we have  
\begin{align*}
M_{t\omega}(R, r, \beta, n)=\emptyset,  
\end{align*}
for any $t\in (0, t_1)$ and $(R, r, n) \in \mathbb{Z}^{\oplus 3}$
 with $R \ge 1$ and $n\neq 0$. 
\end{prop}
\begin{proof}
Suppose that 
there is an object $E\in M_{t\omega}(R, r, \beta, n)$
for $t\in (0, t_1)$ and $R \ge 1$. 
By Proposition~\ref{prop:bij}, 
we may assume that $n<0$. 
By Lemma~\ref{lem:rank1}, 
there is an exact sequence in $\aA_{\omega}$, 
\begin{align*}
0 \to A \to E \to B \to 0, 
\end{align*}
such that $A \in \bB_{\omega}$ and $B\in \langle\pi^{\ast}\Pic(\mathbb{P}^1) \rangle_{\ex}$.
We have 
\begin{align*}
\cl_0(A)=(r', \beta, n),
\end{align*}
for some $r' \in \mathbb{Z}$. 
By the $Z_{t\omega}$-semistability of $E$, 
we have $\arg Z_{t\omega}(A)\le \pi/2$, or 
equivalently 
\begin{align}
n-\frac{1}{2}t^2 \omega^2 r' \ge 0. 
\end{align}
The above inequality should be satisfied 
for any $t\in (0, t_1)$, 
therefore we must have $n\ge 0$. 
This contradicts to $n<0$, 
hence we have $M_{t\omega}(R, r, \beta, n)=\emptyset$
for $t\in (0, t_1)$, $R \ge 1$ and $n\neq 0$. 
\end{proof}

\subsection{Rank zero semistable objects for small $t$}
Using the technique in the previous subsections, 
we give the following proposition 
on rank zero $Z_{t\omega}$-semistable
objects for small $t$. 
This result will be used later. 
\begin{prop}\label{prop:rank0}
For fixed $\beta$ and $\omega$
with $\beta \cdot \omega>0$, there is 
$t'>0$ such that 
the following set of objects 
is constant for $0<t<t'$, 
\begin{align*}
\bigcup_{r \in \mathbb{Z}}M_{t\omega}(0, r, \beta, 0). 
\end{align*}
\end{prop}
\begin{proof}
The proof is similar to the proof of Proposition~\ref{prop:wall2}, 
but we need to modify the argument in some places. 
Let us take $E \in M_{t\omega}(0, r, \beta, 0)$
for $t\in \mathbb{R}_{>0}$. By Lemma~\ref{bound:rank0} below, 
we have a finite number of possibilities for $r$. 
Hence we can take $t''>0$ such that 
\begin{align}\label{ReE}
\lvert \Ree Z_{t\omega}(E) \rvert
=\frac{1}{2}t^2 r^2 \omega^2 <1,
\end{align}
for any $E\in M_{t\omega}(0, r, \beta, 0)$
with $0<t<t''$. 
Take an object $A \in \bB_{\omega}$ 
such that $A$ is a subobject or quotient of 
$E$ in $\bB_{\omega}$ and satisfies  
\begin{align}\label{ReA}
\arg Z_{t\omega}(A)=\arg Z_{t\omega}(E),
\end{align}
for some $0<t<t''$. 
By Lemma~\ref{lem:rank10}, there is a filtration 
\begin{align*}
0=A_0 \subset A_1 \subset A_2 \subset A_3=A,
\end{align*}
in $\bB_{\omega}$
such that 
$K_i=A_i/A_{i-1}$ satisfies the condition (\ref{filt:filt}).
We
 write $\cl_0(K_i)=(r_i, \beta_i, n_i)$. 
By the $Z_{t\omega}$-semistability of $E$ 
and the
inequality (\ref{ReE}), we can easily see that
$n_1$ is bounded below and $n_3$ is bounded above. 
Hence Lemma~\ref{finset} implies that there are only finite 
number of possibilities for 
$(r_1, n_1)$ and 
$(r_3, n_3)$. 

Now we note 
\begin{align}\label{A1}
\lvert \Ree Z_{t\omega}(A) \rvert=
\left\lvert n'-\frac{1}{2}r't^2 \omega^2 \right\rvert <1, 
\end{align}
by the conditions (\ref{ReE}) and (\ref{ReA}).
Since $r'=r_1 +r_3$ is bounded, 
the inequality (\ref{A1}) gives a 
lower bound
of $t>0$ for the existence of such object 
$A \in \bB_{\omega}$ with $r' \neq 0$. 
If we denote that lower bound by $t'$, 
then $t'$ satisfies the desired condition. 
Note that if $r'=0$, then 
(\ref{A1}) is only possible when $n'=0$. 
However in that case $\arg Z_{t\omega}(A)=\pi/2$
for any $t$, and we don't need to take account of 
such objects. 
\end{proof}

\begin{lem}\label{bound:rank0}
For fixed 
$a \in \mathbb{R}_{>0}$
and an ample divisor $\omega$ on $S$, 
the set of $r \in \mathbb{Z}$ such 
that $M_{t\omega}(0, r, \beta, 0)\neq \emptyset$
for some $t\in \mathbb{R}_{>0}$
and $0<-\beta \cdot \omega \le a$
is a finite set. 
\end{lem}
\begin{proof}
By Proposition~\ref{prop:bij}, 
it is enough to consider possible values 
$r \in \mathbb{Z}$ with $r<0$. 
Let us take $E \in M_{t\omega}(0, r, \beta, 0)$
with $r<0$.
By Lemma~\ref{lem:rank10}, there is a filtration 
\begin{align*}
0=E_0 \subset E_1 \subset E_2 \subset E_3=E,
\end{align*}
in $B_{\omega}$
such that $K_i=E_i/E_{i-1}$ satisfies 
the condition (\ref{filt:filt}). 
We write 
$\cl_0(K_i)=(r_i, \beta_i, n_i)$. 
Since $r<0$, we have 
$\arg Z_{t\omega}(E) \in (0, \pi/2)$. 
By the $Z_{t\omega}$-semistability of $E$, we have 
\begin{align*}
\arg Z_{t\omega}(E_i) \le \arg Z_{t\omega}(E) <\frac{\pi}{2}, 
\end{align*}
for $i=1, 2$. 
Hence we have $n_1 \ge 0$ and $n_1+n_2 \ge 0$, 
therefore $n_3 =-(n_1 +n_2) \le 0$. 
Since $-\beta_i \cdot \omega \le a$, 
we can apply Lemma~\ref{finset}
and conclude that  
$r_1$ and $r_3$ are bounded.
Hence $r=r_1 +r_3$ is also bounded. 
\end{proof}

\subsection{Proof of Lemma~\ref{rmk:lim}}\label{subsec:mulim}
In this subsection, we prove Lemma~\ref{rmk:lim}, which 
is restated as follows: 
\begin{lem}\label{rmk:lim2}
Take an object $E\in D^b \Coh(\overline{X})$
satisfying 
\begin{align}\label{cond:chern}
\ch(E)=(R, 0, -\beta, -n) \in \Gamma \subset H^{\ast}(\overline{X}, 
\mathbb{Q}),
\end{align}
for $R \le 1$. 
Then $E$ 
is an $\mu_{i\omega}$-limit semistable
object in $\aA(0)$  
iff $E[1]$ is an $\mu_{i\omega}$-limit semistable 
object 
in the sense of~\cite[Section~3]{Tolim2}. 
\end{lem}
\begin{proof}
We only show the case of $R=1$. 
The proof for the case of $R=0$ case is 
easier and we omit it. 
\begin{ssssstep}
The definition of $\mu_{i\omega}$-limit stability 
in~\cite[Section~3]{Tolim2}.
\end{ssssstep}
We first recall the notion of $\mu_{i\omega}$-limit
stability in the sense of~\cite[Section~3]{Tolim2}.   
In~\cite{Bay}, \cite{Tolim}, 
the notion of \textit{polynomial stability} and \textit{limit stability}
are defined on the following category of 
\textit{perverse coherent sheaves}, 
\begin{align*}
\aA^p \cneq \langle 
\Coh^{\ge 2}(\overline{X})[1], \Coh^{\le 1}(\overline{X}) \rangle_{\ex}
\subset D^b \Coh(\overline{X}).  
\end{align*}
Here $\Coh^{\le 1}(\overline{X})$
consists of sheaves $F$ on $\overline{X}$ 
with $\dim F \le 1$ and 
$\Coh^{\ge 2}(\overline{X})$ is the right
orthogonal complement of $\Coh^{\le 1}(\overline{X})$
in $\Coh(\overline{X})$. 

By~\cite[Lemma~2.16]{Tolim}, 
there exists a torsion pair $(\aA_{1}^p, \aA_{1/2}^p)$
on $\aA^p$, defined by 
\begin{align*}
\aA_{1}^p &\cneq \langle F[1], \oO_x : 
F \mbox{ is pure two dimensional, }
x\in \overline{X} \rangle_{\ex}, \\
\aA_{1/2}^p &\cneq \{ E\in \aA^p : \Hom(F, E)=0 \mbox{ for any }
F\in \aA_{1}^p \}. 
\end{align*}
Note that if $F$ is a pure one dimensional sheaf
on $\overline{X}$, then 
$F \in \aA_{1/2}^p$. 

For $E, F \in \aA_{1/2}^p$, a morphism 
$u \colon E \to F$ in $\aA^p$ is 
called a \textit{strict monomorphism} if 
$u$ is injective in $\aA^p$ and
$\Cok(u) \in \aA_{1/2}^p$. Similarly 
$u$ is called a \textit{strict epimorphism}
if $u$ is surjective in $\aA^p$ 
and $\Ker(u) \in \aA_{1/2}^p$. 
By~\cite[Proposition~3.13]{Tolim2}, 
an object $E\in \aA_{1/2}^{p}$ 
with $\rank(E)=-1$ is 
$\mu_{i\omega}$-\textit{limit semistable}
in the sense of~\cite[Section~3]{Tolim2} 
iff the following conditions hold:
\begin{itemize}
\item For any pure one dimensional sheaf $F \neq 0$
which admits a strict monomorphism $F \hookrightarrow E$
in $\aA_{1/2}^p$,  
we have $\ch_3(F) \le 0$. 
\item For any pure one dimensional sheaf $G \neq 0$
which admits a strict epimorphism $E \twoheadrightarrow G$
in $\aA_{1/2}^p$, 
we have $\ch_3(G) \ge 0$.
\end{itemize}
\begin{ssssstep}
Comparison of $\aA_{1/2}^p$ and $\aA(0)$.  
\end{ssssstep}
Let $\aA(0)$ be the category 
defined in Definition~\ref{defi:A(r)}. 
The categories $\aA_{1/2}^p$ and $\aA(0)$ are related
as follows: 
let $\aA(0)^{\dag}$ be the following
category, 
\begin{align*}
\aA(0)^{\dag} \cneq \{
E \in \aA(0) : \Hom(\oO_x[-1], E)=0, x \in \overline{X}\}. 
\end{align*}
Then it is easy to check that 
\begin{align*}
\aA(0)^{\dag} \subset \aA_{1/2}^p[-1]. 
\end{align*}
Also by replacing $\aA_{1/2}^p$, $\aA^p$
by $\aA(0)^{\dag}$, $\aA(0)$
respectively, we have the notions of 
strict monomorphisms and strict epimorphisms
in $\aA(0)^{\dag}$. 
Then the same proof of Lemma~\ref{lem:cri} shows that 
an object $E \in \aA(0)$ with $\rank(E)=1$
is $\mu_{i\omega}$-limit semistable 
in the sense of Definition~\ref{def:lim}
iff the following conditions hold: 
\begin{itemize}
\item For any pure one dimensional sheaf 
$0\neq F \in \Coh_{\pi}^{\le 1}(\overline{X})$
which admits a strict monomorphism $F[-1] \hookrightarrow E$
in $\aA(0)^{\dag}$, 
we have $\ch_3(F) \le 0$. 
\item For any pure one dimensional sheaf 
$0\neq G \in \Coh_{\pi}^{\le 1}(\overline{X})$
which admits a strict epimorphism $E \twoheadrightarrow G[-1]$
in $\aA(0)^{\dag}$, 
we have $\ch_3(G) \ge 0$.
\end{itemize}
\begin{ssssstep}
Proof of Lemma~\ref{rmk:lim2}. 
\end{ssssstep}
Now let us take $E \in \aA_{1/2}^p[-1]$
satisfying 
the condition (\ref{cond:chern})
for $R=1$. 
By the above arguments, it is 
enough to show the following:
\begin{itemize}
\item We have $E \in \aA(0)^{\dag}$. 
\item For any strict monomorphism $F \hookrightarrow E[1]$
in $\aA_{1/2}^p$ 
with $F$ pure one dimensional sheaf, 
we have $F \in \Coh_{\pi}^{\le 1}(\overline{X})$
and $F[-1] \to E$ is a strict monomorphism in $\aA(0)^{\dag}$. 
\item For any strict epimorphism $E[1] \twoheadrightarrow G$
in $\aA_{1/2}^p$ 
with $G$ pure one dimensional sheaf, 
we have $G \in \Coh_{\pi}^{\le 1}(\overline{X})$
and $E \to G[-1]$ is a strict monomorphism in $\aA(0)^{\dag}$.
\end{itemize}
First we prove $E \in \aA(0)^{\dag}$. By~\cite[Lemma~3.2]{Tolim}, 
we have $\hH^0(E) =I_C$
for a curve $C \subset \overline{X}$
 and $\hH^1(E)$ is a one dimensional 
sheaf. 
Hence 
\begin{align*}
\beta =[C] + [\hH^1(E)]. 
\end{align*}
Since $(1, 0, -\beta, -n) \in \Gamma$, 
the curve $C$ is supported on fibers of $\pi$
and $\hH^1(E) \in \Coh_{\pi}^{\le 1}(\overline{X})$. 
This implies that $E \in \aA(0) \cap (\aA_{1/2}^{p}[-1])
=\aA(0)^{\dag}$. 

Next we prove the second condition. 
The proof of the third one is similar and we omit it. 
Let $F \hookrightarrow E[1]$
be a strict monomorphism in $\aA_{1/2}^p$
for a pure one dimensional sheaf $F$, 
and set $G \cneq E[1]/F \in \aA_{1/2}^p$. 
We have the exact sequence of sheaves, 
\begin{align*}
0 \to \hH^0(E) \stackrel{i}{\to}
 \hH^{-1}(G) \to F \to \hH^1(E) \stackrel{j}\to \hH^0(G) \to 0. 
\end{align*}
By~\cite[Lemma~3.2]{Tolim},
$\hH^0(E)$ and $\hH^{-1}(G)$ are written 
as $I_C$, $I_{C'}$ for curves $C, C'$
in $\overline{X}$
respectively. Since $E \in \aA(0)^{\dag}$, 
$C$ is supported on fibers of $\pi$,
hence
$C'$ 
and $\Cok(i)$ are supported on fibers of $\pi$. 
 Also since $\hH^1(E) \in \Coh_{\pi}^{\le 1}(\overline{X})$, 
we have $\Ker(j), \hH^0(G) \in \Coh_{\pi}^{\le 1}(\overline{X})$
by the above sequence. Therefore we have 
$F \in \Coh_{\le 1}(\overline{X})$ and 
$G[-1] \in \aA(0)$.
Since $G[-1] \in \aA_{1/2}^p[-1]$, it follows that 
$G[-1] \in \aA(0)^{\dag}$, hence 
$F[-1] \to E$ is a strict monomorphism 
in $\aA(0)^{\dag}$. 
\end{proof}

\subsection{Proof of Proposition~\ref{prop:m1}}\label{subsec:finset}
In this subsection, we prove Proposition~\ref{prop:m1}, 
which is restated as follows: 
\begin{prop}\label{prop:m12}
In the same situation of Proposition~\ref{prop:wall2},
we have   
\begin{align*}
M_{t\omega}(R, r, \beta, n)=M_{\rm{lim}}(R, r, \beta, n),  
\end{align*}
for any $t\in (t_{k-1}, \infty)$
and $R \le 1$ satisfying 
$\arg Z_{t\omega}(R, r, \beta, n)=\pi/2$. 
\end{prop}
We take $(R, r, \beta, n) \in \Gamma$
as in the statement. 
For simplicity, we show the case of $R=1$. 
The proof for $R=0$ is similar and easier, so 
we omit it. 
We divide the proof into 2 steps. 
\begin{sstep}
For $t>t_k$, we have 
\begin{align*}
M_{t\omega}(1, r, \beta, n) \subset M_{\rm{lim}}(1, r, \beta, n).
\end{align*}
\end{sstep}
\begin{proof}
Take an object $E\in M_{t\omega}(1, r, \beta, n)$
for $t>t_k$ 
and a filtration in $\aA_{\omega}$, 
\begin{align*}
E_1 \subset E_2 \subset E_3=E, 
\end{align*}
given by Lemma~\ref{lem:rank10}. 
Suppose that $E_1 \neq 0$.  
Then the $Z_{t\omega}$-semistability 
of $E$ implies that 
\begin{align}\label{ineq:lim}
\arg Z_{t\omega}(E_1) &\le \frac{\pi}{2}. 
\end{align}
We write $\cl_0(E_1)=(r_1, \beta_1, n_1) \in \Gamma_0$. 
Then the inequality (\ref{ineq:lim}) 
is equivalent to 
\begin{align}\notag
\Ree Z_{t\omega}(E_1)=n_1 -\frac{1}{2}t^2 \omega^2 r_1 \ge 0. 
\end{align}
The above inequality should be satisfied for all $t>t_k$. 
However since $r_1>0$, the above inequality is 
not satisfied for $t\gg 0$. 
This is a contradiction, hence $E_1=0$. 
A similar argument also shows that
$E/E_2=0$, hence
$E\in \aA(r)$ follows.

In order to show that $E$ is $\mu_{i\omega}$-limit
semistable, we take an exact sequence in $\aA(r)$, 
\begin{align*}
0 \to F \to E \to G \to 0. 
\end{align*}
If $F \in \Coh_{\pi}^{\le 1}(\overline{X})[-1]$, then the 
$Z_{t\omega}$-stability yields, 
\begin{align*}
\Ree Z_{t\omega}(F)=\ch_3(F) \ge 0. 
\end{align*}
Similarly if $G \in \Coh_{\pi}^{\le 1}(\overline{X})[-1]$, 
we obtain $\ch_3(G) \le 0$. 
Therefore $E$ is $\mu_{i\omega}$-limit semistable, 
i.e. $E \in M_{\rm{lim}}(1, r, \beta, n)$. 
\end{proof}
\begin{sstep}
For $t>t_k$, we have 
\begin{align*}
M_{\rm{lim}}(1, r, \beta, n) \subset M_{t\omega}(1, r, \beta, n).
\end{align*}
\end{sstep}
\begin{proof}
Take an object $E\in M_{\rm{lim}}(1, r, \beta, n)$, and 
an exact sequence in $\aA_{\omega}$, 
\begin{align}\label{AEB}
0 \to A \to E \to B \to 0. 
\end{align}
Since $\rank(E)=1$, 
one of $A$ or $B$ is an object in $\bB_{\omega}$. 
Suppose that $A \in \bB_{\omega}$, 
and it destabilizes $E$ with respect to $Z_{t\omega}$-stability, 
\begin{align}\label{want:A}
\arg Z_{t\omega}(A) > \frac{\pi}{2}.
\end{align} 
We first show that $A \in \tT_{\omega}[-1]$,
i.e. $\hH^0(A)=0$. 
Suppose by contradiction that 
$\hH^0(A) \neq 0$. 
Since $\hH^0(A)$ is a torsion sheaf
on $\overline{X}$
and $E\in \aA(r)$, 
the definition of $\aA(r)$
yields that  
\begin{align*}
\Hom(\hH^0(A), E)=0.
\end{align*}
However this is a contradiction 
since $\hH^0(A)$ is a subobject of $E$ in $\aA_{\omega}$. 
Hence $\hH^0(A)=0$ and $A \in \tT_{\omega}[-1]$ follows. 

Since $A \in \tT_{\omega}[-1]$, 
we can take an exact sequence in $\aA_{\omega}$, 
\begin{align*}
0 \to T''[-1] \to A \to T'[-1] \to 0, 
\end{align*}
with $T'' \in \Coh_{\pi}^{\le 1}(\overline{X})$
and $T' \in \tT_{\omega}^{\rm{pure}}$. 
The composition of injections in $\aA_{\omega}$, 
\begin{align*}
T''[-1] \hookrightarrow A \hookrightarrow E,
\end{align*}
is also an injection in $\aA(r)$
by Lemma~\ref{lem:Ar1} below.
Hence
the $\mu_{i\omega}$-limit semistability of $E$
yields 
$\ch_3(T''[-1]) \ge 0$, or equivalently
\begin{align}\label{Ztlepi}
\arg Z_{t\omega}(T''[-1]) \le \frac{\pi}{2},
\end{align}
for all $t\in \mathbb{R}_{>0}$. 
By (\ref{want:A}) and (\ref{Ztlepi}), 
we have $T' \neq 0$ and 
\begin{align*}
\arg Z_{t\omega}(T'[-1]) &\ge \arg Z_{t\omega}(A) \\
 & > \frac{\pi}{2}. 
\end{align*}
If we write $\cl_0(T')=(r', \beta', n')$, 
then $r'>0$ and 
the above inequality yields, 
\begin{align}\label{vio}
-n'+\frac{1}{2}t^2 \omega^2 r' <0. 
\end{align}
By applying Lemma~\ref{finset}, we see that
$(r', n')$ have only a finite number of 
possibilities for fixed $\beta$ and $\omega$. 
Therefore 
there is a
 constant
$M(\beta, \omega) >0$
which depends only on $\beta$ and $\omega$
 such that
if 
$t>M(\beta, \omega)$, 
then the inequality (\ref{vio})
is violated.
This means that for such $t$, 
the inequality (\ref{want:A}) is not satisfied, 
i.e. $\arg Z_{t\omega}(A) \le \pi/2$ follows. 

In the case of $B\in \bB_{\omega}$, we can 
similarly prove
the inequality 
$\arg Z_{t\omega}(B) \ge \pi/2$ 
for $t>M(\beta, \omega)$, 
 by replacing $M(\beta, \omega)$ if necessary.
Therefore 
$E$ is $Z_{t\omega}$-semistable for 
$t>M(\beta, \omega)$, hence for $t>t_k$. 
\end{proof}

We have used the following lemma. 
\begin{lem}\label{lem:Ar1}
For $E\in \aA(r)$
with $\rank(E)=1$
and $F \in \Coh_{\pi}^{\le1}(\overline{X})$,
 take an exact 
sequence in $\aA_{\omega}$, 
\begin{align}\label{Ar1seq}
0 \to F[-1] \to E \to G \to 0. 
\end{align}
Then we have $G\in \aA(r)$, 
hence the sequence (\ref{Ar1seq}) is an exact sequence in 
$\aA(r)$. 
\end{lem}
\begin{proof}
Taking the cohomology of (\ref{Ar1seq}), 
we have the exact sequence of sheaves, 
\begin{align}\label{longse}
0 \to \hH^0(E) \to \hH^0(G) \to F \to \hH^1(E) \to \hH^1(G) \to 0. 
\end{align}
Since $\hH^1(E) \in \Coh_{\pi}^{\le 1}(\overline{X})$, 
we have
$\hH^1(G) \in \Coh_{\pi}^{\le 1}(\overline{X})$.
In particular we have 
\begin{align}\label{HG1}
\hH^1(G)[-1] \in \aA(r).
\end{align}
Suppose that
the maximal torsion subsheaf 
 $\hH^0(G)_{\rm{tor}} \subset \hH^0(G)$ is non-zero. 
Then $\hH^0(G)_{\rm{tor}}$ is a pure two dimensional
sheaf, 
since $G\in \aA_{\omega} \subset \aA_{\omega}'$
and $\aA_{\omega}'$ is a tilting by 
$(\tT_{\omega}, \fF_{\omega}')$. 
(cf.~Definition~\ref{defi:Ao'}.)
Also since $F\in \Coh_{\pi}^{\le 1}(\overline{X})$,
the sequence (\ref{longse}) implies that
$\hH^0(E)$ is isomorphic to 
$\hH^0(G)$ in codimension one. 
In particular, the maximal 
torsion subsheaf $\hH^0(E)_{\rm{tor}} \subset \hH^0(E)$
is also a two dimensional sheaf. 
However this contradicts to 
$E \in \aA(r)$ and the definition of $\aA(r)$. 
Therefore $\hH^0(G)$ is a torsion free 
sheaf of rank one, and it can be written as 
\begin{align*}
\hH^0(G) \cong L \otimes I_Z, 
\end{align*}
for some $L\in \Pic(\overline{X})$
and $Z\subset \overline{X}$ with $\dim Z \le 1$. 
By the assumptions $E\in \aA(r)$
and 
$F\in\Coh_{\pi}^{\le 1}(\overline{X})$, 
it is easy to see 
from the sequence (\ref{longse})
that $L \in \pi^{\ast} \Pic(\mathbb{P}^1)$
and $Z$ is supported on the fibers of $\pi$.
Hence it follows that 
\begin{align}\label{HG2}
\hH^0(G) \in \aA(r).
\end{align}
By (\ref{HG1}) and (\ref{HG2}), 
we have $G\in\aA(r)$.
\end{proof}

\section{Results on the category $\aA_{\omega}(1/2)$}\label{sec:A12}
In this section, we give proofs of 
some results on the category $\aA_{\omega}(1/2)$
introduced in Subsection~\ref{subsec:abelian12}. 
In particular, we prove
Lemma~\ref{lem:lem12} in Subsection~\ref{subsec:lem12}, 
 Lemma~\ref{lem:t0c} in Subsection~\ref{subsec:t0c}
and Proposition~\ref{prop:propA12} in Subsection~\ref{subsec:propA12}. 
First we note that, by the existence of Harder-Narasimhan filtrations
with respect to $Z_{0\omega}$-stability, 
(cf.~Definition~\ref{defi:polynomial},)
there is a filtration for any $E\in \aA_{\omega}$,
\begin{align}\label{A12}
0=E_0 \subset E_1 \subset E_2 \subset E_3=E, 
\end{align}
such that 
\begin{align}\label{A123}
E_1 \in \aA_{\omega}(1), \ E_2/E_1 \in \aA_{\omega}(1/2), \ 
E/E_2 \in \aA_{\omega}(0). 
\end{align}
We also note that 
\begin{align}\label{vanish:12}
\Hom(E_1, E_2)=0, \ E_i \in \aA_{\omega}(\phi_i), 
\end{align}
if $\phi_1>\phi_2$.
We note that, by setting $\aA_{\omega}(\phi+1)=\aA_{\omega}(\phi)[1]$, 
the family of subcategories $\aA_{\omega}(\phi) \subset \dD$
for $\phi \in \mathbb{R}$
determines a slicing on $\dD$. 
(cf.~\cite[Definition~3.3]{Brs1}.)
\subsection{Proof of Lemma~\ref{lem:lem12}}\label{subsec:lem12}
In this subsection, we prove Lemma~\ref{lem:lem12}, which 
is restated as follows: 
\begin{lem}\label{lem:lem123}
(i) An object $E \in \aA_{\omega}$
is $Z_{0\omega}$-(semi)stable if and only 
if $E$ is $Z_{t\omega}$-(semi)stable
for $0<t \ll 1$. 
 
(ii)
Any object $E \in \aA_{\omega}(1/2)$ satisfies 
$\ch_3(E)=0$. 

(iii) The category $\aA_{\omega}(1/2)$ is an 
abelian subcategory of $\aA_{\omega}$.  
\end{lem}
\begin{proof}
For $A \in \bB_{\omega}$ 
with $\cl_0(A)=(r, \beta, n)$,
the definition of $Z_{t\omega, 0}$
yields the following: 
\begin{itemize}
\item We have 
$\arg Z_{t\omega, 0}(A) \to \pi/2$
for $t \to 0$ iff $n=0$ and $\beta \cdot \omega \neq 0$. 
\item We have 
$\arg Z_{t\omega, 0}(A) \to 0 $
for $t \to 0$ iff $n>0$ and $\beta \cdot \omega \neq 0$.
\item We have 
$\arg Z_{t\omega, 0}(A) \to \pi $
for $t \to 0$ iff $n<0$ and $\beta \cdot \omega \neq 0$, 
or $\beta \cdot \omega=0$. 
\end{itemize}
Noting above, 
the results of 
(i) and (ii) easily follow from
the definitions of $Z_{t\omega}$,  
$\aA_{\omega}(1/2)$, and the results, proofs of 
Proposition~\ref{prop:m23}, Proposition~\ref{prop:rank0}. 
In order to check (iii), take 
$E, E'\in \aA_{\omega}(1/2)$ and 
a non-zero morphism in $\aA_{\omega}$, 
\begin{align*}
u \colon E \to E'.
\end{align*}
We show that $\Ker(u)$, $\Imm(u)$ and $\Cok(u)$
in $\aA_{\omega}$ are contained in $\aA_{\omega}(1/2)$. 
Since $u$ is decomposed
as $E\twoheadrightarrow \Imm(u) \hookrightarrow E'$
in $\aA_{\omega}$, 
we have 
\begin{align*}
\Hom(F, \Imm(u))=\Hom(\Imm(u), F')=0,
\end{align*}
for any $F \in \aA_{\omega}(1)$ and $F' \in \aA_{\omega}(0)$
by (\ref{vanish:12}). This implies 
$\Imm (u) \in \aA_{\omega}(1/2)$
by the existence of a filtration (\ref{A12})
satisfying (\ref{A123}).  
Therefore we may assume that $u$ is injective or surjective 
in $\aA_{\omega}$. 
Suppose that $u$ is surjective. 
Then we have $\Hom(F, \Ker(u))=0$
for any $F \in \aA_{\omega}(1)$,  
hence we have 
\begin{align*}
\Ker(u) \in \langle \aA_{\omega}(1/2), \aA_{\omega}(0) \rangle_{\ex}. 
\end{align*}
There is an exact sequence in $\aA_{\omega}$, 
\begin{align*}
0 \to A_1 \to \Ker(u) \to A_2 \to 0,
\end{align*}
with $A_1 \in \aA_{\omega}(1/2)$ and $A_2 \in \aA_{\omega}(0)$. 
Since we have 
\begin{align*}
\ch_3(\Ker(u)) &=\ch_3(E)-\ch_3(E') \\
&=0,
\end{align*}
and $\ch_3(A_1)=0$ by (ii), 
we have $\ch_3(A_2)=0$ if $A_2 \neq 0$. 
However this contradicts to $\arg Z_{t\omega}(A_2) \to 0$
for $t \to 0$. 
Hence $A_2=0$ and $\Ker(u) \in \aA_{\omega}(1/2)$
follows. A similar argument shows 
that $\Cok(u) \in \aA_{\omega}$ when 
$u$ is an injection in $\aA_{\omega}$. 
\end{proof}

\subsection{Proof of Lemma~\ref{lem:t0c}}\label{subsec:t0c}
In this subsection, we prove Lemma~\ref{lem:t0c}, 
which is restated as follows: 
\begin{lem}\label{lem:t0c2}
An object $E \in \aA_{\omega}$ is 
$Z_{0\omega}$-semistable
satisfying 
\begin{align*}
\lim_{t\to 0}\arg Z_{t\omega}(E) = \pi/2,
\end{align*}
if and only if 
$E \in \aA_{\omega}(1/2)$ and $E$ is 
$\widehat{Z}_{\omega, 1/2}$-semistable.
\end{lem}
\begin{proof}
First assume that $E \in \aA_{\omega}$ is $Z_{0\omega}$-semistable
with $\arg Z_{t\omega}(E) \to \pi/2$
for $t \to 0$. 
By the definition of $\aA_{\omega}(1/2)$, we have 
$E \in \aA_{\omega}(1/2)$. 
Take an exact sequence in $\aA_{\omega}(1/2)$, 
\begin{align}\label{seq:FEG0}
0 \to F \to E \to G \to 0. 
\end{align}
The above sequence is also an exact sequence in $\aA_{\omega}$. 
By the $Z_{0\omega}$-(semi)stability of $E$, we have 
\begin{align}\label{ineq:o12}
\arg Z_{t\omega}(F) \le \arg Z_{t\omega}(G), 
\end{align}
for $0<t \ll 1$. 
Since $\ch_3(F)=\ch_3(G)=0$ by Lemma~\ref{lem:lem123}, the 
above inequality implies 
\begin{align}\notag
\arg \widehat{Z}_{\omega, 1/2}(F) \le \arg \widehat{Z}_{\omega, 1/2}(G). 
\end{align}
Therefore $E$ is $\widehat{Z}_{\omega, 1/2}$-semistable in $\aA_{\omega}(1/2)$. 

Conversely, suppose that $E \in \aA_{\omega}(1/2)$
is $\widehat{Z}_{\omega, 1/2}$-semistable, and take 
an exact sequence in $\aA_{\omega}$, 
\begin{align}\label{seq:FEG00}
0 \to F' \to E \to G' \to 0. 
\end{align}
We would like to see that 
\begin{align}\label{ineq:o123}
\arg Z_{t\omega}(F') \le \arg Z_{t\omega}(G'), 
\end{align}
for $0<t \ll 1$. 
If both of $\rank(F')$ and $\rank(G')$ are 
positive, then 
we have 
\begin{align*}
\arg Z_{t\omega}(F')=\arg Z_{t\omega}(G')=\frac{\pi}{2}, 
\end{align*}
for $0<t \ll 1$. 
Therefore we may assume that
$\rank(F')=0$ or $\rank(G')=0$. 
We discuss the case of $\rank(F')=0$. 
The other case is similarly discussed. 
 As in (\ref{A12}), we take a filtration in $\aA_{\omega}$,
\begin{align*}
0=F_0' \subset F_1' \subset F_2' \subset F_3'=F', 
\end{align*}
such that 
$F_1' \in \aA_{\omega}(1)$, $F_2'/F_1' \in \aA_{\omega}(1/2)$
and $F'/F_2' \in \aA_{\omega}(0)$. 
Since $E \in \aA_{\omega}(1/2)$, we have
$\Hom(F_1', E)=0$, hence  
$F_1'=0$. Suppose that 
$F'/F_2' \neq 0$. Then 
we have $\rank(F')=0$, $\ch_3(F')>0$, 
hence 
$\arg Z_{t\omega}(F') \to 0$ for $t \to 0$. 
Since $\arg Z_{t\omega}(G')=\pi/2$, 
the inequality (\ref{ineq:o123}) 
 is satisfied for $0< t\ll 1$. 
Suppose that $F'/F_2' =0$. 
Then the sequence (\ref{seq:FEG0})
is an exact sequence in $\aA_{\omega}(1/2)$ by 
Lemma~\ref{lem:lem123}.
Hence by the $\widehat{Z}_{\omega, 1/2}$-semistability
of $E$, 
we have 
\begin{align}\notag
\arg \widehat{Z}_{\omega, 1/2}(F') \le \arg \widehat{Z}_{\omega, 1/2}(G'),
\end{align}
which implies the inequality
 (\ref{ineq:o123}) for $0<t\ll 1$.  
\end{proof}

\subsection{Proof of Proposition~\ref{prop:propA12}}\label{subsec:propA12}
In this subsection, we prove Proposition~\ref{prop:propA12}, which is 
restated as follows: 
we have the following proposition. 
\begin{prop}\label{prop:propA123}
For fixed $\beta \in \mathrm{NS}(S)$ and 
an ample divisor $\omega$ on $S$,
there is a finite sequence, 
\begin{align*}
0=\theta_k<\theta_{k-1}< \cdots <\theta_{1}<\theta_0=1/2,
\end{align*} 
such that the following holds. 

(i) The set of objects 
\begin{align*}
\bigcup_{(R, r), R \ge 1}\widehat{M}_{\omega, \theta}(R, r, \beta),
\end{align*}
is constant for $\theta \in (\theta_{i-1}, \theta_i)$. 

(ii) For $0<t \ll 1$ and any 
$(R, r, \beta) \in \widehat{\Gamma}$, we have 
\begin{align*}
\widehat{M}_{\omega, 1/2}(R, r, \beta)=M_{t\omega}(R, r, \beta, 0). 
\end{align*}

(iii) For $\theta \in (0, \theta_{k-1})$, we have 
\begin{align*}
\widehat{M}_{\omega, \theta}(1, r, \beta)=\left\{ \begin{array}{cc}
\{\pi^{\ast}\oO_{\mathbb{P}^1}(r) \}, & \mbox{ if } \beta=0, \\
\emptyset, & \mbox{ if } \beta \neq 0. 
\end{array} \right. 
\end{align*}
\end{prop}
\begin{proof}
(i) Take $E\in \bB_{\omega}(1/2)$
with $\widehat{\cl}_0(E)=(r', \beta')$. 
Suppose that $0\le -\beta' \cdot \omega \le -\beta \cdot \omega$, 
and let $F_1, F_2, \cdots, F_N \in \bB_{\omega}$ be the
Harder-Narasimhan factors of $E$ with respect
to $Z_{0\omega}$-stability. 
Since $E\in \bB_{\omega}(1/2)$, 
we have $F_i \in \bB_{\omega}(1/2)$
for all $i$, and we write $\widehat{\cl}_0(F_i)=(r_i, \beta_i)$. 
Because $-\beta_i \cdot \omega \le -\beta \cdot \omega$,
Lemma~\ref{lem:lem123} (i) and  
Lemma~\ref{bound:rank0} imply that there is only a 
finite number of possibilities for $r_i$
w.r.t. fixed $\beta$ and $\omega$. 
Also noting that $N\le -\beta \cdot \omega$, the value
\begin{align*}
r'=\sum_{i=1}^{N} r_i, 
\end{align*}
is also bounded. Therefore 
for fixed $\beta$ and $\omega$, the set of 
$\theta \in (0, 1/2]$ satisfying 
\begin{align*}
\widehat{Z}_{\omega, \theta}(E) 
&=-r' -(\omega \cdot \beta') \sqrt{-1} \\
&\in \mathbb{R}_{>0} e^{i\pi \theta},
\end{align*}
for some $E \in \bB_{\omega}(1/2)$
with $\widehat{\cl}_0(E)=(r', \beta')$, 
$-\beta' \cdot \omega \le -\beta \cdot \omega$
is a finite set. 
If we denote this finite set by 
$0=\theta_k <\theta_{k-1} < \cdots <\theta_{1}<\theta_0=1/2$, 
then $\theta_{\bullet}$ satisfies the desired 
condition. 

(ii) The result of (ii) is a consequence of 
Lemma~\ref{lem:lem123} (i) and Lemma~\ref{lem:t0c2}. 

(iii) For
an object $E \in \widehat{M}_{\omega, \theta}(1, r, \beta)$,
take an exact sequence in $\aA_{\omega}$, 
\begin{align}\label{seq:AEB:A12}
0 \to A \to E \to B \to 0, 
\end{align}
with $A \in \bB_{\omega}$ and $B\in  \pi^{\ast}\Pic(\mathbb{P}^1)$, as 
in Lemma~\ref{lem:rank1}.
By Lemma~\ref{lem:Ostable}, we have $B \in \aA_{\omega}(1/2)$, 
hence $A \in \bB_{\omega}(1/2)$ by Lemma~\ref{lem:lem123}, 
i.e. (\ref{seq:AEB:A12}) is an exact sequence in $\aA_{\omega}(1/2)$. 
Suppose that $A$ is non-zero. 
Then we can write 
$\widehat{\cl}_0(A)=(r', \beta)$ for some $r' \in \mathbb{Z}$, 
and $\beta$ should satisfy $\beta \cdot \omega \neq 0$. 
By the $\widehat{Z}_{\omega, \theta}$-semistability of $E$, we have 
\begin{align*}
\arg \widehat{Z}_{\omega, \theta}(A) \le \pi \theta.  
\end{align*}
The above inequality is equivalent to 
\begin{align*}
\frac{r'}{\beta \cdot \omega} \ge \frac{1}{\tan \pi\theta}.
\end{align*}
Since the RHS goes to $\infty$ for $\theta \to 0$, the 
object $E$ is $\widehat{Z}_{\omega, \theta}$-semistable
for $0<\theta \ll 1$ only if 
$A=0$, i.e. $E\cong \pi^{\ast}\oO_{\mathbb{P}^1}(r)$. 
Therefore we have $\widehat{M}_{\omega, \theta}(1, r, \beta) =\emptyset$
for $\beta \neq 0$ and $0<\theta \ll 1$, 
and an only possible object in 
$\widehat{M}_{\omega, \theta}(1, r, 0)$
for $0<\theta \ll 1$ 
is $\pi^{\ast}\oO_{\mathbb{P}^1}(r)$. 
On the other hand, 
since $\aA_{\omega}(1/2) \subset \cC_{\omega}^{\perp}$, 
(cf.~(\ref{ROCom}),)
Lemma~\ref{lem:Ostable} immediately 
implies that $\pi^{\ast}\oO_{\mathbb{P}^1}(r)$ 
is $\widehat{Z}_{\omega, \theta}$-stable 
for any $\theta \in (0, 1)$.
Therefore we obtain the result. 
\end{proof}

Todai Institute for Advanced Studies (TODIAS), 

Kavli Institute for the Physics and 
Mathematics of the Universe, 

University of Tokyo, 
5-1-5 Kashiwanoha, Kashiwa, 277-8583, Japan.

\textit{E-mail address}: yukinobu.toda@ipmu.jp

\end{document}